%% file: main.tex
\documentclass[msom,nonblindrev]{informs3} 

\OneAndAHalfSpacedXI 

\usepackage{endnotes}

 \usepackage{algpseudocode}
\usepackage{tikz}

\usepackage{natbib}
 \bibpunct[, ]{(}{)}{,}{a}{}{,}%
 \def\bibfont{\small}%
\input{preamble}

\usepackage[ruled]{algorithm2e} 

\SetAlFnt{\small}
\SetAlCapFnt{\small}
\SetAlCapNameFnt{\small}
\SetAlCapHSkip{0pt}
\IncMargin{-\parindent}

\TheoremsNumberedThrough     
\ECRepeatTheorems

\EquationsNumberedThrough    

\begin{document}


\RUNAUTHOR{Hssaine, Topaloglu and van Ryzin}

\RUNTITLE{Online Allocation of Throughput-Constrained Resources}

\TITLE{Online Allocation of Throughput-Constrained Resources Using Proxy Assignments}

\ARTICLEAUTHORS{%

\AUTHOR{Chamsi Hssaine}

\AFF{Department of Data Sciences and Operations,
University of Southern California, Marshall School of Business\\ Los Angeles, CA 90089, \EMAIL{hssaine@usc.edu}}

\AUTHOR{Huseyin Topaloglu}
\AFF{School of Operations Research and Information Engineering, Cornell Tech \\
New York, NY 10044, \EMAIL{ht88@cornell.edu}}

\AUTHOR{Garrett van Ryzin}
\AFF{Amazon\\ New York, NY 10001 \EMAIL{ryzing@amazon.com}}
}

\ABSTRACT{%
\input{parts/abstract-new}
}%





\maketitle

	\input{parts/intro-new}
	\input{parts/preliminaries}
        \input{parts/main-result}
 	\input{parts/experiments}

	\input{parts/conclusion}

    \medskip 

    \noindent\textbf{Acknowledgments:} {We thank Daan Rutten for help with providing data and context information for the computational experiments. We moreover thank Santiago Balseiro, Haihao Lu, and Jeannette Song for insightful conversations on our work. Finally, we acknowledge the participants of the Durham Early Career Scholars Workshop and the Banff Workshop on Dynamic Matching and Allocation for useful feedback on a preliminary version of this work.}

\newpage

\bibliographystyle{informs2014} 
\bibliography{references} 


%
%
%

\newpage

\crefalias{section}{appendix}
\begin{APPENDICES}
\OneAndAHalfSpacedXI
    \input{parts/single_epoch_alg}
    \input{parts/apx_nonstationary}
    \input{parts/apx_exper}
\end{APPENDICES}






\end{document}

%% file: preamble.tex
\usepackage{amsmath,amsfonts,cancel}
\usepackage{thmtools}
\usepackage{adjustbox}
\usepackage[shortlabels]{enumitem}
\usepackage{graphicx}
\definecolor{edits}{rgb}{0,0,0}
\usepackage{mathtools}
\usepackage{dsfont}
\usepackage{tcolorbox}
\usepackage{booktabs}
\usepackage{bbm}
\usepackage{xspace}
\usepackage{subcaption}
\usepackage{tensor}
\usepackage{bm}
\usepackage{tikz}
\usetikzlibrary{patterns}
\usetikzlibrary{arrows.meta, decorations.pathreplacing, positioning}
\usetikzlibrary{calc}
\usetikzlibrary{matrix}
\usepackage{natbib}

\usepackage{placeins}
\usepackage[colorlinks=true,breaklinks=true,bookmarks=true,urlcolor=magenta,citecolor=magenta,linkcolor=magenta,bookmarksopen=false,draft=false]{hyperref}
\usepackage[capitalize]{cleveref}
\usepackage{pseudocode}

\usepackage{color-edits}
\addauthor{ch}{blue}

\newcommand{\proxyx}[2]{\tilde{x}_{#1}^{#2}}
\newcommand{\dual}[2]{\mu_{#1}^{#2}}
\newcommand{\auxa}[2]{a_{#1}^{#2}}
\newcommand{\staget}[1]{\mathcal{S}_t}

\newcommand{\nonstationaryalg}{\textsc{NS-Proxy-Dual-GD}}

\crefname{assumption}{assumption}{assumptions}
\crefname{algocf}{algorithm}{algorithm}

\mathchardef\mhyphen="2D 
\ifdefined \argmin
\else \DeclareMathOperator*{\argmax}{arg\,min}
\fi
\ifdefined \argmax
\else \DeclareMathOperator*{\argmax}{arg\,max}
\fi

\let\originalleft\left
\let\originalright\right
\renewcommand{\left}{\mathopen{}\mathclose\bgroup\originalleft}
\renewcommand{\right}{\aftergroup\egroup\originalright}


%% file: parts/abstract-new.tex
{\textbf{Problem definition:} We study a {variation of the canonical online resource allocation problem in which resources are {\it throughput}, rather than {\it budget}, constrained}. As in {the classical setting}, the decision-maker must assign sequentially arriving jobs to one of multiple available resources. However, in addition to the assignment costs incurred from these decisions, the decision-maker is also penalized for deviating from exogenous, {time-varying target assignment rates for each resource}{, which represent the resources' respective throughput capacities throughout the horizon}. The goal is to minimize the total expected assignment and deviation penalty costs incurred throughout the horizon when the distribution of assignment costs is unknown. \textbf{Methodology/results:} We first show that naive extensions of state-of-the-art algorithms for classical budget-constrained resource allocation problems can fail dramatically when applied to {throughput-constrained} resource allocation. We then propose a novel ``proxy assignment" primal-dual algorithm that uses current arrivals to simulate the effect of future arrivals. We prove that our algorithm achieves the optimal $O(\sqrt{T})$ regret bound when the assignment costs of the arriving jobs are drawn i.i.d. from a fixed distribution. We demonstrate the practical performance of our approach by conducting numerical experiments on synthetic datasets, as well as real-world datasets from retail fulfillment operations. \textbf{Managerial implications:} Our framework applies to many practical settings where resource capacity is {throughput}-based and potentially time-varying, such as assigning jobs to a pool of workers where staffing levels change over time, or assigning shipments to trucks whose departure times are scheduled throughout the day. In contrast to traditional online resource allocation, the financial costs of assignments must be balanced against the cost of deviating from target consumption {rates}. Our work provides decision-makers with the tools to efficiently manage this trade-off.}

%% file: parts/intro-new.tex
\section{Introduction}

Classical online resource allocation problems model settings in which a decision-maker needs to allocate a limited resource to a sequence of incoming arrivals over a finite horizon. The main challenge in such settings is uncertainty about future arrivals; should the decision-maker allocate a resource to an arrival today and collect a known {reward}, or should she reserve the resource for a future arrival with unknown but potentially high {reward}? While these classical models have found great success in applications such as airline seat allocation and allocating budgets in online advertising, they assume that the decision-maker's objective (most frequently expected reward maximization) is invariant to the {\it timing} of resource consumption. {Budget} constraints in these settings are simple aggregate consumption constraints; that is, only the {\it total} consumption matters, not the trajectory of consumption {\it through time}.
This assumption {also holds in} more recent works that have considered non-separable objectives such as fairness and load balancing, in which the decision-maker evaluates the quality of the total consumption of resources along other dimensions throughout the entire horizon \citep{agrawal2014fast,balseiro2021regularized,chen2022network}.

In this paper we focus on settings in which the trajectory of consumption is a first-order consideration for the decision-maker. This is true in many physical processing settings, where the {resource being allocated is {\it throughput}, rather than {\it budget}, constrained}.
For example, in online retail fulfillment operations, physical processes such as order picking and packing are a mix of manual and automated flow processes that have {throughput} constraints {specifying the maximum rate at which orders can be processed per hour; these constraints are} imposed by mechanical capacities and labor staffing levels. Moreover, these {throughput} capacities may vary throughout the work day, e.g., due to varying staffing levels. Systems controlling the assignment of {the flow of} work need to trade off between assigning enough work to avoid idleness, and assigning too much work, which results in backlogs and congestion. {This lies in contrast to budget-constrained settings, where there is no penalty for not assigning work to a resource for a significant part of the horizon while over-assignment to resources is not allowed.} Another example where the {trajectory} of {consumption} is critical is outbound transportation, where packages are loaded onto trucks, either for transfer to downstream facilities in the network or to customers' doorsteps. Truck departures vary throughout the day based on transportation scheduling decisions, and each departure has a strict cutoff time by which the truck needs to be filled. Assigning too few packages risks trucks departing underutilized, which wastes costly trucking capacity, while assigning too many packages risks exceeding truck capacity, causing loading dock backlogs and potential delivery promise misses.

In practice, {\it target allocations} are a common way to manage such {throughput} capacity allocation decisions.\footnote{The practices described here are representative of how Amazon manages many of its supply chain operations, and the results presented in this paper are based on the authors' joint collaboration while at Amazon.} 
These exogenous consumption trajectories, typically computed by an aggregate flow planning system,  are fed into execution systems at the beginning of each day to provide high-level guidance on the rate at which work should be assigned to each resource throughout the day. 
In the context of shipment processing and handling, for instance, targets are specified hourly, and represent the running average fraction of shipments to be processed by manual labor or machines. While actual assignment decisions need not strictly adhere to these target allocations, target adherence (also referred to as {\it target following}) is an important metric against which these decisions are benchmarked. While perfect target adherence can trivially be achieved by allocating arrivals proportionally to their targets at all times, this may come at a high assignment cost, since certain resources may be better suited than others for processing a given arrival. This creates an inherent tension between maintaining adherence to the consumption targets and assigning each arrival to its preferred resource. The goal of this work is to propose {simple and practical} solutions to this challenging task, which we call the {\it throughput-constrained online resource allocation problem}.

\subsection{Main Contributions}

We consider a model in which a decision-maker faces a sequence of arrivals seeking assignment to a set of resources over a discrete horizon of length $T$. Each arrival has an observable type defined by a vector of assignment costs associated with each resource, drawn i.i.d. from an unknown distribution. The horizon is partitioned into $K$ intervals of equal length (referred to as {\it epochs}) over which each resource has an exogenously given {target assignment rate representing its ideal throughput}.\footnote{In practice, these targets {are typically} derived by discretizing a continuous consumption trajectory from an upstream planning system.} In addition to the cost of assigning each arrival to a resource, at the end of each epoch the decision-maker incurs a cost that penalizes the deviation between the realized cumulative consumption of each resource and its respective target consumption. This cost{, a generic proxy for the costs incurred from backlogs, idling, overtime, etc.,} is specified as a known, convex function which we call the {\it deviation cost}. The goal of the decision-maker is to minimize the sum of cumulative assignment and deviation costs throughout the horizon.

When there is a single epoch, this problem is a special case of the recently studied regularized online allocation problem, for which \citet{balseiro2021regularized} designed a sublinear regret algorithm relative to the hindsight optimal solution that has access to the entire sequence of arrivals at the beginning of the horizon. For multiple epochs, however, we show that this problem is qualitatively different. Indeed, consider two natural attempts at extending optimal algorithms for the single-epoch setting: (i) a myopic algorithm which treats each individual epoch as a single-epoch problem with a shorter horizon, and (ii) an algorithm that leverages Lagrangian duality, similar to the optimal single-epoch algorithm. We show both of these naive extensions incur linear regret with respect to the hindsight optimal solution in the worst case (Propositions \ref{prop:myopic-is-bad} and \ref{prop:bad-example}). Importantly, the bad instances we construct to prove these results highlight the challenges specific to multi-epoch target following: On the one hand, good algorithms need to take into account all future epochs in their current decisions. But doing this well with an online algorithm is challenging since the decision-maker, in a sense, needs to simultaneously solve for multiple single-epoch allocation problems over horizons that have not yet been observed.

Our main technical contribution is a primal-dual algorithm (\Cref{alg:multi-stages}) that overcomes these obstacles and achieves the optimal $O(\sqrt{T})$ regret guarantee in expectation (\Cref{thm:regret}). The algorithm uses the novel notion of {\em proxy assignments}; that is, rather than just solving a single assignment cost minimization problem for the arrival at hand, our algorithm solves cost minimization problems for the current and each future epoch using this same arrival as a proxy. The current allocation is executed, but the decisions made for future epochs are only used as a proxy for the real decisions to be made in the future. In this way, we conduct parallel online simulations of what the assignment decisions will be in future epochs, given the current arrival and a dual variable that reflects the deviation between actual (or simulated) resource consumption and target resource consumption for the respective epoch. These simulated assignment decisions --- which, importantly, are never implemented --- effectively decouple decision-making across epochs, all the while keeping track of epoch-specific consumption. The epochs are then ``re-coupled'' in a deviation cost minimization step, which takes as input the dual variables for all epochs and resources. This slightly unintuitive construction avoids the pitfalls of naive extensions; it is forward-looking, but incorporates future targets in a realistic enough way to create good decisions in the current epoch. The proof of the algorithm's performance guarantee relies on defining a hindsight optimum for each epoch, and showing that our algorithm performs well in each epoch relative to its respective hindsight optimum. It then bounds the loss incurred from the assumption in the proxy assignment stage that the arrivals in future epochs are identical to those seen in the current decision-making epoch.

We complement our theoretical results with numerical experiments on both synthetic and real datasets from online retail operations. In our real dataset, we focus on the online control problem of {\it case breaking} in a warehouse.
 {
In particular, when a shipment (or {\it case}) of products arrives at an inbound node (also referred to as a {\it cross-dock}), it must be sorted and prepared for transport to downstream fulfillment centers. The preparation process includes potentially ``breaking" the case --- that is, opening up the case and dividing the contents into units to be sorted and subsequently sent to different downstream nodes. Case-breaking can be performed by a number of resources: either manually, by workers, or by automated robotic machinery, both of which are costly. The {value} of breaking a given case {depends} on {factors such as} the sales velocity of the product and {its} current inventory position at downstream fulfillment centers. {Moreover}, {certain} resources are better-suited to handle certain types of shipments than others; for instance, some machines are designed to handle heavier cases (e.g., those exceeding 20 pounds). Finally, cases may not be broken at all if the costs outweigh the benefits, or if there is insufficient case-break capacity. 

Given the staffing levels of manual case-breaking lines and the physical limits of automated lines, there are {throughput} constraints on the amount of case-breaking that each resource can perform. These {constraints} can vary throughout the day as well; for instance, the number of workers assigned to case-breaking operations may {fluctuate across shifts}, while automated machines have limited throughput rates and conveyance capacity. While exceeding {throughput capacity} results in costly backlogs, under-utilizing {it} causes wasteful idleness. To {manage this trade-off}, warehouse managers regulate the assignment of cases to resources throughout the day by specifying {\it hourly targets} for each resource { --- that is, benchmarks specifying the fraction of shipments that each resource should have broken by the end of each hour}. These targets are {\it cumulative}, reflecting the fact that assignments can have a long-lasting impact throughout the rest of the day (via, e.g., backlogs), as described above. {Finally, mirroring the time-varying nature of the throughput constraints, the targets observed in the real case-breaking data used for our computational experiments exhibit a high degree of nonstationarity (see \Cref{fig:hourly-targets}).}

Operations managers evaluate the efficiency of case-breaking decisions at the end of each day according to two key metrics: (i) the financial costs associated with case-breaking (or lack thereof), and (ii) the deviation from hourly targets. Assigning shipments greedily to the least-costly resource will incur small financial costs, but may result in costly over- or under-utilization of resources. Conversely, strictly adhering to the hourly targets throughout the day may result in significant (and potentially unnecessary) financial costs of assignment. Operations managers seek to efficiently trade off between these two metrics.
}

We compare the performance of our algorithm to two myopic extensions of state-of-the-art algorithms for the single-epoch setting. Across all sets of experiments, our algorithm strongly outperforms these benchmarks, demonstrating the practical importance of jointly optimizing over all epochs. Our experiments moreover provide insight into (i) the trade-off between assignment and deviation costs, and (ii) the impact of target nonstationarity on policy performance.

\subsubsection*{Paper organization.} In the rest of this section, we review the related literature. Our model is defined in \Cref{sec:model}, and we show the inefficacy of naive target-following algorithms in \Cref{sec:challenges}. We propose and analyze our proxy assignment algorithm in \Cref{sec:main-result}. We test our algorithm's performance in computational experiments in \Cref{sec:exper}. Conclusions are provided in \Cref{sec:conclusion}.

\subsection{Related Literature}\label{sec:related-work}

Our work contributes to the extensive literature on online resource allocation, for which the vast majority of papers consider {\it time-separable} objectives (see \citet{balseiro2023survey} for an excellent survey). More recently, however, a small body of work has studied online allocation in settings where the objective is {\it non-separable} and couples allocation decisions throughout the horizon. The non-separable function that acts on the consumption of resources across time is typically referred to as a {\it regularizer}, and has been used to encode fairness considerations, minimum thresholds on the allocation of resources, and load balancing objectives, amongst others. We refer to this problem as the {\it regularized online allocation problem}. 

\smallskip
\noindent\textbf{Regularized online allocation under incomplete information.} The problem we consider is most closely related to existing works on regularized online allocation when requests arrive according to a fixed but unknown distribution. This line of work was first introduced by \citet{agrawal2014fast}, who consider a general online stochastic convex programming problem with concave objective and convex constraints. In their setting, the objective function is revealed to the decision-maker in the beginning of the first round, and takes as argument the average allocation throughout the horizon. They derive a primal-dual algorithm that achieves $O(\sqrt{T})$ regret with respect to the offline benchmark, for both the random permutation model --- in which the sequence of arrivals is fixed in advance and comes in a uniformly random order --- and the stochastic i.i.d. model, which corresponds to the arrival model we consider in this work. Motivated by online ad allocation, \citet{balseiro2021regularized} later consider a model in which the budgeted decision-maker observes the reward and resource consumption function of each arrival at the beginning of each period; additionally, the decision-maker observes a known, concave regularizer at the beginning of the horizon that acts on the average consumption of resources across time. Leveraging the dual mirror descent framework of \citet{balseiro2023best}, they design a primal-dual algorithm that similarly achieves $O(\sqrt{T})$ regret relative to the optimal hindsight allocation in the stochastic i.i.d. setting; for the adversarial setting, they show that this same algorithm asymptotically guarantees at least $1/\alpha$ fraction of the hindsight optimum, for some constant $\alpha \geq 1$, under mild conditions on the regularizer. \citet{ma2022optimal} later show that $O(\log T)$ is achievable for this problem in the stochastic i.i.d. setting when the reward function satisfies a locally second-order growth condition. \citet{chen2022network} apply these ideas to the problem of price-based network revenue management under demand uncertainty, when the decision-maker must also regularize consumption throughout the horizon. Also related is \citet{chen2015dynamic}, who study the problem of online matching with concave returns in an adversarial setting. 

The deviation cost functions in our setting are examples of regularizers, as they act on the average consumption of each resource in each epoch. However, our paper differs from the above works in that we consider a {\it multi-epoch} regularized allocation problem. To the best of our knowledge, we are the first to consider this practically relevant variant of the regularized online allocation problem in a stochastic setting.{\footnote{Within the adversarial framework, \cite{bhalgat2012online} study a thematically related ad allocation problem with time-varying upper bounds on the number of impressions each advertiser receives by the end of pre-specified time intervals.}} Though we leverage some of the algorithmic developments of \citet{balseiro2021regularized} for our primal-dual algorithm, we show that incorporating {\it multiple} regularizers throughout the horizon, as opposed to a single one at the end of the horizon, requires the novel technique of proxy assignments, as described in \cref{sec:challenges}.

\smallskip

\noindent\textbf{Regularized online allocation under complete information.} Settings in which the decision-maker knows the distribution from which requests are drawn have also been studied in recent years. For instance, motivated by the fact that load balancing is a costly endeavor in retail and computing operations, \citet{freund2023end} study unconstrained end-of-horizon load balancing problems through the lens of the power-of-$d$-choices framework, designing lightweight algorithms that effectively trade off between the cost of load balancing (similar to an assignment cost in our setup) and that of having an imbalanced load across resources at the end of the horizon. \citet{li2024online} considers the online ad delivery problem with customer choice, when the platform seeks to balance clicks across advertisers.  \citet{freund2025fair} study the problem of allocating rewards to individuals in order to maximize retention, subject to long-run average distributional fairness constraints. In a similar vein, \citet{iyengar2024online} study the problem of designing policies that match gig economy workers to jobs, under long-run average minimum allocation guarantees.  
In contrast to the setting we consider, the regularization considered in these latter papers couples {\it all} decisions throughout the horizon, as opposed to auditing fairness across time. \citet{sinclair2023sequential}, \citet{banerjee2023online}, and \cite{onyeze2025sequential} situate themselves at the other extreme, studying fairness-efficiency trade-offs when the decision-maker compares allocations across each pair of periods in the horizon. 

\smallskip

\noindent\textbf{Online convex optimization with long-term constraints.} This paper also relates at a high level to the line of work on online convex optimization with long-term constraints. Under this framework, the decision-maker seeks to maximize cumulative rewards throughout the horizon; however, rewards are revealed {\it after} the decision is made, rather than before, as in our model. Decisions are moreover coupled across time by a regularizer. The goal is to simultaneously achieve sublinear regret with respect to the best fixed decision in hindsight, as well as sublinear constraint violations \citep{mahdavi2012trading,jenatton2016adaptive,yu2020low}. Primal-dual algorithms have similarly been leveraged to great success within this framework. {Our problem also bears some similarities to linear quadratic regularization in the stochastic control literature, which studies settings in which both assignment and deviation cost functions are quadratic \citep{abbasi2011regret}. 
}

\smallskip

\noindent\textbf{Fair dynamic allocation.} Finally, as alluded to above, much of the motivation surrounding regularized online allocation stems from fairness considerations. The problem of fair dynamic allocation has grown in popularity in the operations literature over the past decade. Though philosophically related --- since fairness objectives are by definition regularized --- the models studied are less closely related to our work as they treat fairness as the single objective, without efficiency or revenue considerations \citep{lien2014sequential, ma2022group, bateni2022fair,balseiro2022uniformly, manshadi2023fair}. In our setting, sole target-tracking is easy, and can be solved at the start of the horizon. The main challenge is in efficiently trading off between the unknown costs of assignment and the costs incurred from target deviation in each epoch.

%% file: parts/preliminaries.tex
\section{Problem Formulation}\label{sec:model}

{In this section we present our model of the {throughput-constrained} online resource allocation problem. A discussion of our modeling assumptions is provided at the end of this section.}

\smallskip 

\noindent\textbf{Notation.} In what follows, for $k \in \mathbb{N}$ we let $[k] = \{1,\ldots,k\}$. We use the notation $(\cdot)^+ = \max\{\cdot,0\}$ to denote the positive part of a real scalar, and let $\lfloor\cdot\rfloor$ and $\lceil\cdot\rceil$ respectively denote the floor and ceiling functions. Finally, we let $\emptyset$ be the empty set.

\smallskip 

\noindent\textbf{Arrival process.}
We consider a problem wherein a decision-maker makes a sequence of assignment decisions to a set of $m$ distinct resources throughout a finite horizon of length $T \in \mathbb{N}$. In each period $t \in [T]$, an arrival of type $j^t$ is drawn from a finite set of $n$ distinct types. Each type $j \in [n]$ is defined by a vector of assignment costs $c_j = (c_{j1},\ldots,c_{jm})\in\mathbb{R}^{m}$ and a set of resources $\mathcal{S}_j\subseteq [m]$ to which it can be assigned, where $c_{ji}$ is the cost of assigning an arrival of type $j$ to resource $i$. 
We assume each arrival consumes one unit of resource capacity (i.e., capacity is measured in number of arrivals that can be processed).\footnote{The assumption that an arrival consumes a single unit of capacity is for expositional simplicity. Entirely analogous arguments show that our algorithm retains its theoretical guarantees when arrivals consume multiple units of resource, as long as the consumption quantity is drawn i.i.d. from a distribution with bounded support and constant variance, in each period.}
If an arrival has not been assigned to any resource, we either say it has been rejected, or it has been assigned to the outside option, also denoted by the empty set. In this case, the decision-maker incurs a cost of 0. We assume assignment costs are bounded, with $c_{\max} = \max_{i,j}|c_{ji}| < \infty$.

We assume the arrival process is {stationary and independent}, denoting by $\omega = (j^1,\ldots,j^T)$ a sample path of arrivals throughout the horizon. Define $p_j = \mathbb{P}(j^t = j)$ for all $t$, with $\sum_{j \in [n]} p_j = 1$; we moreover assume that $p_j$ does not scale with $T$, and that $p = (p_1,\ldots,p_n)$ is {\it unknown} to the decision-maker. Finally, we let $\Lambda(t) = \left(\Lambda_1(t),\ldots,\Lambda_n(t)\right) \in \mathbb{N}^n$ denote the vector of arrivals of each type by period the end of $t$, with $\Lambda_j(t) = \sum_{\tau \leq t} \mathds{1}\{j^{\tau} = j\}$. For ease of notation we define $\Lambda_j(t_1:t_2) = \sum_{\tau \in (t_1,t_2]}\mathds{1}\{j^\tau = j\}$ for all $t_1 < t_2$.

\smallskip

\noindent\textbf{Deviation costs.}
The horizon is partitioned into $K \in \mathbb{N}$ epochs with $K = \Theta(1)$. Without loss of generality, we assume $T$ is a multiple of $K$; we moreover assume that there is an equal number of $T/K$ arrivals in each epoch. {(Stationarity in the number of arrivals per epoch is not critical; a simple time transformation to extend the model to non-stationary arrivals is provided in Appendix \ref{apx:nonstationary}.)} We use $\mathcal{T}_k = \left\{\frac{(k-1)T}{K}+1,...,\frac{kT}{K}\right\}$ to denote the periods in epoch $k \in [K]$. We moreover let \mbox{$\omega_k = (j^t, t \in \mathcal{T}_k)$} be the sequence of arrivals in epoch $k$.



\smallskip 

At the end of each epoch $k$, the decision-maker computes the running average number of assignments to each resource $i$, {from the start of the horizon up until the end of epoch $k$}, i.e., from period 1 to period $kT/K$. (We will alternatively refer to this running average number of assignments as the running average consumption, or the cumulative average consumption.) She then compares this running average consumption to a target consumption, denoted by $\rho_{ki} \in [0,1]$. For each period up until the end of epoch $k$, she incurs an amortized cost for deviating from $\rho_{ki}$. Let $g_{ki}: [0,1]\to\mathbb{R}_{\geq 0}$ denote the deviation cost function, and assume $g_{ki}$ is $L$-Lipschitz continuous and convex, with $g_{ki}(\rho_{ki}) = 0$. We provide examples below.

\smallskip 

\begin{example}
[No penalty] The deviation cost is trivially $g_{ki}(a) = 0$ for all $a \in [0,1]$. The decision-maker's problem reduces to an unconstrained cost minimization problem, for which the optimal solution is to assign each arrival to the resource for which the cost of assignment is the most negative (or to reject the arrival if all assignment costs are strictly positive). 
\end{example}

\smallskip

\begin{example}[Underage/overage penalty]\label{ex:overage} The deviation cost is defined as $\mbox{$g_{ki}(a)=\delta_{ki}^+(a-\rho_{ki})^++\delta_{ki}^-(\rho_{ki}-a)^+$}$ for all $a \in [0,1]$, $\delta_{ki}^+, \delta_{ki}^- \in \mathbb{R}_{\geq 0}$. For $\delta_{ki}^+ > 0$, this deviation cost models a type of soft budget constraint, e.g., the cost of overtime pay if labor is over-assigned work. For $\delta_{ki}^- > 0$, this deviation cost models a minimum consumption threshold for the resource, e.g., to avoid costly idling. Finally, a special case of this penalty function is when $\delta_{ki}^+ = \delta_{ki}^- > 0$, which reduces to the absolute deviation $g_{ki}(a) = \delta_{ki}^+|a-\rho_{ki}|$. In this case, there is a specific {throughput} level that the decision-maker seeks to adhere to in each epoch. This metric was important in the supply chain context that motivated our work, so we used it for our computational experiments in \cref{sec:exper}. 
\end{example}

\smallskip

\begin{example}[Squared deviation penalty] The deviation cost is defined as the squared deviation between the running average consumption of resource $i$ in epoch $k$ and the target $\rho_{ki}$, i.e., $\mbox{$g_{ki}(a) = \delta_{ki}(a-\rho_{ki})^2$}$ for all $a\in [0,1]$, $\delta_{ki} > 0$. As for the absolute deviation cost function, this deviation cost models settings in which the decision-maker has a specific {throughput rate} to which she wants to adhere. The squared deviation penalty, however, penalizes target deviation more severely. 
\end{example}

\smallskip

Given a running average consumption $a$ over the past $kT/K$ periods, the {\it total} deviation cost incurred by the decision-maker for resource $i$ is given by $\frac{kT}{K}\cdot g_{ki}(a)$. The motivation behind this scaling is as follows: since $\rho_{ki}$ represents a per-period rate of consumption up until the end of epoch $k$, one can associate each period $t \leq kT/K$ with a cost of deviating $g_{ki}(a)$, since each assignment decision in those periods contributed to this cost. Summing this amortized cost over all $t \leq kT/K$, we obtain our final scaling. Note that when $K = 1$, we recover the same scaling as in \citet{balseiro2021regularized} and \citet{agrawal2014fast}.

\smallskip 

\begin{remark}
Though the deviation cost incurred by the decision-maker is separable across resources, the targets themselves may depend on each other. For instance, we may have $\mbox{$\rho_{ki} = 1/m$}$ for all $i \in [m]$, representing that arrivals should be allocated fairly across resources. We may also have, for instance, \mbox{$\sum_{i=1}^m \rho_{ki} = 1$}, with $\rho_{ki}$ proportional to the capacity of resource $i$ relative to that of the other resources.
\end{remark}

\smallskip 

\begin{remark}
Our model allows for both positive and negative assignment costs. In a standard resource allocation setting, there would be no reason to assign an arrival of type $j$ to resource $i$ if $c_{ji} > 0$. This is because the assignment cost of the outside option is normalized to 0, and there is no penalty for target deviation. In our setting, however, this is no longer true. Consider for instance the squared deviation function $g_{ki}(a) = \delta_{ki}(a-\rho_{ki})^2$, where $\delta_{ki}$ is large and $\rho_{ki} > 0$. In this case, it may indeed be optimal to assign the arrival to resource $i$ in epoch $k$, despite a positive assignment cost.
\end{remark}

\smallskip

\noindent\textbf{Online control policy.} The decision-maker's goal is to minimize the sum of assignment and deviation costs incurred throughout the horizon. To formalize this, we use $\pi$ to denote an online policy, which takes as input the current period $t$, the type of the arrival, and the history of past arrivals and actions, and outputs a feasible assignment of the arrival to the set of resources. For $t\in [T]$, we use the random variable $Z^\pi_{ji}(t) \in \mathbb{N}$ to denote the number of type $j$ arrivals assigned to resource $i$ by the end of period $t$, with $\mbox{$Z^\pi(t) = (Z_{ji}^\pi(t), j\in[n], i\in[m])\in \mathbb{N}^{n\times m}$}$, and $Z^\pi_{ji}(0) = 0$. For ease of notation, for any vector of assignments $Z(t) \in \mathbb{N}^{n\times m}$, we define $Z_i(t) = \sum_{j\in[n]}Z_{ji}(t)$ to be the total number of arrivals assigned to resource $i$ by the end of period $t$. For any two periods $t_1 < t_2$, we use the shorthand $\mbox{$Z_{ji}(t_1:t_2) = Z_{ji}(t_2)-Z_{ji}(t_1)$}$ to refer to the number of type $j$ arrivals assigned to resource $i$ in the interval $(t_1,t_2]$, with $\mbox{$Z_i(t_1:t_2) = \sum_{j\in[n]}Z_{ji}(t_1:t_2)$}$.

The total cost incurred by $\pi$ throughout the horizon is:
\begin{align}\label{eq:alg-cost}
V^{\pi}[\omega] := \sum_{j \in [n]}\sum_{i\in[m]}c_{ji} Z_{ji}^{\pi}(T) + \frac{T}{K}\sum_{k \in [K]}\sum_{i\in[m]} kg_{ki}\left(\frac{Z_i^{\pi}(kT/K)}{kT/K}\right)
\end{align}

\smallskip

\noindent\textbf{Offline benchmark.} We compare the performance of any online policy to the offline benchmark (also referred to as the {\it hindsight optimum}) which has full knowledge of $\omega$ upfront, and solves a convex program to compute the optimal assignment of arrivals to resources:
\begin{align}
V^{\textsc{off}}[\omega] := \min  &\sum_{j \in [n]}\sum_{i\in[m]}c_{ji} Z_{ji}(T) + \frac{T}{K}\sum_{k \in [K]}\sum_{i\in[m]} kg_{ki}\left(\frac{Z_i(kT/K)}{kT/K}\right) \notag  \\
\text{s.t.} & \sum_{i\in [m]} Z_{ji}\left(\frac{(k-1)T}{K}:\frac{kT}{K}\right) \leq \Lambda_j\left(\frac{(k-1)T}{K}:\frac{kT}{K}\right) \quad \forall \ j \in [n], k \in [K] \label{eq:feasible-assignment} \\
&Z_{ji}(T) = 0 \quad \forall  \ j \in [n], \ i \not\in \mathcal{S}_j \notag \\
&Z_{ji}(kT/K) \in \mathbb{N} \quad \forall \ j \in [n], i \in [m], k \in [K], \notag
\end{align}
where \eqref{eq:feasible-assignment} encodes the constraint that the number of assignments of type $j$ during epoch $k$ cannot exceed the number of type $j$ arrivals during that same epoch.
 
We measure the {\it regret} of policy $\pi$ as the difference between the cost of the hindsight optimum and that incurred by the policy, i.e., 
\begin{align}
\textsc{Reg}[\omega] := V^\pi[\omega] - V^{\textsc{off}}[\omega]
\end{align}

Our aim is to design policies with low expected regret $\mathbb{E}[\textsc{Reg}]$, where the expectation is taken over the randomness in the arrival sequence and the algorithm's decisions. Specifically, we would like the expected regret to be sublinear in the horizon length $T$.

\smallskip 

\begin{remark}
When $K = 1$, our setting is a special case of the problem considered by \citet{balseiro2021regularized}, for which the expected regret of any online policy is lower bounded by $\Omega(\sqrt{T})$ in the worst case.\footnote{Though their $\Omega(\sqrt{T})$ lower bound is derived from a finite capacity constraint $bT$ for a single resource, an equivalent instance can be constructed for our uncapacitated setting, with the deviation cost for a single epoch $g(a) = \delta \left(a-b\right)^+$, for large enough $\delta > 0$.}
\end{remark}

{\subsection{Discussion of Modeling Assumptions}
We conclude this section by discussing our main modeling assumptions. 

\smallskip 

\noindent\textbf{On the cumulative nature of deviation costs.} One of the defining features of our model is the fact that the deviation costs associated with each epoch are {\it cumulative}, i.e., they regularize the {\it running} average assignment from the beginning of the horizon to the end of the epoch, as opposed to the average assignment {\it within} each epoch. As alluded to in the introduction, one setting in which such cumulative regularization is important is in physical processing settings, where backlogs and idling are a concern. To see this, consider a single-resource setting in which $T=4$ and $K = 2$. Suppose moreover that the resource ideally processes one arrival per epoch (this can alternatively be viewed as taking two periods to process each arrival).  Our model represents this setting by defining targets $\rho_1=\rho_2 = 1/2$. This is equivalent to defining an ``absolute'' target of one arrival accepted by the end of the first epoch, and two arrivals accepted by the end of the second epoch. Non-cumulative targets, on the other hand, would specify an absolute target of one accepted arrival {in} each epoch. Consider now a sample path such that the decision-maker accepts two arrivals in the first epoch (which would occur if, for instance, the assignment costs for both arrivals were very low). In reality, the resource still has to process that additional arrival, which can be thought of as ``spilling over'' to the next epoch. Then, the work done to process that extra arrival has essentially taken up the ``budget of work'' that the resource can perform in the second epoch. By modeling the second target as a cumulative target of two arrivals across {\it both} epochs, we are effectively representing the idea that by rejecting all arrivals in the second epoch, the decision-maker will have cleared that backlog. In contrast, in the non-cumulative model, if the decision-maker accepted one arrival in the second epoch (which would be considered ``on-target''), she would now have {\it two} arrivals to process in the second epoch, due to the first-epoch spillover. Since the resource's ideal processing rate is one arrival per epoch, however, one of these two arrivals will spill over to after the end of the horizon, a ``bad outcome'' that is not represented by non-cumulative targets.

Cumulative targets are similarly able to penalize idling in a way that non-cumulative targets cannot. Consider again the above setting, and suppose the decision-maker rejected both arrivals in the first epoch (due to, e.g., very high assignment costs), and accepted one arrival in the second epoch. In this case, non-cumulative targets would only penalize idling in the first epoch. However, what this ignores is that the resource {\it could have} processed two arrivals throughout the entire horizon. Therefore, not only was the resource inefficiently utilized in the first epoch, but its total ``budget of work'' for the entire horizon was unused. Again, in this case, cumulative targets are able to capture such a bad outcome.

 We note that our model can also capture hard, cumulative capacity constraints throughout the horizon by defining the deviation cost functions as overage penalties, as in \Cref{ex:overage}, for large enough $\delta_{ki}^+ > 0$, $k \in [K], i \in [m]$. Such cumulative constraints arise in settings such as outbound truck transportation, where trucks are loaded in an online fashion. For instance, a decision-maker may have two trucks with finite capacity $C$. Both trucks have a scheduled departure time of 5 p.m.; while one is available throughout the day, the other only becomes available for loading after 12 p.m. In this case, there is a budget constraint of $C$ for arrivals in the first half of the day, and a budget constraint of $2C$ for all arrivals throughout the day. Cumulative budget constraints may also arise in online advertising, where an advertiser has a maximum budget on her ad spend throughout the day, but wants to pace the rate at which this budget is allocated (e.g., avoiding outcomes where the budget is used up entirely in the first half of the day) \citep{bhalgat2012online}.  While these considerations can be dealt with via within-epoch budgets, cumulative budgets give the decision-maker enough flexibility to spend more in later epochs if earlier epochs saw limited spend. In this sense, cumulative constraints can be viewed as a flexible middle ground between unpaced (i.e., a single end-of-horizon budget) and strictly paced (i.e., within-epoch budgets) consumptions. 


\smallskip 

\noindent\textbf{On the time-varying targets.} Another important feature of our model is the {\it exogenous} and potentially {\it time-varying} nature of targets. For instance, in the case breaking setting discussed in the introduction, targets are determined well in advance of the decision-making horizon. They are typically generated by forecasts of arrival rates and resource capacities in each hour, and set in a way that limits predicted backlogs and idling. As we will see in real-world data we analyze, in practice these targets can vary significantly throughout the day. {This can be due to the throughput capacity itself being time-varying.} For example, labor and policies may restrict the flexibility of staffing levels throughout day, due to minimum/maximum shift lengths and mandated break times; process designs may restrict the number and configuration of workers assigned to a process; machines may go down for regular scheduled maintenance, etc. Such phenomena result in variability of available capacity throughout the day, independent of the rate of work that arrives in each epoch. {In addition, even if throughput capacity is constant, time-varying targets arise when a nonstationary arrival rate is transformed into an instance of our stationary arrival rate model by rescaling time to make the number of arrival per epoch constant. We outline this time transformation reduction in Appendix \ref{apx:nonstationary}.}
}

\input{parts/challenges}

%% file: parts/challenges.tex
\section{Challenges of Multi-Epoch Target Following}\label{sec:challenges}

In this section we establish the unique challenges associated with multi-epoch target following by analyzing the performance of (i) myopic policies that fail to take into account future targets, and (ii) naive modifications of state-of-the-art algorithms for single-epoch problems. 
We show that both can potentially result in severely suboptimal performance.

\subsection{Failure of Myopic Policies}\label{sec:bad-example-myopic}

We first define the myopic problem associated with a single epoch $k$, which minimizes assignment and deviation costs only for that epoch, with no eye toward the future. 

\smallskip 

\begin{definition}
Given current aggregate consumption $z = (z_1,\ldots,z_m)\in\mathbb{N}^m$, the \emph{myopic offline optimum} for epoch $k \in [K]$ is defined as:
\begin{align}
{V}_k^{\textsc{myo-off}}\left[\omega_k \mid z\right] := \min  &\sum_{j \in [n]}\sum_{i\in[m]}c_{ji} Z_{ji}\left(\frac{(k-1)T}{K}:\frac{kT}{K}\right) + \frac{kT}{K}\sum_{i\in[m]} g_{ki}\left(\frac{z_i + Z_i\left(\frac{(k-1)T}{K}:\frac{kT}{K}\right)}{kT/K}\right) \notag \\
\emph{s.t.} & \sum_{i\in [m]} Z_{ji}\left(\frac{(k-1)T}{K}:\frac{kT}{K}\right) \leq \Lambda_j\left(\frac{(k-1)T}{K}:\frac{kT}{K}\right) \quad \forall \ j \in [n] \notag \\
&Z_{ji}\left(\frac{(k-1)T}{K}:T\right) = 0 \quad \forall \ j \in [n], i \not\in\mathcal{S}_j\notag \\
&Z_{ji}\left(\frac{(k-1)T}{K}:\frac{kT}{K}\right) \in \mathbb{N} \quad \forall \ j \in [n], i \in [m]. \notag 
\end{align}
{When $k=1$, we abuse notation and omit the dependence of $V_1^{\textsc{myo-off}}[\omega_1 \mid z]$ on $z$.}
\end{definition}

\smallskip

For any policy $\pi$, let $V_k^\pi\left[\omega_k \mid Z^\pi\left(\frac{(k-1)T}{K}\right)\right]$ denote the assignment and deviation costs incurred during epoch $k$.  We say that $\pi$ has {\it myopic vanishing regret} if there exists $\alpha \in [0,1)$ such that, for all $k \in [K]$ and $Z^\pi\left(\frac{(k-1)T}{K}\right) \in \mathbb{N}^m$,
\[\mathbb{E}\left[V_k^\pi\left[\omega_k \mid Z^\pi\left(\frac{(k-1)T}{K}\right)\right]-{V}_k^{\textsc{myo-off}}\left[\omega_k \mid Z^\pi\left(\frac{(k-1)T}{K}\right)\right] \bigg{|} Z^\pi\left(\frac{(k-1)T}{K}\right)\right] = O(T^\alpha).\]
Notice that the myopic problem for epoch $k$ is a single-epoch regularized allocation problem. Since state-of-the-art algorithms guarantee $O(\sqrt{T/K})$ regret relative to the myopic offline optimum \citep{balseiro2021regularized}, these latter policies have myopic vanishing regret.

The following proposition shows that there exist trivial instances for which myopic algorithms can incur linear regret with respect to the true, multi-epoch offline optimum. 
\begin{proposition}\label{prop:myopic-is-bad}
{There exists an instance for which $\mbox{$\mathbb{E}[\textsc{Reg}] = \Omega(T)$}$ for any policy $\pi$ that has myopic vanishing regret.}
\end{proposition}

\begin{proof}{\it Proof.}
{Consider an instance in which there is $m=1$ resource, $n=1$ type and deterministic assignment cost $c_1 = -1$. Let $K = 2$, with deviation cost functions $g_1(a) = \delta a$ and \mbox{$g_2(a) = \delta(1-a)$}, for all $a \in [0,1]$, and $\delta > 1$ (i.e., epochs 1 and 2 respectively have target consumptions 0 and 1).}
For ease of notation, we suppress the dependence of all quantities on the type and resource. Plugging the assignment and deviation costs for this instance into \eqref{eq:alg-cost}, for any policy $\pi$ we have:
\begin{align}\label{eq:myo-lb}
V^\pi[\omega] &= -Z^\pi(T) +\delta \cdot \frac{T}{2}\left(\frac{Z^\pi(T/2)}{T/2} +2\left(1-\frac{Z^\pi(T)}{T}\right)\right) \notag \\ 
&= \left(-Z^\pi(T/2) + \delta\cdot \frac{T}{2}\cdot \frac{Z^\pi(T/2)}{T/2} \right) + \left(-Z^\pi(T/2:T) + \delta \cdot \frac{T}{2}\cdot 2\left(1-\frac{Z^\pi(T/2) + Z^\pi(T/2:T)}{T}\right) \right).
\end{align}
Notice that the first term in parentheses is lower bounded by 0, since $\delta > 1$. Again using $\delta > 1$, we have that the optimal myopic solution in epoch 1 is to reject all arrivals, yielding $V_1^{\textsc{myo-off}}[\omega_1] = 0$. Moreover, since $\pi$ has myopic vanishing regret, this implies that, in epoch 1:
\begin{align*}
\mathbb{E}\left[-Z^\pi(T/2) + \delta\cdot \frac{T}{2}\cdot \frac{Z^\pi(T/2)}{T/2}\right] = O(T^{\alpha}) \iff \mathbb{E}\left[Z^\pi(T/2)\right] = O(T^{\alpha}).
\end{align*}
Taking the expectation of $V^\pi[\omega]$ over the randomness in the algorithm's decisions, plugging these facts back into \eqref{eq:myo-lb} and re-arranging, there exists a constant $C > 0$ such that:
\begin{align*}
\mathbb{E}\left[V^\pi[\omega]\right] &\geq \mathbb{E}\left[-Z^\pi(T/2:T) + \delta \cdot \frac{T}{2}\cdot 2\left(1-\frac{ Z^\pi(T/2:T)}{T}\right) \right]-\delta CT^{\alpha} \\
&\geq V_2^{\textsc{myo-off}}[\omega_2 \mid 0] - \delta CT^{\alpha},
\end{align*}
where the second inequality follows from the fact that $Z^\pi(T/2:T)$ is feasible for $V_2^{\textsc{myo-off}}[\omega \mid 0]$, and removes the expectation as there is no randomness in the arrival sequence.

Hence, it suffices to show that $V_2^{\textsc{myo-off}}[\omega_2 \mid 0] - V^{\textsc{off}}[\omega] = \Omega(T)$. Re-writing the objective of the myopic offline optimum for epoch 2, we have:
\[V_2^{\textsc{myo-off}}[\omega_2 \mid 0] = \min \  -Z(T/2:T) + \delta T \left(1-\frac{Z(T/2:T)}{T}\right) = -T/2 +\delta T/2,\]
achieved at $Z(T/2:T) = T/2$.

For the true offline optimum (which optimizes across both epochs), we have:
\begin{align*}
V^{\textsc{off}}[\omega] &= \min \ -Z(T) + \delta\left(Z\left(\frac{T}{2}\right) + (T-Z(T))\right)= \min \ \delta T -Z(T) -\delta\cdot Z\left(\frac{T}{2}:T\right)= -T + \delta T/2.
\end{align*}

Putting these two equalities together, we obtain $V_2^{\textsc{myo-off}}[\omega_2 \mid 0]- V^{\textsc{off}}[\omega] = T/2$, completing the proof.
 \hfill\Halmos
\end{proof}

\subsection{Failure of Naive Lagrangian-Based Algorithms}\label{sec:bad-example-naive}

\cref{prop:myopic-is-bad} illustrates that, for an algorithm to have any hope of attaining sublinear regret with respect to the offline optimum, it must take into account targets associated with future epochs when making its assignment decisions. We now show that natural extensions of primal-dual algorithms for the single-epoch setting yield poor performance, again in very simple settings.

In order to do so, we first describe the dual mirror descent framework for the single-epoch problem, developed in \citet{balseiro2021regularized}. At a high level, their algorithm makes a primal assignment decision in each period $t$, denoted by $x^t \in \mathcal{X}^t$, where $\mathcal{X}^t$ represents the set of feasible assignments for this arrival. It then introduces an auxiliary variable $a \in [0,1]^m$ that represents the average consumption of each resource throughout the horizon, and a dual variable $\mu\in\mathbb{R}^m$ that allows the algorithm to decouple $x^t$ from this average consumption. The algorithm makes three decisions in each period. First, given the current arrival and shadow price of each resource, it assigns the arrival to the resource that minimizes the dual-adjusted assignment cost; it then computes an estimate of a dual-adjusted ``idealized'' average consumption $a^t$ for the horizon. Finally, it updates the shadow price for each resource to reconcile the difference in actual and idealized consumptions. For clarity of exposition, we present their algorithm in Appendix \ref{apx:balseiro-et-al}.

Since the auxiliary variable $a$ represents the average consumption of resources throughout the horizon, the natural extension to the multi-epoch setting is to introduce $K$ auxiliary variables $a_k \in \mathbb{R}^m$, $k \in [K]$, each representing the running average consumption throughout the last $k$ epochs. To formalize this, we reformulate the offline problem as a function of the primal assignment decision $x^t$ made in period $t$. We moreover abuse notation and let $c^t = (c^t_1,\ldots,c^t_m)$ denote the cost vector associated the arrival in the $t$-th period. Then, the hindsight optimum can equivalently be written as:
\begin{align}
V^{\textsc{off}}[\omega] = \min &\sum_{t\in[T]}\sum_{i\in[m]}c^t_{i} x^t_i + \frac{T}{K}\sum_{k \in [K]}\sum_{i\in[m]} kg_{ki}\left(a_{ki}\right) \notag  \\
\text{s.t.} 
\qquad &a_{ki} = \frac{\sum_{t\leq kT/K} x^t_i}{kT/K} \quad \forall \ k\in[K], i \in [m] \label{constr:relax}\\
&x^t \in \mathcal{X}^t \quad \forall \ t \in [T], \quad a \in [0,1]^{K\times m}. \notag 
\end{align}
Consider now the Lagrangian relaxation that relaxes constraint \eqref{constr:relax} by introducing dual variables $\mu\in\mathbb{R}^{K\times m}$. Letting $k^t$ denote the epoch that period $t$ is in, the corresponding Lagrange dual function is given by:
\begin{align*}
\mathcal{L}(\mu \mid \omega) =  &\sum_{t\in[T]}\min_{x^t\in\mathcal{X}^t}\sum_{i\in[m]}\left(c^t_{i}-\sum_{k\geq k^t}\mu_{ki}\right) x^t_i + \frac{T}{K}\sum_{k \in [K]}\sum_{i\in[m]} k \min_{a_{ki}\in[0,1]}\left(g_{ki}\left(a_{ki}\right) + \mu_{ki}a_{ki}\right). \notag 
\end{align*}
\Cref{alg:bad-example} describes the natural primal-dual algorithm arising from $\mathcal{L}(\mu\mid\omega)$.

\begin{algorithm}[t]
\DontPrintSemicolon 
\KwIn{Stepsize $\eta = O(1)$, initial shadow prices $\mu^{1} \in \mathbb{R}^{K\times m}$.}
\KwOut{Sequence of assignment decisions.}
\For{$t = 1, \ldots, T$}{
    Observe arrival $t$.\;   
    Compute assignment decision:
   \quad $x^t \in \arg\min_{x\in\mathcal{X}^t}\sum_{i\in[m]}x_i\left(c^t_i-\sum_{k\geq k^t}\mu_{ki}^t\right).$\;
Compute idealized average consumptions for all $k \geq k^t, i \in [m]$:
\quad $a_{ki}^t \in\arg\min_{a\in [0,1]} g_{ki}(a) + \mu_{ki}^ta.$\;
    Update shadow prices via OGD step:
\quad $\mu^{t+1}_{ki} = \mu^{t}_{ki} + \eta(a^{t}_{ki}-{x}^t_i).$
    }
	\caption{Naive Primal-Dual Algorithm}
	\label{alg:bad-example}
\end{algorithm}

\begin{proposition}\label{prop:bad-example}
{There exists an instance in which  \Cref{alg:bad-example} incurs linear regret for any $\eta = O(1)$.}
\end{proposition}

\begin{proof}{\it Proof.}
{Consider an instance in which $m = 1$, $K = 2$, and with a single type that has an assignment cost of 0. For ease of notation, in the remainder of the instance we omit the dependence of all quantities on the index of the single resource. The deviation costs for each epoch are given by $\mbox{$g_1(a_1) = \delta|a_1-\rho_1|$}$, $g_2(a_2) = \delta|a_2-\rho_2|$, and large $\delta > 0$. Suppose moreover that $\rho_1 \in [0,1]$ and $\mbox{$2\rho_2-\rho_1 \in \left[0,1\right]$}$. For any such $\rho_1$, $\rho_2$, the assignment that accepts $\rho_1$ fraction of assignments in the first $T/2$ periods and $2\rho_2-\rho_1$ fraction of assignments in the last $T/2$ periods achieves the two cumulative targets in each epoch.\footnote{We assume for technical simplicity that $T\rho_1$ and $T(2\rho_2-\rho_1)$ are integral.} Finally, let $\delta$ be large enough that \Cref{alg:bad-example} yields $a^t_{k} = \rho_k$ for all $t$.\footnote{The assumption that $a^t_k = \rho_k$ for all $t$ is for expositional simplicity. As long as $\delta$ is a large enough constant, for any $\eta = O(1)$ there exists a constant $t_0$ past which $a^t_k = \rho_k$ for all $t > t_0$, which is all we need to prove our result.}
}

Since assignment costs are uniformly 0 throughout the horizon, for this instance we have $V^{\textsc{off}}[\omega] = 0$, with the optimal solution accepting $\rho_1$ fraction of arrivals in the first epoch, and $2\rho_2-\rho_1$ fraction of arrivals in the second epoch.

We analyze $\frac{\sum_{t\leq T/2}x^t}{T/2}$, the average consumption of the resource in the first epoch under \Cref{alg:bad-example}. Since $a_1^t = \rho_1$, $a_2^t = \rho_2$ for all $t \leq T/2$, the dual variables satisfy:
\begin{align}\label{eq:bad-example-csp}
&\mu_1^{t+1} + \mu_2^{t+1} = \mu_1^t + \mu_2^t + \eta(\rho_1 + \rho_2 - 2x^t) \notag \\
\implies &x^t = \frac{\rho_1 +\rho_2}{2} + \frac{(\mu_1^t + \mu_2^t)-(\mu_1^{t+1} + \mu_2^{t+1})}{2\eta} \notag \\
\implies & \frac{\sum_{t\leq T/2}x^t}{T/2} = \frac{\rho_1+\rho_2}{2} + \frac{(\mu_1^1+\mu_2^1)-(\mu_1^{T/2+1}+\mu_2^{T/2+1})}{\eta T}.
\end{align}

Notice that $x^t = \mathds{1}\{\mu_1^t + \mu_2^t > 0\}$ for all $t \leq T/2$. The dual update step in \Cref{alg:bad-example} implies that $\mu_1^t + \mu_2^t$ is a process with negative drift whenever it is positive, and positive drift whenever it is negative. Since $a^t_{k}$ and $x^t$ are constants that do not depend on $T$, $\mu_1^{t}+\mu_2^{t}$ must be bounded in absolute value by a constant, for any $\eta = O(1)$. So, for $\eta T = \Theta(T^{\epsilon})$, for some $\epsilon > 0$: 
\begin{align}\label{eq:lim}
\lim_{T\to\infty}\frac{(\mu_1^1+\mu_2^1)-(\mu_1^{T/2+1}+\mu_2^{T/2+1})}{\eta T} = 0.
\end{align}
Plugging this fact into \eqref{eq:bad-example-csp}, this then implies that $\frac{\sum_{t\leq T/2}x^t}{T/2}$ converges to $\frac{\rho_1+\rho_2}{2}$ as $T$ grows large. 

Separately, for $\eta T = \Theta(T^{-\epsilon})$ for some $\epsilon \geq 0$, $\frac{\sum_{t\leq T}x^t}{T/2}$ either converges to $\frac{\rho_1+\rho_2}{2}$, or fails to converge.

In the second epoch, similar arguments give us that $\frac{\sum_{t = T/2+1}^T x^t}{T/2}$ converges to $\rho_2$, or fails to converge as $T$ grows large. Hence, the {\it running} average consumption at the end of the second epoch either converges to $\frac{0.5(\rho_1+\rho_2)}{2}+\frac{\rho_2}{2} = \frac14\rho_1 + \frac34\rho_2$, or fails to converge. 

In all cases, there is at least a constant gap between the target and actual running average consumptions in both epochs, for large enough $T$. We thus obtain $\textsc{Reg}^\pi[\omega] = \Omega(T)$. 
\hfill\Halmos
\end{proof}

\smallskip

The proof of \cref{prop:bad-example} illustrates that naively using the dual variables for {future} epochs in the primal assignment decision in period $t$ can yield poor performance, since this induces the primal decision to track all idealized consumptions {\it simultaneously}. This results in undesirable ``mixing'' of targets across epochs, as we observed for the above instance: \Cref{alg:bad-example} accepted $\frac{\rho_1 + \rho_2}{2}$ fraction of arrivals in the first epoch, as opposed to $\rho_1$ fraction of arrivals. In the following section we design an algorithm that avoids this undesirable behavior.

%% file: parts/main-result.tex
\section{Main Result}\label{sec:main-result}

\Cref{alg:bad-example} is motivated by the observation that, when making an assignment decision, an online algorithm must take into account the fact that its current consumption affects its ability to adhere to targets in future epochs. Moreover, dual variables that reflect the deviation between the algorithm's current consumption and the idealized consumption in future epochs contain useful information that an algorithm should use in its decisions. \cref{prop:bad-example} showed, however, that {\it how} to incorporate this information effectively is more subtle than it at first appears. In this section we design an algorithm based on using proxy assignments that addresses this challenge. We then prove its associated sublinear regret guarantee. 

\subsection{Algorithm Description} 

We first describe our algorithm at a high level. As in \Cref{alg:bad-example}, our algorithm maintains a dual variable (also referred to as a shadow price), for each epoch $k \in [K]$. Our algorithm, however, sidesteps the undesirable mixing of targets exhibited in \Cref{prop:bad-example} by {only} using the dual variable for the {\it current} epoch in the primal assignment decision. Then, to keep track of how primal assignment decisions affect future epochs, it also computes a {\it proxy} assignment decision for each future epoch. At a high level, each proxy assignment decision simulates the primal assignment decision the algorithm {\it would} make in that future epoch for the same arrival, given the arrival type and the simulated consumption of the resource. This decision is the solution to a dual-adjusted assignment cost minimization problem, where the dual price used is the one associated with the future epoch of interest. The proxy assignment decisions for future epochs functionally act as endogenous predictions of future consumption. Importantly, they are never actually implemented. However, they play a crucial role in that they indirectly feed back into the primal assignment decisions in the current epoch to increase (resp., decrease) the current consumption of the resource, as the predicted consumption in future epochs begins to deviate from their respective targets.  

Separate from computing the proxy assignment decisions in each period, our algorithm computes auxiliary variables for the current and all future epochs that represent the idealized average consumption of each resource {\it during} each epoch (as opposed to the running average consumption, as in \Cref{alg:bad-example}). These idealized average consumptions solve a dual-adjusted deviation cost minimization problem that couples together the average consumptions across the current and all future epochs. Finally, the dual variables for each epoch are updated in such a way to reconcile the discrepancy between the proxy assignment decision and idealized average consumption. This is done via a standard dual gradient descent step.

We next provide a formal description of our algorithm --- Dual Gradient Descent with Proxy Assignments --- in \Cref{alg:multi-stages}. Consider an arbitrary period $t$, and let $k$ denote the epoch in which it is contained. Given the current shadow price associated with the current and future epochs, denoted by $\mu_{k'}^t$ for $k' \geq k$, our algorithm first computes the proxy assignment decision for the $t$-th arrival for each $k' \geq k$, according to \eqref{eq:proxy-step}. The primal assignment decision for period $t$ is identical to the proxy decision for epoch $k$. We then compute the idealized average consumptions for all epochs $k' \geq k$, as described in \eqref{eq:many-stages-aux-problem}. These idealized average consumptions take as input the current shadow prices $\mu_{k'}^t$, for all $k' \geq k$, as well as the cumulative consumption up until the end of epoch $k-1$, i.e., $\sum_{t' \leq (k-1)T/K}x^{t'}$. Finally, \eqref{eq:ogd-step} describes the dual gradient descent update, which reconciles the proxy assignment decision and the idealized average consumption for each epoch $k' \geq k$. If the idealized consumption $a_{k'i}^{t}$ exceeds the proxy assignment decision $\tilde{x}_{k'i}^{t}$ for resource $i$ in epoch $k'$, $\mu_{k'i}^{t}$ increases. Hence, in period $t+1$, the algorithm is more likely to proxy assign the arrival to resource $i$ for epoch $k'$, according to \eqref{eq:proxy-step}. On the other hand, the idealized consumption for resource $i$ in epoch $k'$ will likely decrease, by \eqref{eq:many-stages-aux-problem}. Conversely, if $\tilde{x}_{k'i}^t$ exceeds $a_{k'i}^t$, $\mu_{k'i}^{t}$ decreases, in which case the reverse holds: the algorithm is less likely to proxy assign the arrival to resource $i$ for epoch $k'$, and the idealized average consumption will likely increase.

\begin{algorithm}[t]
\DontPrintSemicolon 
\KwIn{Stepsize $\eta$, initial shadow prices $\mu^1\in\mathbb{R}^{K\times m}$}
\KwOut{Sequence of assignment decisions}
\For{$t = 1, \ldots, T$}{
    Let $k$ be the current epoch. \;
    If $t = (k-1)T/K+1$, set $\mu^t_{k'} = \mu^1_{k'}$ for all $k' \geq k$. \;
    Observe arrival type $j = j^t$, and make proxy assignment decisions for all $k' \geq k$:
    \begin{equation}\label{eq:proxy-step}
    \begin{aligned}
        \proxyx{k'}{t} \in \arg\min_{x\in \mathcal{X}^t} \quad &\sum_{i \in [m]}x_i(c_{ji}-\dual{k'i}{t})
    \end{aligned}
    \end{equation}\;
    Make primal assignment decision $x^t = \proxyx{k}{t}$.\;

    Compute idealized average consumptions for all future epochs:
    \begin{equation}\label{eq:many-stages-aux-problem}
    \begin{aligned}
        a^t \in \arg\min_{a \in [0,1]^{(K+1-k)\times m}} \quad &\sum_{k' \geq k}k'\sum_{i\in[m]}g_{k'i}\left(\frac{\sum_{t' \leq (k-1)T/K}x^{t'}_i}{k'T/K} + \sum_{k''=k}^{k'}\frac{a_{k'',i}}{k'}\right) + \sum_{k' \geq k}\sum_{i\in[m]}\mu^t_{k'i}a_{k'i} 
    \end{aligned}
    \end{equation}\;
    Update shadow prices via OGD step:
    \begin{align}\label{eq:ogd-step}
    \dual{k'}{t+1} = \dual{k'}{t} + \eta(\auxa{k'}{t}-\proxyx{k'}{t}) \quad \forall \ k'\geq k.
    \end{align}
        }
	\caption{Dual Gradient Descent with Proxy Assignments}
	\label{alg:multi-stages}
\end{algorithm}

The main conceptual contribution of our algorithm is the introduction of the proxy assignment decisions. The subtle dependence between the proxy decisions and the actual assignment decisions (as opposed to the explicit dependence in \Cref{alg:bad-example}) turns out to be the key to simultaneously managing multiple epochs. For clarity of exposition, we illustrate the indirect dependence between these two quantities in Figures \ref{fig:feedback-loop-2} and \ref{fig:feedback-loop-1}.  \Cref{fig:feedback-loop-2} first shows the dependence of the primal assignment decision on past proxy assignment decisions. Observe that it takes two periods for the proxy assignment decisions to be reflected in the primal assignment $x^t$. This is done through $\mu_k^t$ via \eqref{eq:proxy-step}, which depends on $a_k^{t-1}$ via \eqref{eq:ogd-step}. Since $a_k^{t-1}$ solves the joint deviation problem \eqref{eq:many-stages-aux-problem}, it depends on the shadow prices associated with {\it all} future epochs in period $t-1$, each of which is a function of the corresponding proxy assignment decision made in period $t-2$, again via \eqref{eq:ogd-step}.  Conversely, \Cref{fig:feedback-loop-1} shows how past primal assignment decisions are reflected in the proxy assignment decisions for future epochs. In this case, it similarly takes two periods for information about primal assignment to propagate to the proxy consumption trajectory. As above, the link between the two is the joint deviation problem \eqref{eq:many-stages-aux-problem}  {Finally, we note again that in Appendix \ref{apx:nonstationary} we extend our algorithm to the nonstationary setting in which the number of arrivals varies across epochs.}

\input{parts/proxy-to-actual-tikz}

\input{parts/actual-to-proxy-tikz}


\subsection{Performance Guarantee and Analysis}

We now analyze the expected regret of \Cref{alg:multi-stages}. Our main result for this section is the following.

\begin{theorem}\label{thm:regret}
For $\eta = \Theta(\sqrt{K/T})$, the gap between the expected cost incurred by \Cref{alg:multi-stages} and the offline benchmark is upper bounded by
$$\mathbb{E}[\textsc{Reg}] = O(K^{5/2}\sqrt{T}).$$
\end{theorem}

\subsubsection{Proof outline.} One can view the original target-following problem as requiring the decision-maker to solve multiple single-epoch problems, each of which has a progressively longer time horizon, with a {\it single} decision in each period. Our algorithm essentially relaxes the requirement that a single decision be played for all epochs through the proxy assignment decisions. In other words, it decouples the $K$ epochs by allowing the decision-maker to play $K$ different (though correlated) decisions in each period. {The crux of our contribution is the reduction from the original problem to $K$ easier, single-epoch problems to be solved in each period.} Given this, the analysis of our algorithm's regret bound reduces to bounding two sources of loss: (i) the loss incurred by \Cref{alg:multi-stages} if it could actually implement the proxy assignment decisions in each epoch, and (ii) the loss incurred from the fact that the decision-maker can only implement a single decision per epoch. We refer to this latter source of loss as the {\it decoupling loss}. 

It is a priori non-obvious that the decoupling loss of our algorithm is low. In particular, this would be a significant source of regret if the proxy assignment decisions in a given epoch failed to be representative of the primal assignment decisions that were made in future epochs. Recall, the proxy assignments act as our algorithm's predictions of future consumption; \Cref{fig:feedback-loop-2} illustrates that inaccurate predictions about consumption in future epochs may result in bad primal assignment decisions in the current epoch. This would then lead to large regret with respect to the true hindsight optimum. Observe moreover that the proxy assignment decisions in any given period are necessarily inaccurate, since our algorithm does not have access to the true assignment costs in future epochs; it uses the current assignment costs as a {\it proxy} for those future costs. In our analysis, however, we exploit the stationarity of our setting to show that using the type of the current arrival for the proxy assignment decision results in simulated consumption trajectories that are approximately representative of the the true consumption trajectories. We formalize these ideas below. 

\subsubsection{Proof of \cref{thm:regret}.} We first introduce the notion of cumulative {\it proxy} cost incurred by the policy $\pi$ induced by \Cref{alg:multi-stages}. In particular, the proxy costs in epoch $k \in [K]$ are the costs incurred by the {\it proxy} assignment decisions in that epoch, in addition to the deviation costs for $k'\geq k$ under those proxy decisions. To formalize this, we let $\widetilde{Z}_{ji,k}^\pi\left(\frac{(k-1)T}{K}:T\right) = \sum_{t \in \mathcal{T}_k}\sum_{k'\geq k}\tilde{x}_{k'i}^{t}\mathds{1}\{j^{t} = j\}$ be the number of proxy assignment decisions of type $j$ made to resource $i$ throughout epoch $k$. We moreover let $\widetilde{Z}_{i,k}^\pi\left(\frac{(k-1)T}{K}:\frac{k'T}{K}\right) = \sum_{t\in\mathcal{T}_k}\sum_{k''=k}^{k'}\tilde{x}_{k''i}^t$ be the  number of proxy assignment decisions to resource $i$ for epochs $k$ through $k'$ made during epoch $k$. With this notation in hand, we formalize the notion of cumulative proxy costs below.

\smallskip 

\begin{definition}\label{def:proxy-online}
Consider epoch $k \in [K]$. Given arrival sequence $\omega_k = (j^t, t\in\mathcal{T}_k)$ and current cumulative assignment $z = (z_1,\ldots,z_m) \in \mathbb{N}^m$, the \emph{cumulative proxy cost} incurred by $\pi$ in epoch $k$ is given by:
\begin{align}\label{eq:proxy-cost}
\widetilde{V}_k^{\pi}[\omega_k \mid z] := \sum_{j \in [n]}\sum_{i\in[m]}c_{ji}\widetilde{Z}_{ji,k}^{\pi}\left(\frac{(k-1)T}{K}:T\right) + \frac{T}{K}\sum_{k'\geq k}\sum_{i\in[m]} k'g_{k'i}\left(\frac{z_i + \widetilde{Z}_{i,k}^{\pi}\left(\frac{(k-1)T}{K}:\frac{k'T}{K}\right)}{k'T/K}\right)
\end{align}
\end{definition}

\smallskip

 We next define the corresponding proxy offline benchmark associated with epoch $k$, which minimizes the proxy costs from epoch $k$ onwards, given upfront information about the sequence of arrivals in epoch $k$. 

\smallskip
 
\begin{definition}\label{def:proxy-offline}
Consider epoch $k\in[K]$. Given arrival sequence $\omega_k = (j^t, t\in\mathcal{T}_k)$ and current cumulative assignment $z = (z_1,\ldots,z_m) \in \mathbb{N}^m$, the \emph{proxy offline optimum} for epoch $k$ is defined as:
\begin{align}
\widetilde{V}_k^{\textsc{off}}[\omega_k \mid z] := \min  &\sum_{j \in [n]}\sum_{i\in[m]}c_{ji} Z_{ji}\left(\frac{(k-1)T}{K}:T\right) + \frac{T}{K}\sum_{k'\geq k}\sum_{i\in[m]} k'g_{k'i}\left(\frac{z_i + Z_i\left(\frac{(k-1)T}{K}:\frac{k'T}{K}\right)}{k'T/K}\right) \notag \\
\text{s.t.} & \sum_{i\in [m]} Z_{ji}\left(\frac{(k'-1)T}{K}:\frac{k'T}{K}\right) \leq \Lambda_j\left(\frac{(k-1)T}{K}:\frac{kT}{K}\right) \quad \forall \ j \in [n], k' \geq k \notag \\
&Z_{ji}\left(\frac{(k-1)T}{K}:T\right) = 0 \quad \forall \ j \in [n], i \not\in\mathcal{S}_j\notag \\
&Z_{ji}\left(\frac{(k'-1)T}{K}:\frac{k'T}{K}\right) \in \mathbb{N} \quad \forall \ j \in [n], i \in [m], k' \geq k. \notag 
\end{align}
\end{definition}

\smallskip

Note that though the proxy offline benchmark looks similar to the offline benchmark $V^{\textsc{off}}[\omega]$, the proxy offline benchmark (i) only optimizes from epoch $k$ onwards, and (ii) assumes that the sequence of arrivals for all $k' > k$ is identical to that observed in $k$, i.e.,  that \mbox{$\Lambda_j\left(\frac{(k'-1)T}{K}: \frac{k'T}{K}\right) = \Lambda_j\left(\frac{(k-1)T}{K}: \frac{kT}{K}\right)$}. Let $\mbox{$\widetilde{\textsc{Reg}}_k[\omega_k] = \widetilde{V}^\pi_k\left[\omega_k \mid Z^\pi\left(\frac{(k-1)T}{K}\right)\right] - \widetilde{V}^{\textsc{off}}_k\left[\omega_k \mid Z^\pi\left(\frac{(k-1)T}{K}\right)\right]$}$ be the gap between our algorithm's cumulative proxy costs and the proxy offline benchmark in epoch $k$, which we refer to as the {\it proxy regret} in epoch $k$. For ease of notation, we omit the dependence of $\widetilde{V}_1^\pi[\omega_1 \mid Z^\pi(0)]$ on $Z^\pi(0)$ throughout this section.


We first formalize the relationship between the actual costs incurred by our algorithm and its cumulative proxy costs. Specifically, we reformulate the actual costs incurred as the difference between its cumulative proxy costs in all epochs, and the assignment and deviation costs that would have been incurred under non-implemented decisions in each epoch (the two last terms in the statement of the proposition).

\begin{proposition}\label{lem:alg-costs-to-proxy-costs}
\begin{align*}
V^\pi[\omega] =& \sum_{k\in[K]}\widetilde{V}_k^{\pi}\left[\omega_k \mid Z^\pi\left(\frac{(k-1)T}{K}\right)\right]-\sum_{k\in[K]}\sum_{j\in[n]}\sum_{i\in[m]}c_{ji}\widetilde{Z}_{ji,k}^{\pi}\left(\frac{kT}{K}:T\right) \\&\qquad -\sum_{k\in[K]}\frac{T}{K}\sum_{k'> k}\sum_{i\in[m]} k'g_{k'i}\left(\frac{Z_i^\pi\left(\frac{kT}{K}\right) + \widetilde{Z}_{i,k}^{\pi}\left(\frac{kT}{K}:\frac{k'T}{K}\right)}{k'T/K}\right).
\end{align*}
\end{proposition}

\begin{proof}{\it Proof of \cref{lem:alg-costs-to-proxy-costs}.}
By construction, $\widetilde{Z}_{ji,k}^\pi\left(\frac{(k-1)T}{K}:\frac{kT}{K}\right) =  Z_{ji}^\pi\left(\frac{(k-1)T}{K}:\frac{kT}{K}\right)$. Using this observation, we re-write the cumulative proxy cost in epoch $k$ as:
\begin{align*}
\widetilde{V}_k^{\pi}\left[\omega_k \mid Z^\pi\left(\frac{(k-1)T}{K}\right)\right] = &\sum_{j \in [n]}\sum_{i\in[m]}c_{ji}\sum_{k'\geq k}\widetilde{Z}_{ji,k}^{\pi}\left(\frac{(k'-1)T}{K}:\frac{k'T}{K}\right)  \\&+ \frac{T}{K}\sum_{k'\geq k}\sum_{i\in[m]} k'g_{k'i}\left(\frac{Z_i^\pi\left(\frac{(k-1)T}{K}\right) + \widetilde{Z}_{i,k}^{\pi}\left(\frac{(k-1)T}{K}:\frac{k'T}{K}\right)}{k'T/K}\right) \\
= &\left(\sum_{j \in [n]}\sum_{i\in[m]}c_{ji}{Z}_{ji}^{\pi}\left(\frac{(k-1)T}{K}:\frac{kT}{K}\right)\right) + \frac{T}{K}\sum_{i\in[m]}kg_{ki}\left(\frac{Z_i^\pi\left(kT/K\right)}{kT/K}\right) \\ &+\left(\sum_{j\in[n]}\sum_{i\in[m]}c_{ji}\widetilde{Z}_{ji,k}^{\pi}\left(\frac{kT}{K}:T\right)\right)  \\
&+ \frac{T}{K}\sum_{k'> k}\sum_{i\in[m]} k'g_{k'i}\left(\frac{Z_i^\pi\left(\frac{kT}{K}\right) + \widetilde{Z}_{i,k}^{\pi}\left(\frac{kT}{K}:\frac{k'T}{K}\right)}{k'T/K}\right).
\end{align*}
Summing over all $k$ and applying the definition of $V^\pi[\omega]$ (see \eqref{eq:alg-cost}) we obtain:
\begin{align}\label{eq:true-vs-proxy}
\sum_{k\in[K]}\widetilde{V}_k^{\pi}\left[\omega_k \mid Z^\pi\left(\frac{(k-1)T}{K}\right)\right] = &V^\pi[\omega] +\left(\sum_{k\in[K]}\sum_{j\in[n]}\sum_{i\in[m]}c_{ji}\widetilde{Z}_{ji,k}^{\pi}\left(\frac{kT}{K}:T\right)\right) \notag \\
&+\sum_{k\in[K]}\frac{T}{K}\sum_{k'> k}\sum_{i\in[m]} k'g_{k'i}\left(\frac{Z_i^\pi\left(\frac{kT}{K}\right) + \widetilde{Z}_{i,k}^{\pi}\left(\frac{kT}{K}:\frac{k'T}{K}\right)}{k'T/K}\right).
\end{align}
Re-arranging, we obtain the result.\hfill\Halmos
\end{proof}

\smallskip

\cref{lem:unimplemented-costs-to-proxy-offline} next bounds the costs incurred by the non-implemented decisions in each epoch as a function of the proxy offline optimum in the {\it subsequent} epoch, and the loss incurred by using the current arrivals as a proxy for future arrivals.
\begin{lemma}\label{lem:unimplemented-costs-to-proxy-offline}
For all $k \leq K-1$, there exists a constant $C > 0$ independent of $k$ such that
\begin{align*}
&\sum_{j \in [n]}\sum_{i\in[m]}c_{ji}\widetilde{Z}_{ji,k}^\pi\left(\frac{kT}{K}:T\right) + \frac{T}{K}\sum_{k'\geq k+1}\sum_{i\in[m]}k'g_{k'i}\left(\frac{Z_i^\pi\left(\frac{kT}{K}\right) + \widetilde{Z}_{i,k}^\pi\left(\frac{kT}{K}:\frac{k'T}{K}\right)}{k'T/K}\right)\\   &\qquad \geq \widetilde{V}^{\textsc{off}}_{k+1}\left[\omega_{k+1} \mid {Z}^\pi\left(\frac{kT}{K}\right)\right] -CK^2\sum_{j\in[n]}\bigg{|}\Lambda_j\left(\frac{kT}{K}:\frac{(k+1)T}{K}\right) - \Lambda_j\left(\frac{(k-1)T}{K}:\frac{kT}{K}\right)\bigg{|}-nCK^2.
\end{align*}
\end{lemma}

\begin{proof}{\it Proof of \cref{lem:unimplemented-costs-to-proxy-offline}.}
Fix $k \leq K-1$. We use the proxy assignments made in epoch $k$ to construct a feasible solution to $\widetilde{V}^{\textsc{off}}_{k+1}\left[\omega_{k+1} \mid {Z}^\pi\left(\frac{kT}{K}\right)\right]$. For all $k' \geq k+1$, define the solution that proxy assigns a type $j$ resource to period $i$ in epoch $k'$ in the same proportion as our algorithm did in epoch $k$, 
i.e.,
\begin{align}\label{eq:feas-sol}
Z_{ji}\left(\frac{(k'-1)T}{K}:\frac{k'T}{K}\right) = \bigg{\lfloor}{\Lambda_j}\left(\frac{kT}{K}:\frac{(k+1)T}{K}\right)\cdot\frac{\widetilde{Z}_{ji,k}^\pi\left(\frac{(k'-1)T}{K}:\frac{k'T}{K}\right)}{\Lambda_j\left(\frac{(k-1)T}{K}:\frac{kT}{K}\right)}\bigg{\rfloor}.
\end{align}
Under this solution, the proxy assignment costs are upper bounded by:
\begin{align}
& \sum_{j \in [n]}\sum_{i\in[m]}c_{ji}\sum_{k' \geq k+1} Z_{ji}\left(\frac{(k'-1)T}{K}:\frac{k'T}{K}\right) \notag \\
 &\leq \left(\sum_{j \in [n]}\sum_{i\in[m]}c_{ji} \sum_{k'\geq k+1} \Lambda_j\left(\frac{kT}{K}:\frac{(k+1)T}{K}\right) \cdot \frac{\widetilde{Z}_{ji,k}^\pi\left(\frac{(k'-1)T}{K}:\frac{k'T}{K}\right)}{\Lambda_j\left(\frac{(k-1)T}{K}:\frac{kT}{K}\right)}\right)+nmc_{\max}(K-k) \label{eq:rounding-error-asg}\\
 &=\sum_{j \in [n]}\sum_{i\in[m]}c_{ji}\widetilde{Z}_{ji,k}^\pi\left(\frac{kT}{K}:T\right) \notag \\
 &\quad+ \sum_{j \in [n]}\sum_{i\in[m]}c_{ji} \sum_{k'\geq k+1} \widetilde{Z}_{ji,k}^\pi\left(\frac{(k'-1)T}{K}:\frac{k'T}{K}\right)\left( \frac{\Lambda_j\left(\frac{kT}{K}:\frac{(k+1)T}{K}\right)}{\Lambda_j\left(\frac{(k-1)T}{K}:\frac{kT}{K}\right)}-1 \right) +nmc_{\max}(K-k)\label{eq:proxy-asg-under-feas} \\
 &\leq \sum_{j \in [n]}\sum_{i\in[m]}c_{ji}\widetilde{Z}_{ji,k}^\pi\left(\frac{kT}{K}:T\right) \notag \\ &\quad + Kmc_{\max}\sum_{j\in[n]}\bigg{|}\Lambda_j\left(\frac{kT}{K}:\frac{(k+1)T}{K}\right) - \Lambda_j\left(\frac{(k-1)T}{K}:\frac{kT}{K}\right)\bigg{|} +nmc_{\max}(K-k),\label{eq:proxy-asg-under-feas-2}
\end{align}
where \eqref{eq:rounding-error-asg} bounds the loss from relaxing the integrality of \eqref{eq:feas-sol}, and \eqref{eq:proxy-asg-under-feas} adds and subtracts $\mbox{$\widetilde{Z}_{ji,k}^\pi\left(\frac{kT}{K}:T\right)=\sum_{k'\geq k+1}\widetilde{Z}_{ji,k}^\pi\left(\frac{(k'-1)T}{K}:\frac{k'T}{K}\right)$}$. Finally, \eqref{eq:proxy-asg-under-feas-2} follows from factoring $\Lambda_j\left(\frac{(k-1)T}{K}:\frac{kT}{K}\right)$ in the denominator, bounding the difference by its absolute value and using that $\frac{\widetilde{Z}^\pi_{ji,k}\left(\frac{(k'-1)T}{K}:\frac{k'T}{K}\right)}{\Lambda_j\left(\frac{(k-1)T}{K}:\frac{kT}{K}\right)} \leq 1$ for $\mbox{$k' \geq k+1$}$ by construction. 

We apply similar arguments to the proxy deviation costs under this solution. In particular, since $g_{k'i}$ is Lipschitz, for all $k'\geq k+1$:
\begin{align}\label{eq:proxy-dev-under-feas}
\frac{k'T}{K}\cdot g_{k'i}\left(\frac{Z_i^\pi\left(\frac{kT}{K}\right) + Z_i\left(\frac{kT}{K}:\frac{k'T}{K}\right)}{k'T/K}\right) &\leq \frac{k'T}{K}\cdot g_{k'i}\left(\frac{Z_i^\pi\left(\frac{kT}{K}\right) + \widetilde{Z}_{i,k}^\pi\left(\frac{kT}{K}:\frac{k'T}{K}\right)}{k'T/K}\right) \notag \\ &\quad+ L\bigg{|}\sum_{j\in[n]}\Lambda_j\left(\frac{kT}{K}:\frac{(k+1)T}{K}\right) \cdot \frac{\sum_{k''= k+1}^{k'}\widetilde{Z}_{ji,k}^\pi\left(\frac{(k''-1)T}{K}:\frac{k''T}{K}\right)}{\Lambda_j\left(\frac{(k-1)T}{K}:\frac{kT}{K}\right)} \notag \\&\qquad\qquad -\sum_{j\in[n]} \sum_{k''= k+1}^{k'}\widetilde{Z}_{ji,k}^\pi\left(\frac{(k''-1)T}{K}:\frac{k''T}{K}\right)\bigg{|} + Ln(k'-k) \notag \\
&= \frac{k'T}{K}\cdot g_{k'i}\left(\frac{Z_i^\pi\left(\frac{kT}{K}\right) + \widetilde{Z}_{i,k}^\pi\left(\frac{kT}{K}:\frac{k'T}{K}\right)}{k'T/K}\right) \notag \\ &\quad+ L(k'-k)\sum_{j\in[n]}\bigg{|}\Lambda_j\left(\frac{kT}{K}:\frac{(k+1)T}{K}\right)-\Lambda_j\left(\frac{(k-1)T}{K}:\frac{kT}{K}\right)\bigg{|} \notag \\
&\quad +Ln(k'-k),
\end{align}
where the extra additive loss of $Ln(k'-k)$ is due to rounding \eqref{eq:feas-sol}.

Summing \eqref{eq:proxy-dev-under-feas} over all $i$ and $k'\geq k+1$, and adding to \eqref{eq:proxy-asg-under-feas-2}, we obtain:
\begin{align}\label{eq:vk-off-bound}
&\widetilde{V}^{\textsc{off}}_{k+1}\left[\omega_{k+1} \mid {Z}^\pi\left(\frac{kT}{K}\right)\right] \notag \\ &\leq \left(\sum_{j \in [n]}\sum_{i\in[m]}c_{ji}\sum_{k' \geq k+1} Z_{ji}\left(\frac{(k'-1)T}{K}:\frac{k'T}{K}\right)\right) + \left(\sum_{k'\geq k+1}\sum_{i\in[m]}\frac{k'T}{K}\cdot g_{k'i}\left(\frac{Z_i^\pi\left(\frac{kT}{K}\right) + Z_i\left(\frac{kT}{K}:\frac{k'T}{K}\right)}{k'T/K}\right)\right)\notag \\
&\leq  \sum_{j \in [n]}\sum_{i\in[m]}c_{ji}\widetilde{Z}_{ji,k}^\pi\left(\frac{kT}{K}:T\right) + \frac{T}{K}\sum_{k'\geq k+1}\sum_{i\in[m]}k'g_{k'i}\left(\frac{Z_i^\pi\left(\frac{kT}{K}\right) + \widetilde{Z}_{i,k}^\pi\left(\frac{kT}{K}:\frac{k'T}{K}\right)}{k'T/K}\right)\notag \\ &\quad +\left(Kmc_{\max}+L K^2m\right)\sum_{j\in[n]}\bigg{|}\Lambda_j\left(\frac{kT}{K}:\frac{(k+1)T}{K}\right) - \Lambda_j\left(\frac{(k-1)T}{K}:\frac{kT}{K}\right)\bigg{|} \notag \\ 
&\quad + nmc_{\max} K + LnmK^2,
\end{align}
where the first inequality follows from the fact that $Z$ is a feasible solution to $\widetilde{V}^{\textsc{off}}_{k+1}\left[\omega_{k+1} \mid {Z}^\pi\left(\frac{kT}{K}\right)\right]$. Letting $C = Lm+mc_{\max}$, we obtain the result. \hfill\Halmos
\end{proof}

\smallskip

\cref{lem:unimplemented-costs-to-proxy-offline} implies that the proxy assignment decisions made in each epoch are approximately optimal for the proxy offline optimum in the next epoch, as long as the composition of arrivals throughout these two epochs does not differ by too much. We show that this holds in expectation below, and use this to bound the total cost of our algorithm as a function of the proxy regret across all epochs.

\begin{lemma}\label{lem:alg-costs-vs-proxy-regret}
$\mathbb{E}\left[V^\pi[\omega]\right] \leq \mathbb{E}\left[\widetilde{V}^{\textsc{off}}_1[\omega_1]\right] + \sum_{k\in[K]}\mathbb{E}\left[\widetilde{\textsc{Reg}}_k[\omega_k]\right] + O\left(K^{5/2}\sqrt{T}\right).
$
\end{lemma}

\begin{proof}{\it Proof of \cref{lem:alg-costs-vs-proxy-regret}.}
Recall, by \cref{lem:alg-costs-to-proxy-costs}:
\begin{align}
V^{\pi}[\omega] =& \sum_{k\in[K]}\widetilde{V}_k^{\pi}\left[\omega_k \mid Z^\pi\left(\frac{(k-1)T}{K}\right)\right]\notag \\& \quad -\left[\sum_{k\in[K]}\sum_{j\in[n]}\sum_{i\in[m]}c_{ji}\widetilde{Z}_{ji,k}^{\pi}\left(\frac{kT}{K}:T\right)
+\sum_{k\in[K]}\frac{T}{K}\sum_{k'> k}\sum_{i\in[m]} k'g_{k'i}\left(\frac{Z_i^\pi\left(\frac{kT}{K}\right) + \widetilde{Z}_{i,k}^{\pi}\left(\frac{kT}{K}:\frac{k'T}{K}\right)}{k'T/K}\right) \right]. \notag
\end{align}
Applying \cref{lem:unimplemented-costs-to-proxy-offline} to the above, we have:
\begin{align}
&V^{\pi}[\omega]\leq \sum_{k\in[K]}\widetilde{V}_k^{\pi}\left[\omega_k \mid Z^\pi\left(\frac{(k-1)T}{K}\right)\right] -\left(\sum_{k \leq K-1} \widetilde{V}^{\textsc{off}}_{k+1}\left[\omega_{k+1} \mid {Z}^\pi\left(\frac{kT}{K}\right)\right]\right) \\ &\hspace{3cm} + CK^2\sum_{k\leq K-1}\bigg{|}\Lambda_j\left(\frac{kT}{K}:\frac{(k+1)T}{K}\right) - \Lambda_j\left(\frac{(k-1)T}{K}:\frac{kT}{K}\right)\bigg{|} +nCK^2 \notag \\
\implies &\underbrace{V^{\pi}[\omega]- \widetilde{V}_1^{\pi}\left[\omega_1 \right] - \sum_{k \geq 2} \widetilde{\textsc{Reg}}_k[\omega_k]}_{(I)} \leq  CK^2\sum_{k\leq K-1}\sum_{j\in[n]}\bigg{|}\Lambda_j\left(\frac{kT}{K}:\frac{(k+1)T}{K}\right) - \Lambda_j\left(\frac{(k-1)T}{K}:\frac{kT}{K}\right)\bigg{|} \notag \\ &\hspace{7cm}+ nCK^2,\label{eq:ub-by-abs-diff}
\end{align}
where  \eqref{eq:ub-by-abs-diff} follows from the definition of $\widetilde{\textsc{Reg}}_k[\omega_k]$ as the difference between the proxy online costs and the proxy offline benchmark for epoch $k$. Taking expectations on both sides, we have:
\begin{align*}
\mathbb{E}[(I)] &\leq CK^2\sum_{k\leq K-1}\sum_{j\in[n]}\mathbb{E}\left[\bigg{|}\Lambda_j\left(\frac{kT}{K}:\frac{(k+1)T}{K}\right) - \Lambda_j\left(\frac{(k-1)T}{K}:\frac{kT}{K}\right)\bigg{|}\right] + nCK^2.
\end{align*}
Let $Y_{jk} = \Lambda_j\left(\frac{(k-1)T}{K}:\frac{kT}{K}\right)$, and observe that $\mathbb{E}[Y_{jk}]=\mathbb{E}[Y_{j,k+1}]$ since $Y_{jk}$ and $Y_{j,k+1}$ are i.i.d. Then,
\begin{align}
\mathbb{E}[(I)] &\leq CK^2 \left(\sum_{k\leq K-1}\sum_{j\in[n]}\mathbb{E}\left[\big{|}Y_{jk}-\mathbb{E}[Y_{jk}]\big{|}\right] + \mathbb{E}\left[\big{|}Y_{j,k+1}-\mathbb{E}[Y_{j,k+1}]\big{|}\right] \right) + nCK^2\notag \\
&\leq  2CK^2 \sum_{k\leq K-1}\sum_{j\in[n]}\sqrt{\mathbb{E}\left[\left(Y_{jk}-\mathbb{E}[Y_{jk}]\right)^2\right]} + nCK^2\label{eq:jensens-step} \\
&=2CK^2 \sum_{k\leq K-1}\sum_{j\in[n]}\sqrt{\text{Var}[Y_{jk}]} + nCK^2,\label{eq:variance-step}
\end{align}
where \eqref{eq:jensens-step} uses the fact that $Y_{jk}$ and $Y_{j,k+1}$ are i.i.d., along with Jensen's inequality. Recalling that arrivals are drawn from a multinomial distribution with constant $p_j$ (with respect to $T$), we have that $\mbox{$\text{Var}[Y_{jk}] = O\left(\frac{T}{K}\right)$}$. Plugging this back into \eqref{eq:variance-step} we obtain our final bound, i.e. 
\begin{align*}
\mathbb{E}\left[V^\pi[\omega]\right] &\leq  \mathbb{E}\left[\widetilde{V}_1^{\pi}\left[\omega_1\right]+ \sum_{k \geq 2} \widetilde{\textsc{Reg}}_k[\omega_k]\right] +O\left(K^{5/2}\sqrt{T}\right) \\
&= \sum_{k\in[K]}\mathbb{E}\left[\widetilde{\textsc{Reg}}_k[\omega_k]\right] + \mathbb{E}\left[\widetilde{V}^{\textsc{off}}_1[\omega_1]\right] + O\left(K^{5/2}\sqrt{T}\right),
\end{align*}
where the final equality follows from adding and subtracting $\mathbb{E}\left[\widetilde{V}^{\textsc{off}}_1[\omega_1]\right]$.
\hfill\Halmos
\end{proof}

Hence, it suffices to bound the proxy online costs incurred in the first epoch and the proxy regret in epochs 2 through $K$. \cref{lem:proxy-cost-epoch-1} uses similar arguments as those above to show that $\mathbb{E}[\widetilde{V}_1^\pi[\omega]]$ is within $O(\sqrt{T})$ of the offline optimum.

\smallskip

\begin{lemma}\label{lem:proxy-cost-epoch-1}
$\mathbb{E}\left[\widetilde{V}_1^\textsc{off}[\omega_1]\right] \leq  \mathbb{E}\left[V^{\textsc{off}}[\omega]\right] + O(K^{3/2}\sqrt{T})$.
\end{lemma}

\begin{proof}{\it Proof of \cref{lem:proxy-cost-epoch-1}.}
As in the proof of \cref{lem:alg-costs-vs-proxy-regret}, we construct a feasible solution to $\widetilde{V}^\textsc{off}[\omega_1]$ using an optimal solution to $V^{\textsc{off}}[\omega]$. For any such optimal solution, let $Z_{ji}^{\textsc{off}}\left(\frac{(k-1)T}{K}:\frac{kT}{K}\right)$ denote the number of type $j$ arrivals assigned to resource $i$ in epoch $k$. Then, for all $k \in [K]$, define
\begin{align*}
Z_{ji,1}\left(\frac{(k-1)T}{K}:\frac{kT}{K}\right) = \bigg{\lfloor}\Lambda_j\left(T/K\right)\cdot\frac{Z_{ji}^{\textsc{off}}\left(\frac{(k-1)T}{K}:\frac{kT}{K}\right)}{\Lambda_j\left(\frac{(k-1)T}{K}:\frac{kT}{K}\right)}\bigg{\rfloor},
\end{align*}
the solution that proxy assigns the type $j$ arrival to resource $i$ for epoch $k$ in the same proportion as the offline solution. We can similarly bound the proxy assignment costs under this solution:
\begin{align}\label{eq:proxy1-vs-opt-asg}
\sum_{j\in[n]}\sum_{i\in[m]}c_{ji}Z_{ji,1}\left(T\right) \leq \sum_{j\in[n]}\sum_{i\in[m]}c_{ji}Z_{ji}^{\textsc{off}}\left(T\right) + mc_{\max}\sum_{j\in[n]}\sum_{k\in[K]}\bigg{|}\Lambda_j\left(\frac{T}{K}\right)-\Lambda_j\left(\frac{(k-1)T}{K}:\frac{kT}{K}\right)\bigg{|}
+ CK,
\end{align}
for some constant $C > 0$, where the additional $O(K)$ loss is due to rounding the feasible solution, as in the proof of \Cref{lem:alg-costs-vs-proxy-regret}.

For the proxy deviation costs, we similarly have:
\begin{align}\label{eq:proxy1-vs-opt-dev}
\frac{T}{K}\sum_{k\in[K]}\sum_{i\in[m]}kg_{ki}\left(\frac{Z_i\left(kT/K\right)}{kT/K}\right) &\leq \frac{T}{K}\sum_{k\in[K]}\sum_{i\in[m]}kg_{ki}\left(\frac{Z_i^{\textsc{off}}\left(kT/K\right)}{kT/K}\right) \notag \\ &\qquad +Lm\sum_{j\in[n]}\sum_{k\in[K]}\sum_{k'\leq k}\bigg{|}\Lambda_j\left(\frac{T}{K}\right)-\Lambda_j\left(\frac{(k'-1)T}{K}:\frac{k'T}{K}\right)\bigg{|} + CnK^2.
\end{align}
Putting \eqref{eq:proxy1-vs-opt-asg} and \eqref{eq:proxy1-vs-opt-dev} together, taking expectations on both sides, and applying Jensen's inequality, as in \cref{lem:alg-costs-vs-proxy-regret}, we obtain:
\begin{align*}
\mathbb{E}\left[\widetilde{V}_1^{\textsc{off}}[\omega_1]\right] \leq \mathbb{E}\left[V^{\textsc{off}}[\omega]\right] + O(K^{3/2}\sqrt{T}).
\end{align*}
\hfill\Halmos
\end{proof}

\smallskip

We now bound $\mathbb{E}\left[\widetilde{\textsc{Reg}}_k[\omega_k]\right]$, for all $k\in[K]$.
\begin{lemma}\label{lem:balseiro-bound}
For $\eta = \Theta(\sqrt{K/T})$, $\mathbb{E}\left[\widetilde{\textsc{Reg}}_k[\omega_k]\right] = O(\sqrt{T/K})$.
\end{lemma}

\begin{proof}{\it Proof of \cref{lem:balseiro-bound}.}
This follows by observing that, in each epoch $k$, \Cref{alg:multi-stages} runs the regularized dual mirror descent algorithm of \citet{balseiro2021regularized} with respect to the proxy offline benchmark $\widetilde{V}_k^{\textsc{off}}\left[\omega_k \ \mid Z^\pi\left(\frac{(k-1)T}{K}\right)\right]$.\footnote{The algorithm in \citet{balseiro2021regularized} uses a generic  mirror descent update for the dual variable. We inherit their regret bound by using the well-known fact that online subgradient descent is a special case of online mirror descent \citep{hazan2016introduction}, and that the $O(\sqrt{T})$ scaling that they obtain for their single-epoch problem is invariant to the specific form of mirror descent used.} 

To see the reduction from our proxy problem to the \citet{balseiro2021regularized} setup, consider an online resource allocation problem over a horizon of length $T/K$, with $(K-k+1)m$ resources, indexed by $(k',i)$, for $k'\geq k$, $i\in[m]$, with budget $b_{k'i} = 1$. For each arrival, the decision-maker makes an assignment decision $x^t = (\tilde{x}^t_{k'i}, k' \geq k, i \in[m])$, where $(\tilde{x}^t_{k'i}, i \in [m]) \in \mathcal{X}^t$ for all $k' \geq k$. For this decision $x^t$, it incurs a bounded linear cost $f^t(x^t)$, defined by:
\[f^t(x^t) =  \sum_{k'\geq k}\sum_{i\in[m]}\tilde{x}_{k'i}^t\left(\sum_{j\in[n]}\mathds{1}\{j=j^t\}c_{ji}\right).\]
There is a single regularizer on the time-average consumption of all $(K-k+1)m$ resources:
\[r(a) = \sum_{k'\geq k}\sum_{i\in[m]}k'g_{k'i}\left(\frac{z_i}{k'T/K} + \frac{\sum_{k''=k}a_{k'',i}}{k'}\right).\]
Note that in this case, since $a \in [0,1]^{(K+1-k)\times m}$, the set of dual variables for which the conjugate of $r$ is bounded is $\mathcal{D} = \mathbb{R}$, and it is easy to see that Assumption 2 in \citet{balseiro2021regularized} holds.\footnote{The idealized consumption problem in the dual mirror descent algorithm in \citet{balseiro2021regularized} optimizes over all $a \leq b$; in contrast, we also impose a nonnegativity constraint on our auxiliary variables. All of the same arguments go through  assuming nonnegativity.} Then, applying a stepsize $\eta = \Theta\left((T/K)^{-1/2}\right)$ to Theorem 1 in \citet{balseiro2021regularized}, we obtain the claim.
\hfill\Halmos
\end{proof}

\smallskip

We conclude by putting these bounds together to obtain our main result.
\begin{proof}{\it Proof of \cref{thm:regret}.}
We have:
\begin{align*}
\mathbb{E}\left[V^\pi[\omega]\right] &\leq \mathbb{E}\left[\widetilde{V}^{\textsc{off}}_1[\omega_1]\right] + \sum_{k\in[K]}\mathbb{E}\left[\widetilde{\textsc{Reg}}_k[\omega_k]\right] + O\left(K^{5/2}\sqrt{T}\right)  \qquad \qquad  &\text{(\cref{lem:alg-costs-vs-proxy-regret})}\\
&\leq \mathbb{E}\left[V^{\textsc{off}}[\omega]\right] + \sum_{k\in[K]}\mathbb{E}\left[\widetilde{\textsc{Reg}}_k[\omega_k]\right] + O\left(K^{5/2}\sqrt{T}\right) + O\left(K^{3/2}\sqrt{T}\right)
 &\text{(\cref{lem:proxy-cost-epoch-1})}\\
&\leq  \mathbb{E}\left[V^{\textsc{off}}[\omega]\right] + O\left(K^{5/2}\sqrt{T}\right) + O\left(2K^{3/2}\sqrt{T}\right)
 &\text{(\cref{lem:balseiro-bound})} \\
 \implies \mathbb{E}[\textsc{Reg}] &= O(K^{5/2}\sqrt{T}).&
\end{align*}
\hfill\Halmos
\end{proof}

{Lastly, we note that an important practical benefit of our algorithm is that it is {\it simple} and {\it easy to implement}, as compared to alternative methods in the dynamic resource allocation literature which require frequently resolving large convex programs that use the previous history of observations, and therefore have large data requirements when $T$ is large. In contrast, to achieve the optimal rate of regret in this setting we only need to solve one small convex problem (\ref{eq:many-stages-aux-problem}) and then use the extremely naive ``prediction'' of arrivals in future epochs that simply copies the current arrival $K$ times.}

%% file: parts/proxy-to-actual-tikz.tex
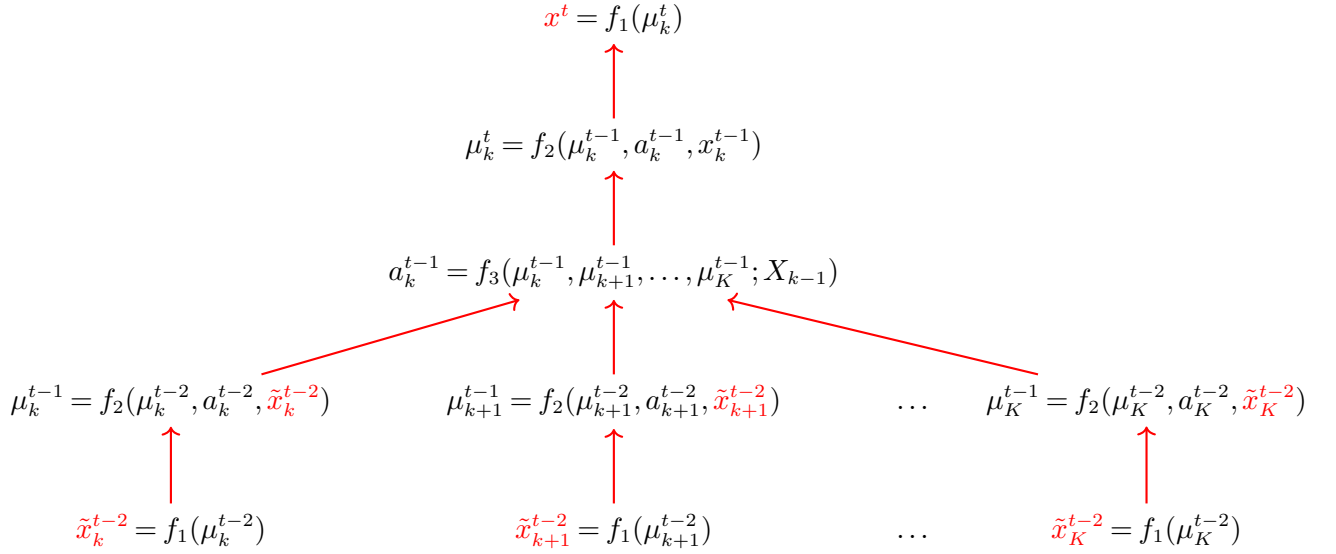
\begin{figure}
\centering
\begin{tikzpicture}[auto]
\matrix[matrix of nodes, row sep=1cm, column sep=0.5cm] (m) {
    & |(x_t)| $\textcolor{red}{x^t} = f_1(\mu_k^t)$ & \\
    & |(mu_t_k)| $\mu_k^t = f_2(\mu_k^{t-1}, a_k^{t-1}, x_k^{t-1})$ & \\
    & |(a_t_1)| $a_k^{t-1} = f_3(\mu_k^{t-1}, \mu_{k+1}^{t-1}, \dots, \mu_K^{t-1};X_{k-1})$ & \\
    |(mu_t_1_k)| $\mu_k^{t-1} = f_2(\mu_k^{t-2}, a_k^{t-2}, \textcolor{red}{\tilde{x}_k^{t-2}})$ & |(mu_t_1_k1)| $\mu_{k+1}^{t-1} = f_2(\mu_{k+1}^{t-2}, a_{k+1}^{t-2}, \textcolor{red}{\tilde{x}_{k+1}^{t-2}})$ & |(dots)| $\ldots$ & |(mu_t_2_k)| $\mu_K^{t-1} = f_2(\mu_K^{t-2}, a_K^{t-2}, \textcolor{red}{\tilde{x}_K^{t-2}})$
    \\
    |(prox_k_t_2)| $\textcolor{red}{\tilde{x}_k^{t-2}} = f_1(\mu_k^{t-2})$ & |(prox_k1_t_2)| $\textcolor{red}{\tilde{x}_{k+1}^{t-2}} = f_1(\mu_{k+1}^{t-2})$ & |(dots)| $\ldots$ & |(prox_K_t_2)| $\textcolor{red}{\tilde{x}_K^{t-2}} = f_1(\mu_K^{t-2})$ \\
};

\draw[red, thick, ->] (mu_t_k) -- (x_t);
\draw[red, thick, ->] (a_t_1) -- (mu_t_k);
\draw[red, thick, ->] (mu_t_1_k) -- (a_t_1);
\draw[red, thick, ->] (mu_t_1_k1) -- (a_t_1);
\draw[red, thick, ->] (mu_t_2_k) -- (a_t_1);
\draw[red, thick, ->] (prox_k_t_2) -- (mu_t_1_k);
\draw[red, thick, ->] (prox_k1_t_2) -- (mu_t_1_k1);
\draw[red, thick, ->] (prox_K_t_2) -- (mu_t_2_k);
\end{tikzpicture}
\caption{Dependence of primal assignment decision $x^t$ on past proxy assignment decisions, $\tilde{x}_k^{t-2}, \tilde{x}_{k+1}^{t-2},\ldots, \tilde{x}_K^{t-2}$. Here, $k$ is used to denote the epoch containing period $t$. The functions $f_1,f_2,f_3$ respectively correspond to subroutines \eqref{eq:proxy-step}, \eqref{eq:ogd-step}, and \eqref{eq:many-stages-aux-problem} in \Cref{alg:multi-stages}, and $X_{k-1}=\sum_{t\leq (k-1)T/K}x^t$ represents the cumulative consumption of resources up until the end of epoch $k-1$. The lowest level of the diagram shows the proxy assignment decisions computed in period $t-2$, for each epoch $k' \geq k$. In the next level, these are used to update the dual variables for each respective epoch. {\it All} dual variables are then used to compute $a_k^{t-1}$, which then feeds into the dual variable update for epoch $k$. This dual variable is used to compute the primal assignment decision in period $t$ at the topmost level.}\label{fig:feedback-loop-2}
\end{figure}

%% file: parts/actual-to-proxy-tikz.tex
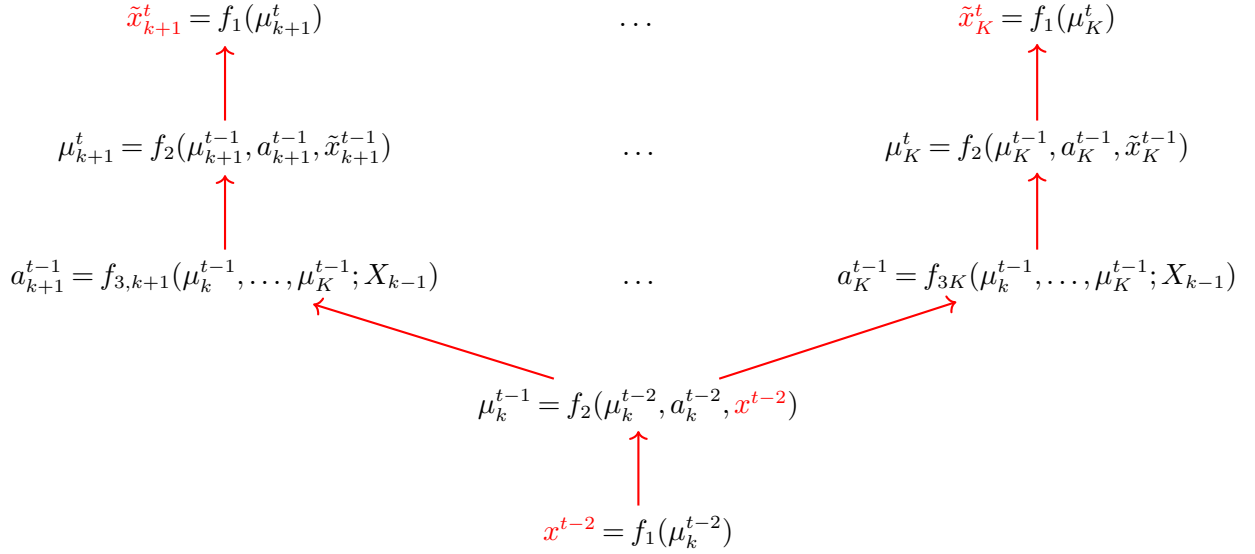
\begin{figure}
\centering
\begin{tikzpicture}[auto]
\matrix[matrix of nodes, row sep=1cm, column sep=0.25cm] (m) {
  |(x_k1_t)| $\textcolor{red}{\tilde{x}_{k+1}^t} = f_1(\mu_{k+1}^t)$ & \ldots & |(x_K_t)| $\textcolor{red}{\tilde{x}_K^t} = f_1(\mu_K^t)$ \\
  |(mu_t_k1)| $\mu_{k+1}^t = f_2(\mu_{k+1}^{t-1}, a_{k+1}^{t-1}, \tilde{x}_{k+1}^{t-1})$ & $\ldots$ &|(mu_t_K)| $\mu_{K}^t = f_2(\mu_K^{t-1}, a_K^{t-1}, \tilde{x}_K^{t-1})$\\
  |(a_t_k1)| $a_{k+1}^{t-1} = f_{3,k+1}(\mu_k^{t-1},\ldots,\mu_K^{t-1};X_{k-1})$ & $\ldots$ &|(a_t_K)| $a_{K}^{t-1} = f_{3K}(\mu_k^{t-1}, \ldots,\mu_K^{t-1};X_{k-1})$\\
     & |(mu_t1_k)| $\mu_k^{t-1} = f_2(\mu_k^{t-2}, a_k^{t-2},\textcolor{red}{{x}^{t-2}})$ & \\
     & |(x_t2_k)| $\textcolor{red}{x^{t-2}} = f_1(\mu_k^{t-2})$ &\\
};

\draw[red, thick, ->] (mu_t_k1) -- (x_k1_t);
\draw[red, thick, ->] (mu_t_K) -- (x_K_t);
\draw[red, thick, ->] (a_t_k1) -- (mu_t_k1);
\draw[red, thick, ->] (a_t_K) -- (mu_t_K);
\draw[red, thick, ->] (mu_t1_k) -- (a_t_k1);
\draw[red, thick, ->] (mu_t1_k) -- (a_t_K);
\draw[red, thick, ->] (x_t2_k) -- (mu_t1_k);
\end{tikzpicture}
\caption{Dependence of proxy assignment decisions in future epochs, $\tilde{x}_{k+1}^t,\ldots,\tilde{x}_K^t$ on past primal assignment decisions in period $t$, where $k$ is used to denote the epoch containing period $t$. The functions $f_1,f_2,f_{3k'}$, $k' \geq k$, respectively correspond to subroutines \eqref{eq:proxy-step}, \eqref{eq:ogd-step}, and \eqref{eq:many-stages-aux-problem} in \Cref{alg:multi-stages}, and $X_{k-1}=\sum_{t\leq (k-1)T/K}x^t$ represents the cumulative consumption of resources up until the end of epoch $k-1$. The lowest level of the diagram shows the primal assignment decision in period $t-2$, which depends only on the dual variable associated with period $k$. This decision is then used to update $\mu_k^{t-1}$ in the next level. Afterwards, $\mu_k^{t-1}$ is used to compute the idealized average consumptions for all $k' \geq k$, according to \eqref{eq:many-stages-aux-problem}. Each of these is then used to update its respective dual variable in period $t$. Finally, in the topmost level each dual variable is used to compute the proxy assignment for epochs $k+1$ to $K$, as in \eqref{eq:proxy-step}.}\label{fig:feedback-loop-1}
\end{figure}

%% file: parts/experiments.tex
\section{Computational Experiments}\label{sec:exper}

In this section we perform extensive computational experiments demonstrating the efficacy of our algorithm on both synthetic and real-world datasets. 

\subsection{Results on Synthetic Dataset}\label{sec:synth}

We first study the numerical performance of our algorithm on a synthetic dataset.

\smallskip 

\textbf{Experimental setup.} We consider a setting in which there are $m = 3$ resources (motivated by the real-world dataset we consider in \cref{sec:real}), and $n = 3$ types of arrivals. For each type, we fix the assignment costs at the beginning of the horizon as follows: the assignment cost for the outside option is normalized to 0 for all types; the cost of assigning type $j$ to resource $j$ (its preferred resource) is sampled from the uniform distribution over ${[-1,-2/3]}$; for all resources $i \neq j$ (the non-preferred resources of type $j$), the assignment cost is sampled independently from the uniform distribution over ${[-2/3,0]}$. We similarly generate the arrival probabilities $p$ by sampling $n$ variables $\pi_1,\ldots,\pi_n$ independently from the uniform distribution over ${[0,1]}$, and letting $p_j = \frac{\pi_j}{\sum_{j'\in[n]}\pi_{j'}}$. 

There are $K = 3$ epochs: the first and last epochs are ``off-peak,'' and defined by a network-level base target sampled from a uniform distribution over ${[0.15,0.6]}$, similarly motivated from the historical range of targets in the real-world dataset considered in \cref{sec:real}. Letting $\varrho$ denote the network-level base target, let $\rho_{1i} = \rho_{3i} = \frac{\varrho}{m}$ for each resource $i$. The middle epoch represents a ``peak'' period during which the base network-level target $\varrho$ increases (resp. decreases) by a multiplicative factor (referred to as the {\it impulse level}) $\gamma \in \{0.5,1,1.5,2\}$. As is commonly done in practice, we penalize deviations from the targets using absolute deviation, i.e., $g_{ki}(a) = \delta|a-\rho_{ki}|$ for each epoch $k$ and resource $i$, for $\delta \in \{0.001,0.01,0.1,1\}$. We henceforth refer to $\delta$ as the {\it deviation scalar}. In the real-world dataset we consider below, resources have different targets. Though we assume resource targets are identical here, this setup still models imbalances across resources due to the fact that each type has a preferred resource, and type probabilities are non-uniform. 

\smallskip 

\textbf{Policies.} We consider three policies: \Cref{alg:multi-stages} (referred to as ``alg'' in all plots), as well as two myopic policies that are variants of the primal-dual algorithm in \citet{balseiro2021regularized} that was designed for the single-epoch problem. These latter two policies are myopic in that they treat each epoch as an independent decision-making horizon of length $T/K$, and seek only to minimize the cumulative assignment and deviation costs {\it within} the epoch, without any eye toward future targets. 

Recall, the targets associated with each epoch are with respect to the {\it cumulative} assignment of arrivals to each resource, i.e., they take into account both the assignments in the associated epoch, as well as those in all previous epochs in addition to those in the epoch of interest. Thus, to fully specify a myopic policy we must translate the cumulative targets into {\it epoch-specific} targets that represent the ideal average consumption with respect to only the arrivals in that epoch. This requires the policy to make an assumption about past assignments. We present two practical assumptions, each of which yields a different policy.

Let $\tilde{\rho}_{ki}$ denote the epoch-specific target for epoch $k$ and resource $i$ computed by any myopic policy. The first policy, which we call \textsc{Myopic-Epoch} (\textsc{ME}), computes $\tilde{\rho}_{ki}$ assuming that the running average consumption up until the end epoch of $k-1$ was exactly the target $\rho_{(k-1),i}$. Hence, to achieve a running average target of ${\rho}_{ki}$ in period $k$, the epoch-specific target must satisfy:
\begin{align}\label{eq:myopic}
\rho_{ki} = \frac{\tilde{\rho}_{ki} + (k-1)\rho_{(k-1),i}}{k} \iff \tilde{\rho}_{ki} = k\rho_{ki} - (k-1)\rho_{(k-1),i}.
\end{align}
While the assumption that past targets were adhered to is not lossy for large values of $\delta$, for smaller values of $\delta$ it is likely unrealistic. The second policy we implement, which we call \textsc{Smart Myopic-Epoch} (\textsc{Smart-ME}), is more adaptive, and computes epoch-specific targets given the algorithm's {\it actual} consumption up until the end of epoch $k-1$. In this case, the epoch-specific target must satisfy:
\begin{align}\label{eq:smart-myopic}
\rho_{ki} = \frac{\tilde{\rho}_{ki}}{k} + \frac{\sum_{t \leq (k-1)T/K} x^t_i}{kT/K}
\iff \tilde{\rho}_{ki} = k\rho_{ki} - \frac{\sum_{t\leq (k-1)T/K}x_i^t}{T/K},
\end{align}
where $x^t_i$ represents the algorithm's primal assignment decision in period $t$ to resource $i$.

We refer the reader to Appendix \ref{apx:exper} for complete descriptions of the two myopic policies. For all three policies we use stepsize $\eta = \sqrt{K/T}$. We average the performance of each policy over 100 sample paths, and compare their respective performance to the offline optimum, denoted by \textsc{Off}.

\smallskip 

\textbf{Results.} We first study the dependence of the three policies' regret on the length of the horizon $T$. \Cref{fig:regret} shows that the \textsc{ME} policy has the worst performance of the three, clearly incurring linear regret. The simple adaptive fix that allows it to use past resource consumption --- yielding the \textsc{Smart-ME} policy --- results in significant improvements. Still, our algorithm significantly outperforms both myopic policies as $T$ grows large, with sublinear regret as established by \cref{thm:regret}.

\begin{figure}
  \centering
  \includegraphics[width=0.5\linewidth]{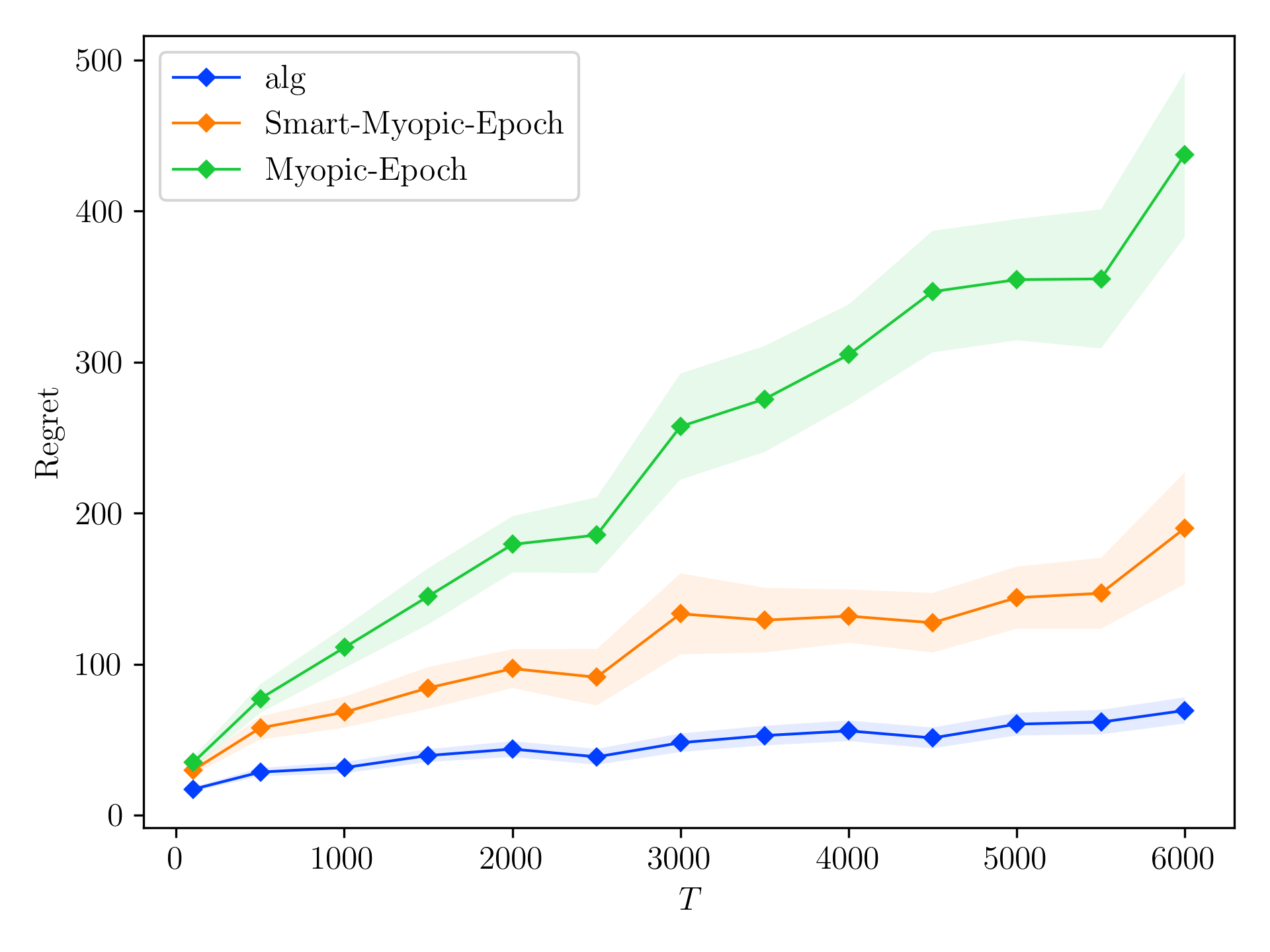}
  \caption{Plot of regret versus $T$ for $\delta = 1$, $\gamma = 2$, and $T \in 3\cdot\left\{\lceil\frac{10^2}{3}\rceil,\lceil\frac{0.5\cdot 10^3}{3}\rceil, \lceil\frac{10^3}{3}\rceil,\lceil\frac{1.5\cdot 10^3}{3}\rceil,\ldots,\lceil\frac{6\cdot 10^3}{3}\rceil\right\}$}
  \label{fig:regret}
\end{figure}

To shed light on the strong performance of \Cref{alg:multi-stages} relative to myopic target-following, we next numerically investigate the impact of the deviation scalar $\delta$ and the impulse level $\gamma$ on the two policies' behavior. For all experiments, we let $T = 501$ and focus solely on \textsc{Smart-ME}.

\Cref{tab:synth-res-total-cost} shows the relative regret of \Cref{alg:multi-stages} and \textsc{Smart-ME}, i.e., $\frac{\mathbb{E}[V^\pi[\omega]-V^{\textsc{off}}[\omega]]}{\mathbb{E}[V^{\textsc{off}}[\omega]]}$, for various values of $\delta$. (We alternatively refer to this as the relative cost difference of these policies.) As more weight is placed on target-following, the total cost relative to the offline optimum increases for both policies. While they perform similarly for small values of $\delta$, for $\delta \geq 0.5$ the difference in performance becomes stark, with \Cref{alg:multi-stages} exceeding the offline optimum costs by just over 6\%, as compared to over 26\% for \textsc{Smart-ME}. This difference only grows as $\delta$ increases, with \textsc{Smart-ME} incurring 2.5 times the cost of the hindsight optimum, as compared to a 70\% relative difference for \Cref{alg:multi-stages}.

\begin{table}
\centering
\begin{tabular}{|c|cc|}
\hline
$\delta$ & \Cref{alg:multi-stages} &  \textsc{Smart-ME} \\
\hline
0.001 & 0.06 & 0.05 \\
0.01 & 0.05 & 0.05 \\
0.1 & 0.23 & 0.44 \\
0.5 & 6.52 & 26.47 \\
1 & 23.63 & 50.74 \\
1.5 & 70.61 & 149.80 \\
\hline
\end{tabular}
\caption{Relative cost difference (in \%) versus \textsc{OFF}, for $\delta \in \{0.001, 0.01, 0.1, 0.5, 1, 1.5\}$, $\gamma = 2$, $T = 501$.}
\label{tab:synth-res-total-cost}
\end{table}

\Cref{fig:pareto-frontier} helps to explain this divergence as $\delta$ increases. When little weight is placed on target-following ($\delta \leq 0.1$), the hindsight optimum (represented by the green curve) greedily assigns arrivals to the resource with the minimum assignment cost. In the limit where $\delta = 0$, the dual variables associated with each resource and epoch converge to zero, yielding a greedy assignment. For $\delta=0.2$, the offline optimum's normalized per-period assignment cost steeply increases to over 0.2, with the mean absolute deviation dropping below 0.125. As $\delta$ increases further, the per-period assignment cost significantly increases in favor of minimizing target deviations, plateauing past $\delta = 1$. Note that, for $\gamma = 2$, perfect target-following is not always feasible, hence the average target deviation not converging to 0. 

\Cref{fig:pareto-frontier} shows the mean absolute deviation from hourly targets versus the assignment cost per period for the two policies, as well as \textsc{Off}, when $\gamma = 2$. It helps to explain why the myopic policy performs poorly for large deviation scalars, as it converges to an average target deviation of 0.06. As illustrated by \cref{prop:myopic-is-bad}, in attempting to hit the target in each epoch myopically, it converges to consumption paths that are globally suboptimal. The hindsight optimal solution, on the other hand, is cognizant of this difficulty and is able to achieve smaller average deviation as a result. \Cref{alg:multi-stages} lies between the two, consistently incurring higher assignment costs in order to achieve lower deviation costs. 

\begin{figure}
  \centering
  \includegraphics[width=0.5\linewidth]{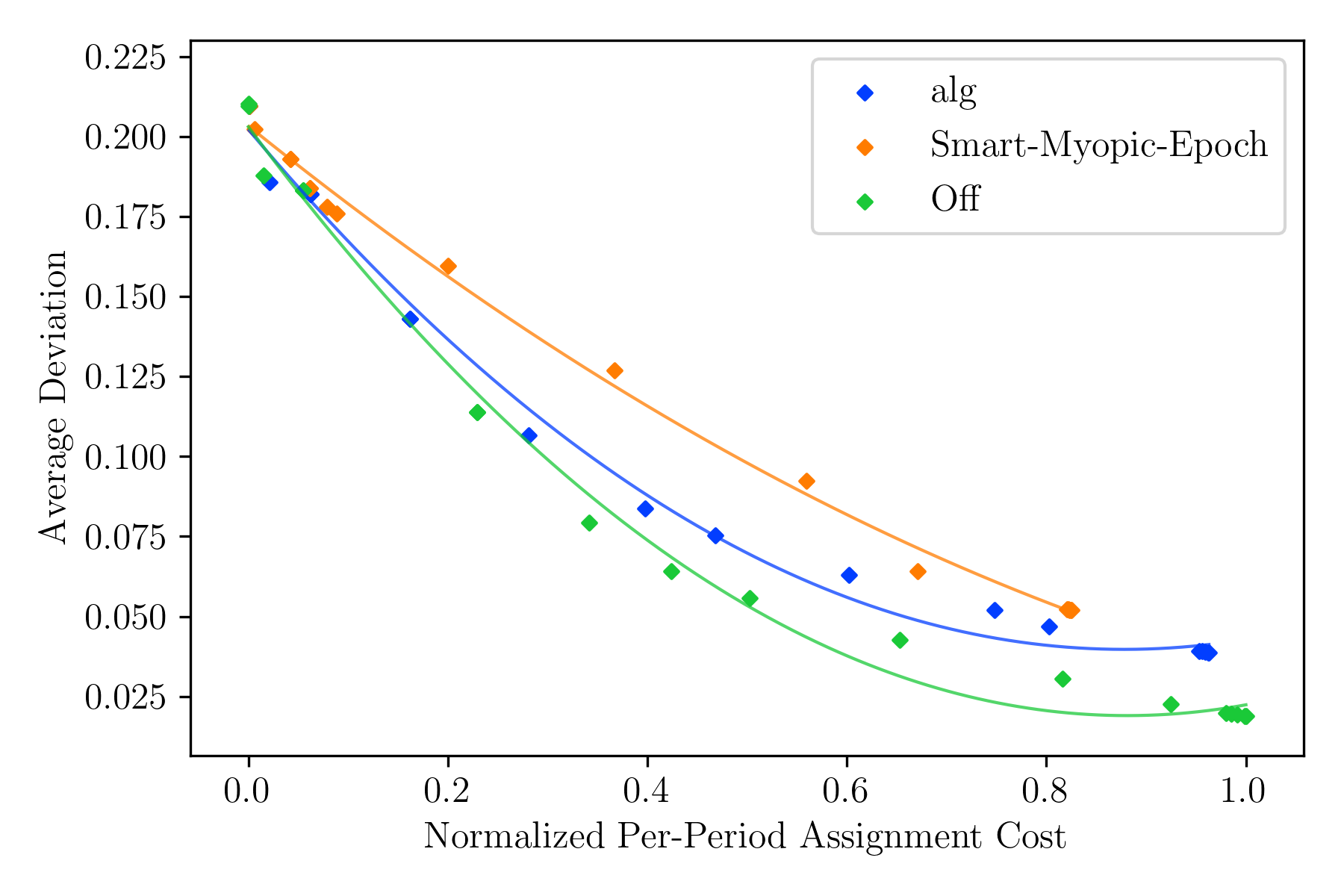}
  \caption{Mean absolute deviation from target versus assignment cost per period, for $\gamma = 2$. Dots from left to right correspond to deviation scalars $\delta = 0.001,0.01,0.1,0.2,\ldots,1.5$, respectively.
  }
  \label{fig:pareto-frontier}
\end{figure}

Finally, \cref{tab:gamma-dependence} shows the dependence of both policies' performance on the impulse level $\gamma$, for $\delta = 1$. For moderate values of $\gamma$ (i.e., $\gamma \in \{1,1.5,2\}$), \cref{tab:impulse-dep} shows that the costs incurred by \Cref{alg:multi-stages} relative to \textsc{Off} are insensitive to impulse level. Our policy's performance degrades to 1.4 times that of the offline optimum for the two extreme values of $\gamma \in \{0.5,2.5\}$, as the variability in targets makes it significantly more difficult to hit them simultaneously. The myopic policy, on the other hand, is significantly more sensitive to the impulse level for all values of $\gamma$. For $\gamma\in \{0.5,2.5\}$, its costs are over 85\% higher than that of the hindsight optimum; in between, the relative cost difference drops from 61\% to 50\%. It is particularly surprising that \textsc{Smart-ME} performs {\it worse} when $\gamma = 1$ (stable target), than  when $\gamma \in \{1.5,2\}$. \cref{tab:impulse-dep-target-dev} shows that, despite the fact that \textsc{Smart-ME} is {\it able} to hit all three targets when $\gamma = 1$, on average it deviates from the targets by 0.05. This is due to the fact that, for the shorter horizon that the myopic policy operates on ($T/K$ as opposed to $T$), target deviation {\it is} in fact optimal, in favor of low assignment costs, since it ignores the repercussions of deviating on future epochs. This behavior remains the same as $\gamma$ increases; \textsc{Off}, on the other hand, incurs higher deviation costs as $\gamma$ increases, hence the improved {relative} performance of the myopic policy.

Overall, these results underscore the importance of taking into account future targets when making the primal assignment decision. Indeed, these demonstrate that even careful modifications of state-of-the-art algorithms for the single epoch problem fail to outperform our algorithm due to their myopic nature.


\begin{table}
    \centering
    \begin{subtable}{0.45\textwidth}
        \centering
        \begin{tabular}{|c|cc|}
        \hline
        $\gamma$ & \Cref{alg:multi-stages} & \textsc{Smart-ME} \\
        \hline
        0.5 & 39.5 & 86.2 \\
        1 & 22.3 & 61.6 \\
        1.5 & 24.3 & 52.9 \\
        2 & 23.6 & 50.7 \\
        2.5 & 39.5 & 85 \\
        \hline
        \end{tabular}
        \caption{\centering{Relative cost difference (in \%) versus \textsc{Off}}}
        \label{tab:impulse-dep}
    \end{subtable}
    \quad
    \begin{subtable}{0.45\textwidth}
        \centering
        \begin{tabular}{|c|ccc|}
        \hline
        $\gamma$ & \textsc{Off} & \Cref{alg:multi-stages} & \textsc{Smart-ME} \\
        \hline
        0.5 & 0.001 & 0.033 & 0.068 \\
        1 & 0 & 0.023 & 0.055 \\
        1.5 & 0.002 & 0.029 & 0.054 \\
        2 & 0.023 & 0.047 & 0.064 \\
        2.5 & 0.044 & 0.067 & 0.081 \\
        \hline
        \end{tabular}
        \caption{\centering{Mean absolute deviation from target}}
        \label{tab:impulse-dep-target-dev}
    \end{subtable}
    \caption{Policy dependence on $\gamma \in \{0.5,1,1.5,2,2.5\}$, $\delta = 1$, $T = 501$.}\label{tab:gamma-dependence}
\end{table}

\subsection{Results on Real-World Warehouse Case-Breaking Data}\label{sec:real}

{We next apply our algorithm to the case-breaking problem described in the introduction, using real-world data from an online retailer. To manage the trade-off between the financial costs associated with case-breaking (i.e., assignment costs) and the deviation from hourly targets, the retailer currently uses a myopic dual price-based control, similar to \textsc{Smart-ME}. In this section, we illustrate the potential improvement from using a forward-looking policy based on our proxy assignment algorithm.}

\smallskip 

\textbf{Dataset description.} We obtained two datasets, each associated with a warehouse in the United States. (Data are shifted and scaled randomly for anonymity.) Each dataset comprises historical shipment case-breaking data over a period of 14 days. There are $m = 3$ resources capable of breaking a shipment, in addition to the ``no-break'' outside option. Each row in a dataset is defined by a 3-tuple $(t,i,q_i^*(t))$, where $t$ corresponds to the start of a two-minute interval on a given day (e.g., 2:34am on April 17, 2024), $i$ is one of the three resources or the outside option, and $q_i^*(t)$ is referred to as the {\it ideal quantity}. This quantity corresponds to the minimum cost assignment of units to resource $i$ during the two-minute interval starting at time $t$, absent target considerations. Formally, using our notation, $q_i^*(t)$ is defined as:
\begin{align}\label{eq:ideal-qty}
q^*_{i}(t) = \sum_{j\in[n]} \Lambda_j(t:t+2)\mathds{1}\left\{c_{ji} \in \arg\min_{i'\in[m] \cup \emptyset}c_{ji'}\right\},
\end{align}
with ties broken arbitrarily. These ideal quantities are tracked in order to compare assignment decisions to a ``first-best'' scenario in which resources are unconstrained. While the total costs under this unconstrained scenario are unachievable, they are an important benchmark to quantify the cost of limited capacity. We highlight that we only observe aggregate assignment quantities in our dataset, as opposed to the actual assignment costs for each arriving shipment and each resource.  Hence, we will have to infer the underlying assignment cost distribution from this censored data; we describe this process further below. 

Finally, for each hour in the day, our dataset includes the planned number of cases assigned to each resource, from the start of the day to the end of the hour. These can be interpreted as ``raw'' hourly targets. We let each hour of the day correspond to an epoch, and translate the raw targets to proportional targets, by dividing each of these by the cumulative number of shipments that arrived at the warehouse up until the end each hour. Formally, letting $P_{ki}$ denote the cumulative raw hourly target for resource $i$ by the end of hour $k$, we define the target $\rho_{ki} = P_{ki}/\Lambda(0:t_k)$, where we use $t=0$ to refer to the start of the day, and $t=t_k$ to denote the time interval corresponding to the end of hour $k$. 

We plot these post-processed targets for one cross-dock in \Cref{fig:hourly-targets}. 
\Cref{fig:hourly-targets} illustrates that (i) inter-day target variability is significant (\cref{fig:hourly-targets-horizon}), and (ii) there exist days for which targets vary significantly throughout the day (\cref{fig:hourly-targets-one-day}). For instance, in \Cref{fig:hourly-targets-one-day} the consumption target for Resource 1 exceeds 20\% of arrivals at the beginning of the day; by the middle of the day, however, this drops below 5\%. Moreover, changes in targets need not be smooth: target consumption for Resource 1 doubles hour-to-hour between hours 14 and 16 on this day. This suggests that myopic target tracking within a single day may result in volatile assignment behavior that fails to recognize the importance of future changing targets. 

\begin{figure}[h]
    \centering
    \begin{subfigure}[b]{0.47\textwidth}
        \centering
        \includegraphics[width=\textwidth]{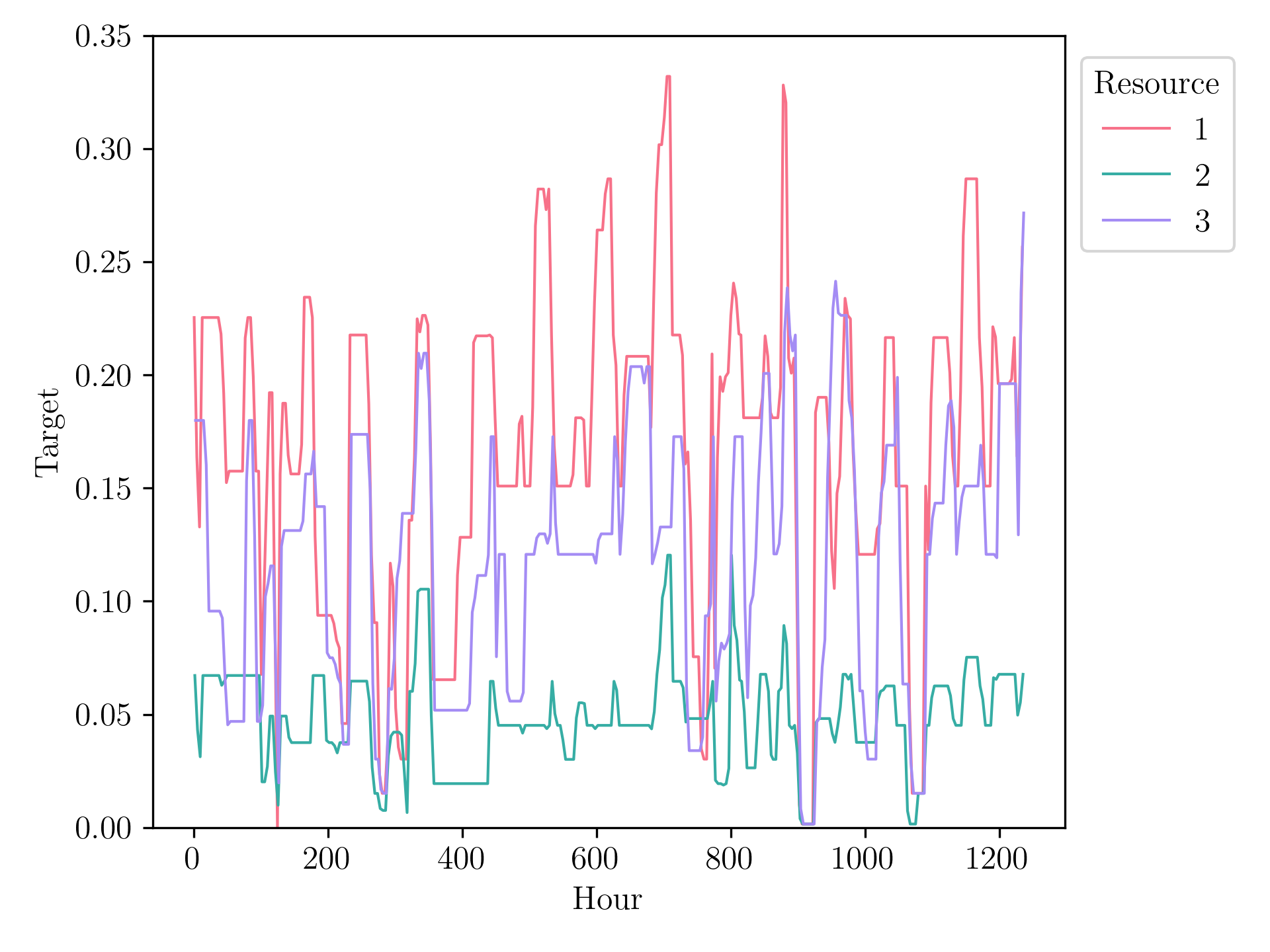}
        \caption{}\label{fig:hourly-targets-horizon}
    \end{subfigure}
    \hfill
    \begin{subfigure}[b]{0.47\textwidth}
        \centering
        \includegraphics[width=\textwidth]{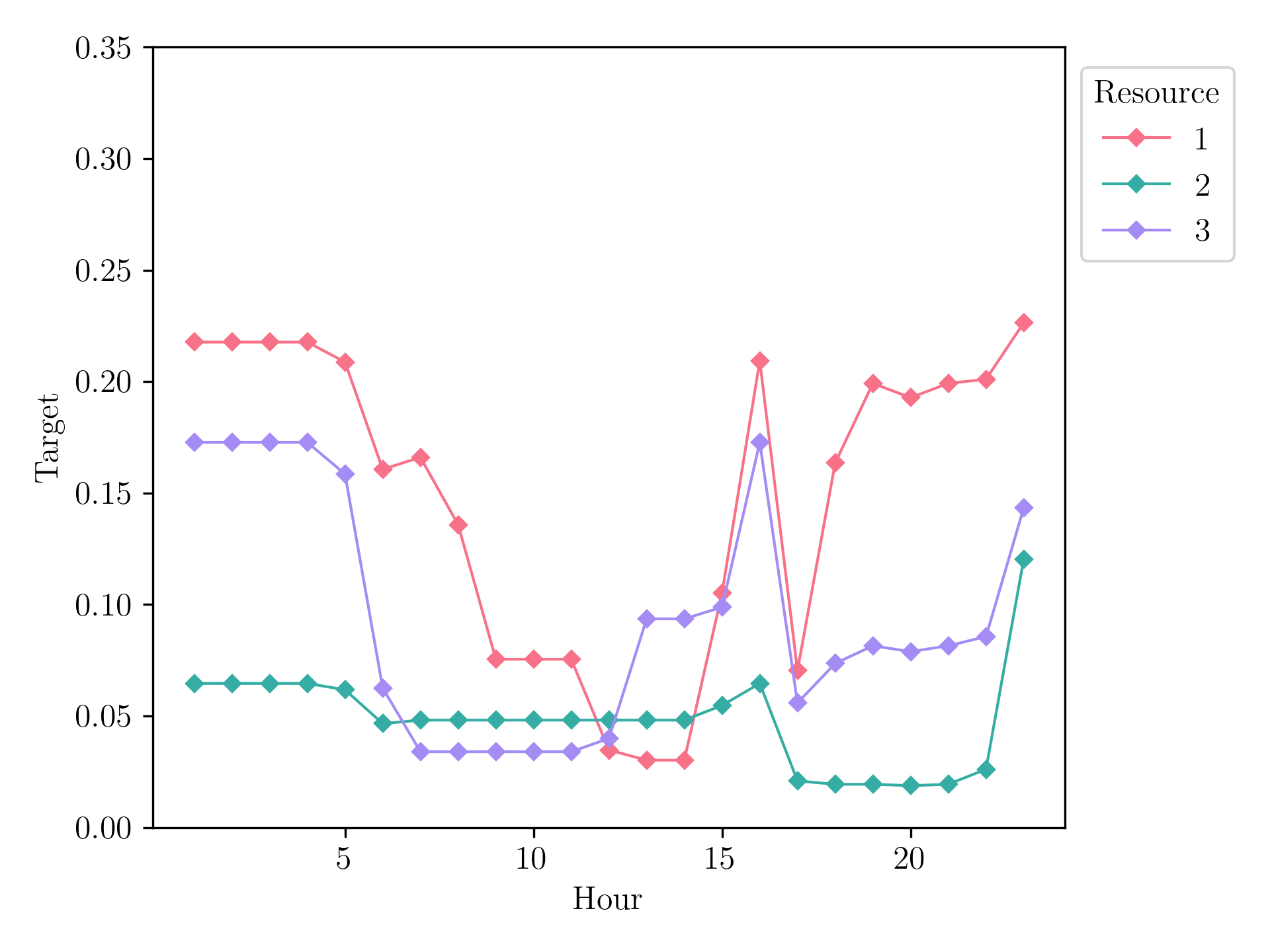}
        \caption{}\label{fig:hourly-targets-one-day}
    \end{subfigure}
    \caption{Hourly targets for Warehouse 1 throughout the two-week horizon (\cref{fig:hourly-targets-horizon}), and during a single day (\cref{fig:hourly-targets-one-day}).}\label{fig:hourly-targets}
\end{figure}


\smallskip

\textbf{Experimental setup.} For each dataset, we consider each of the 14 days to be a separate instance, with each hour of the day representing a different epoch. We let $T$ be the number of cases that arrived that day (rounded down so that $T$ is a multiple of the number of epochs $K$). On average, the facilities each received on the order of $10^4$ cases per day. As in \cref{sec:synth}, we let the deviation cost function $g_{ki}(a) = \delta|a-\rho_{ki}|$ for each epoch $k$ and resource $i$, with $\delta \in \{0.001,0.01,0.1,1\}$.

We estimate a separate assignment cost distribution for each facility. To do so, we assume that the number of cases arriving in a two-minute period has a Poisson distribution with parameter $\lambda$, where $\lambda$ is the average number of cases in each two-minute interval in the dataset. We moreover assume that there are three types of shipments, corresponding to the idea that each resource has a shipment type to which it is best-suited. 

{Recall, our dataset only contains {\it aggregate} assignment information for each two-minute interval. As a result, we do not observe historical assignment costs for arriving case. We do observe, however, the ideal quantity reported in each two-minute interval for each resource. We will use this information to estimate the type distribution via maximum likelihood estimation (MLE). Due to the non-differentiability of the indicator function in \eqref{eq:ideal-qty}, we instead approximate $q_i^*(t)$ via a soft-min, i.e.,:
\begin{align}\label{eq:soft-min}
q^*_{i}(t) = \sum_{j\in[n]} \Lambda_{j}(t:t+2) \cdot \frac{e^{-c_{ji}}}{1+\sum_{i' \in [m]}e^{-c_{ji'}}}.
\end{align}
This naturally lends itself to the assumption that each type $j$ case has assignment cost for resource $i$ drawn from a $\textup{Gumbel}(c_{ji},1)$ distribution, and assignment cost for the outside option drawn from a  $\textup{Gumbel}(0,1)$ distribution.\footnote{While assignment costs drawn from a Gumbel distribution violates the finite-support assumption required for our theoretical results, relaxing the cost distribution to have infinite support highlights the robustness of our algorithm to such practical approximations.
} It moreover allows us to write the likelihood function associated with the ideal quantities in closed form. Once we have this, we estimate the probabilities of each type, along with their respective location parameters (12 parameters total) via MLE. In both cases, the maximum likelihood estimation returned a single type, with location parameters displayed in \Cref{tab:mle-results}. We refer the reader to Appendix \ref{apx:mle} for details on the estimation procedure.
}

\begin{table}
    \centering
        \begin{tabular}{|c|ccc|}
        \hline
        & Resource 1 & Resource 2 &  Resource 3 \\
        \hline
        Warehouse 1&-0.33 & 1.27  & 0.21\\
        Warehouse 2&-0.74 & 1.12 & -0.001\\
        \hline
        \end{tabular}
        \caption{Maximum likelihood estimates of assignment cost parameters for each warehouse. The mean absolute deviation between the actual and ideal predicted fraction of cases under these estimates is 0.12 for Warehouse 1, and 0.1 for Warehouse 2.}\label{tab:mle-results}
\end{table}

For each warehouse, and each of the 14 days in the associated dataset, we randomly generated 25 sequences of $T$ arrivals according to the estimated assignment cost distribution, and compared the average performance of \Cref{alg:multi-stages} and \textsc{Smart-ME} across all 350 replications. We present results for one of the facilities below. Results for the second facility yielded similar insights; as such we relegate them to Appendix \ref{apx:clt2}.

\smallskip 

\textbf{Results.}  \cref{tab:real-res-total-cost} and \Cref{fig:real-pareto-frontier} respectively show the percentage gap between the two policies of interest and the offline optimum, for $\delta \in \{0.001,0.01,0.1,1\}$, as well as the trade-off between target deviation and assignment cost for each of these solutions. For $\delta \in \{0.001,0.01\}$, assignment costs overwhelm deviation cost considerations, with the optimal average target deviation close to 0.14; hence, it is easy for both policies to greedily assign arrivals to the minimum-cost resource and yield strong performance relative to the offline optimum. At $\delta = 0.1$, however, deviation costs are much more significant, with the optimal average target deviation below 0.04. \Cref{alg:multi-stages} is able to achieve close to this target deviation, at slightly lower assignment cost. The myopic policy's performance, however, deteriorates significantly at $\delta = 0.1$, incurring a 66\% higher cost than \textsc{Off}. \Cref{fig:real-pareto-frontier} shows that it is heavily biased toward low assignment costs, and maintains an average target deviation exceeding 0.1. At $\delta = 1$, the average target deviation lies below 0.01; \Cref{alg:multi-stages} is able to remain close to this, at just above 0.01, with near identical assignment costs. Since target deviations are 10 times as costly as when $\delta = 0.1$, we observe a decrease in performance, from 1.21\% to 9.13\% relative to the offline optimum. For this largest value of $\delta$, the myopic policy incurs over 5 times the cost of the offline optimum. This is due to the fact that its mean absolute deviation from the hourly target is around 0.05, despite incurring less than half the assignment costs per period.

\begin{table}
\centering
\begin{tabular}{|c|cc|}
\hline
$\delta$ & \Cref{alg:multi-stages} &  \textsc{Smart-ME} \\
\hline
0.001 & 0.021  & 0.016 \\
0.01 & 0.075 &  0.652 \\
0.1 & 1.21 & 66.31 \\
1 & 9.13 & 481.85 \\
\hline
\end{tabular}
\caption{Relative cost difference (in \%) versus \textsc{OFF}, for $\delta \in \{0.001, 0.01, 0.1, 1\}$}
\label{tab:real-res-total-cost}
\end{table}

\begin{figure}
  \centering
  \includegraphics[width=0.5\linewidth]{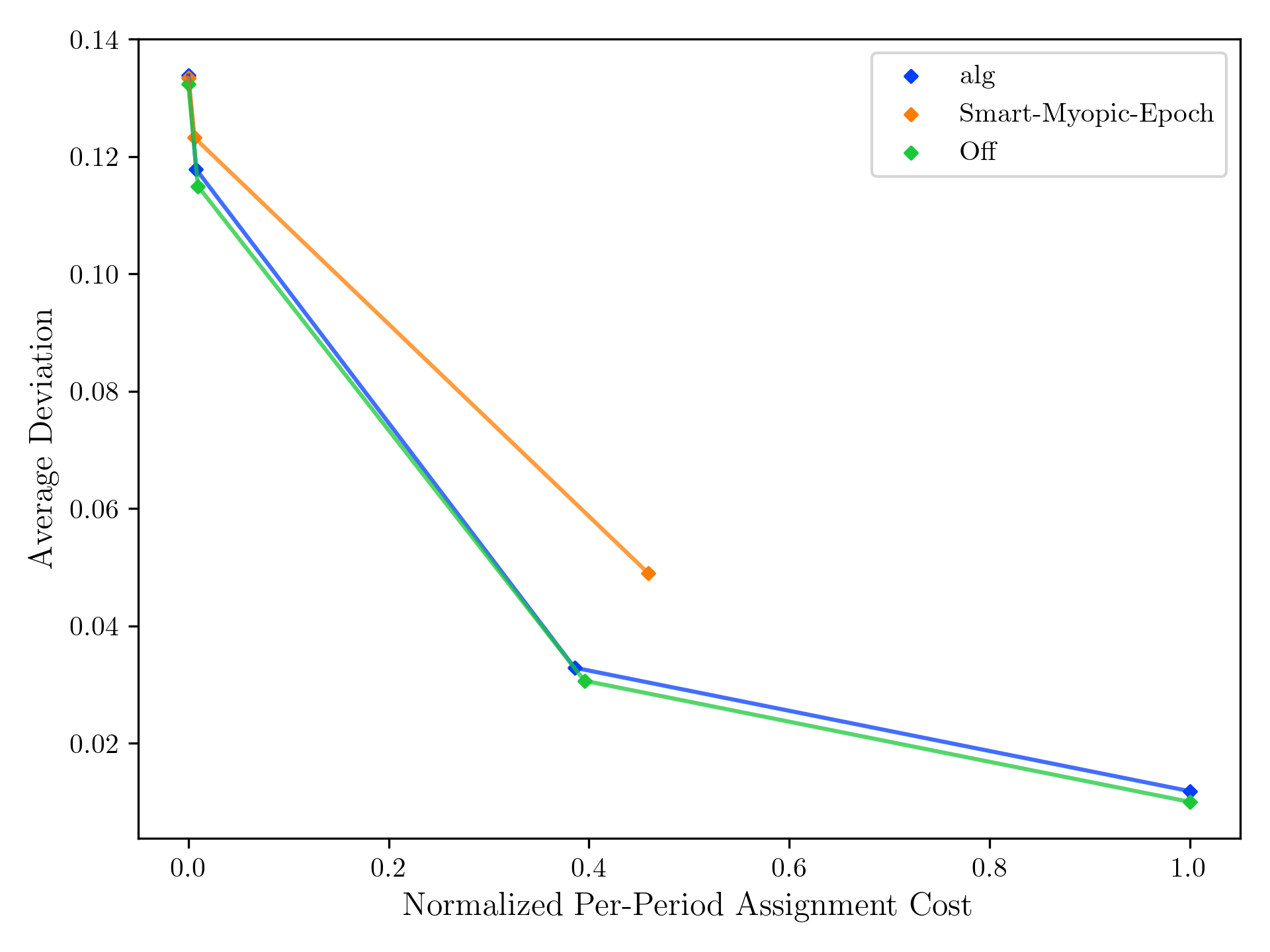}
  \caption{Mean absolute deviation from target versus normalized assignment cost per period for \textsc{OFF}, \Cref{alg:multi-stages} and \textsc{Smart-ME}. Dots from left to right correspond to deviation scalars $\delta = 0.001,0.01,0.1,1$, respectively.
  }
  \label{fig:real-pareto-frontier}
\end{figure}

We conclude by showing our policy's strong performance across {\it all} instances in the dataset for $\mbox{$\delta \geq 0.01$}$, as shown in \Cref{fig:daily-cost-diff}. (As discussed above, our policy slightly underperforms the myopic policy for $\delta = 0.001$, though both achieve near-optimal performance relative to the offline optimum.) Moreover, the myopic policy's performance is highly variable across instances, as it becomes more important to track the hourly targets. This is most obvious for $\delta = 1$; in the best case the myopic policy incurs less than twice the cost of the offline optimum, whereas on one of the days it incurred over 36 times the hindsight optimal cost. \Cref{fig:real-bad-instance} shows the hourly targets throughout this bad instance, as well as the running average consumption under \textsc{Smart-ME}. Despite the magnitude of the deviation scalar, the myopic policy deviates from the hourly targets by over 20 percentage points in the worst case, for Resource 1, and 10 percentage points in the worst case for Resource 3. This is due to the fact that Resources 1 and 3 have significantly lower average assignment costs than Resource 2; hence, for a single epoch it is still optimal to deviate from the target, despite the large deviation scalar. Our policy, on the other hand, takes into account future targets (and the fact that deviations in current epochs will accumulate penalties for the rest of the horizon), and as a result tracks the targets much more closely across all resources.

\begin{figure}[h]
    \centering
    \begin{subfigure}[b]{0.47\textwidth}
        \centering
        \includegraphics[width=\textwidth]{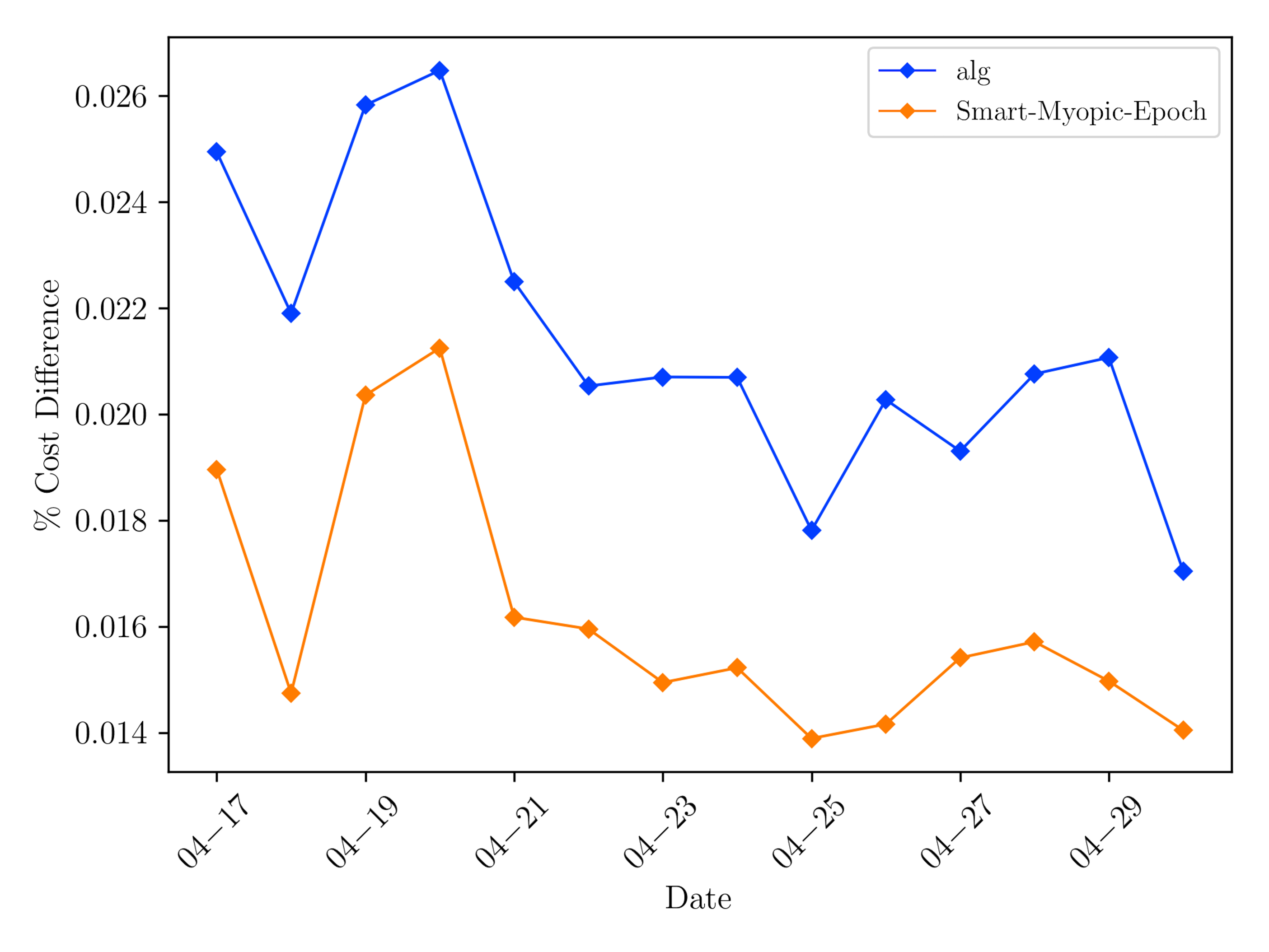}
        \caption{$\delta = 0.001$}
    \end{subfigure}
    \hfill
    \begin{subfigure}[b]{0.47\textwidth}
        \centering
        \includegraphics[width=\textwidth]{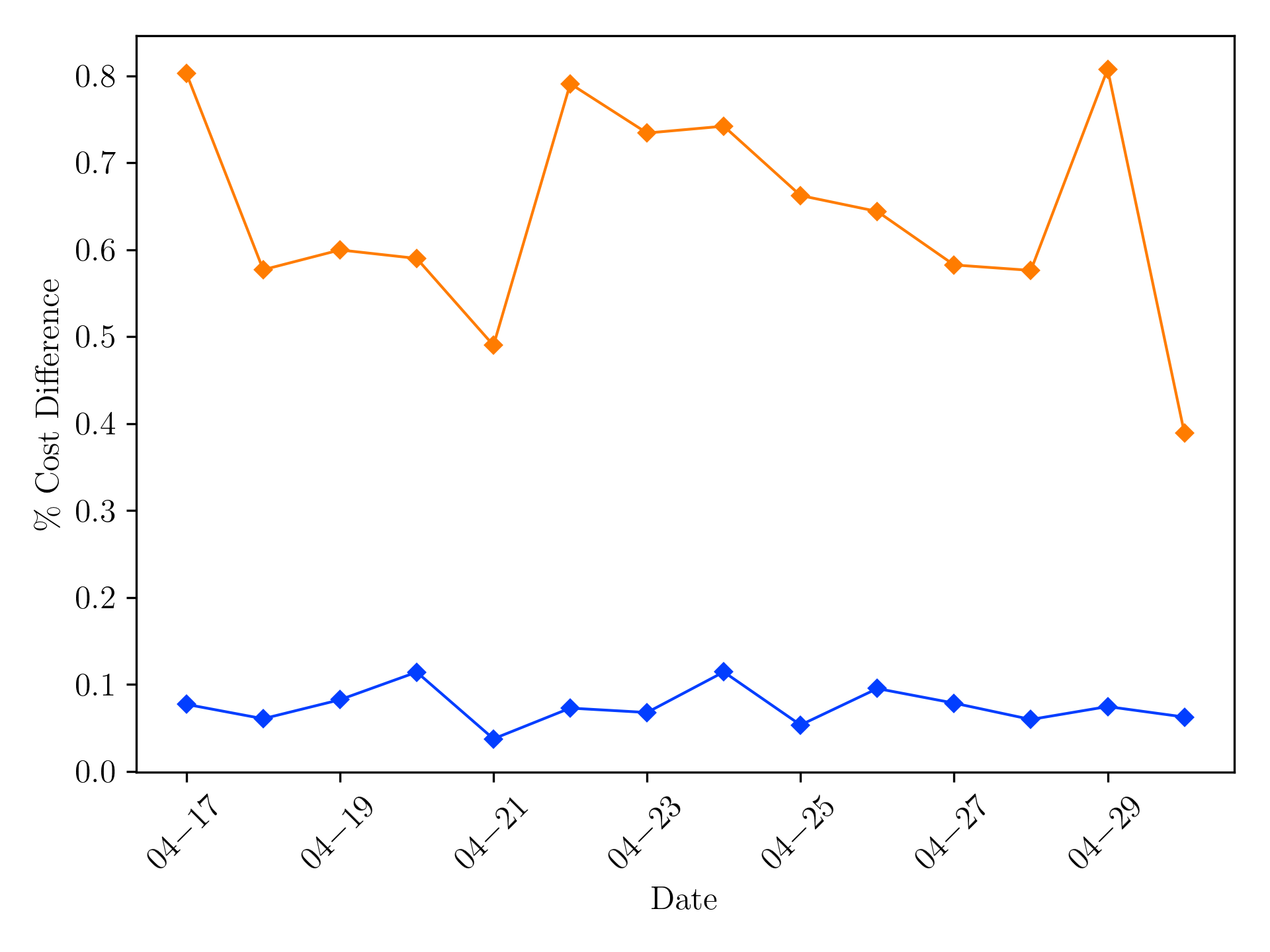}
        \caption{$\delta = 0.01$}
    \end{subfigure}
    \vfill
    \begin{subfigure}[b]{0.47\textwidth}
        \centering
        \includegraphics[width=\textwidth]{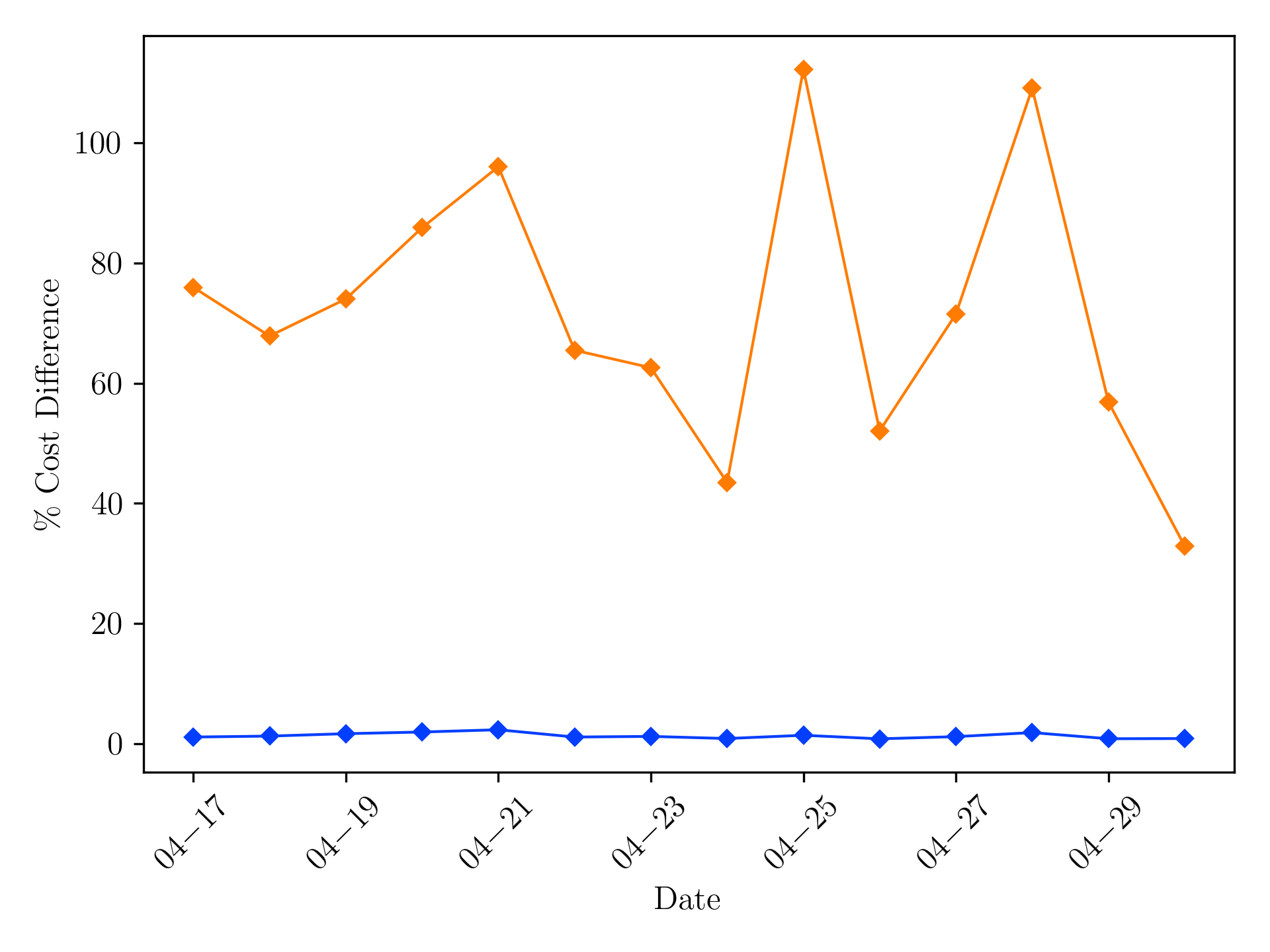}
        \caption{$\delta = 0.1$}
    \end{subfigure}
    \hfill
    \begin{subfigure}[b]{0.47\textwidth}
        \centering
        \includegraphics[width=\textwidth]{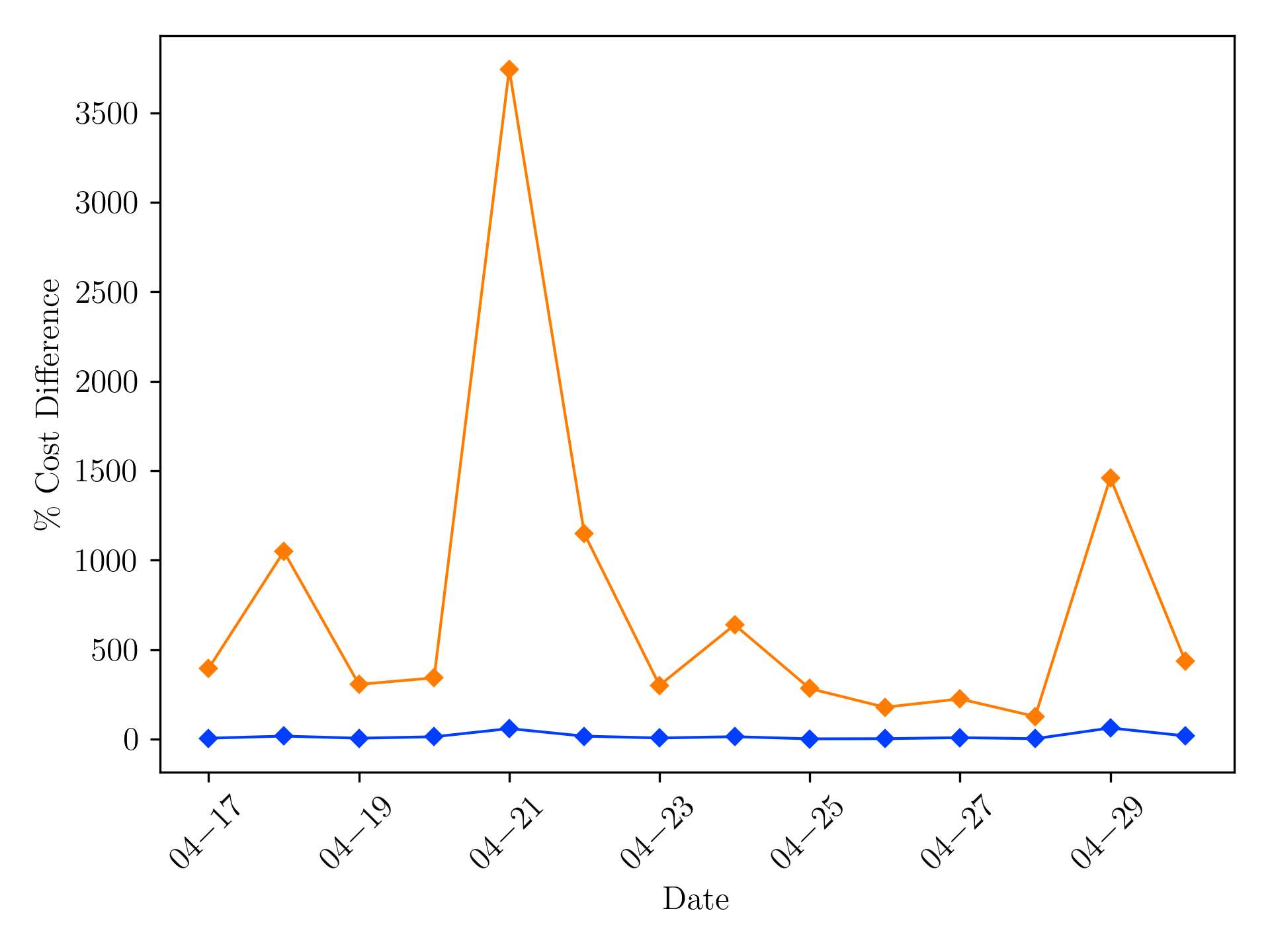}
        \caption{$\delta = 1$}
    \end{subfigure}
    \caption{Daily relative cost difference (in \%) for \Cref{alg:multi-stages} and \textsc{Smart-ME} versus \textsc{OFF}, for $\delta \in \{0.001,0.01,0.1,1\}$. The blue and orange curves respectively represent the performance of \Cref{alg:multi-stages} and \textsc{Smart-ME} for each instance.}
    \label{fig:daily-cost-diff}
\end{figure}

\begin{figure}
  \centering
  \includegraphics[width=0.5\linewidth]{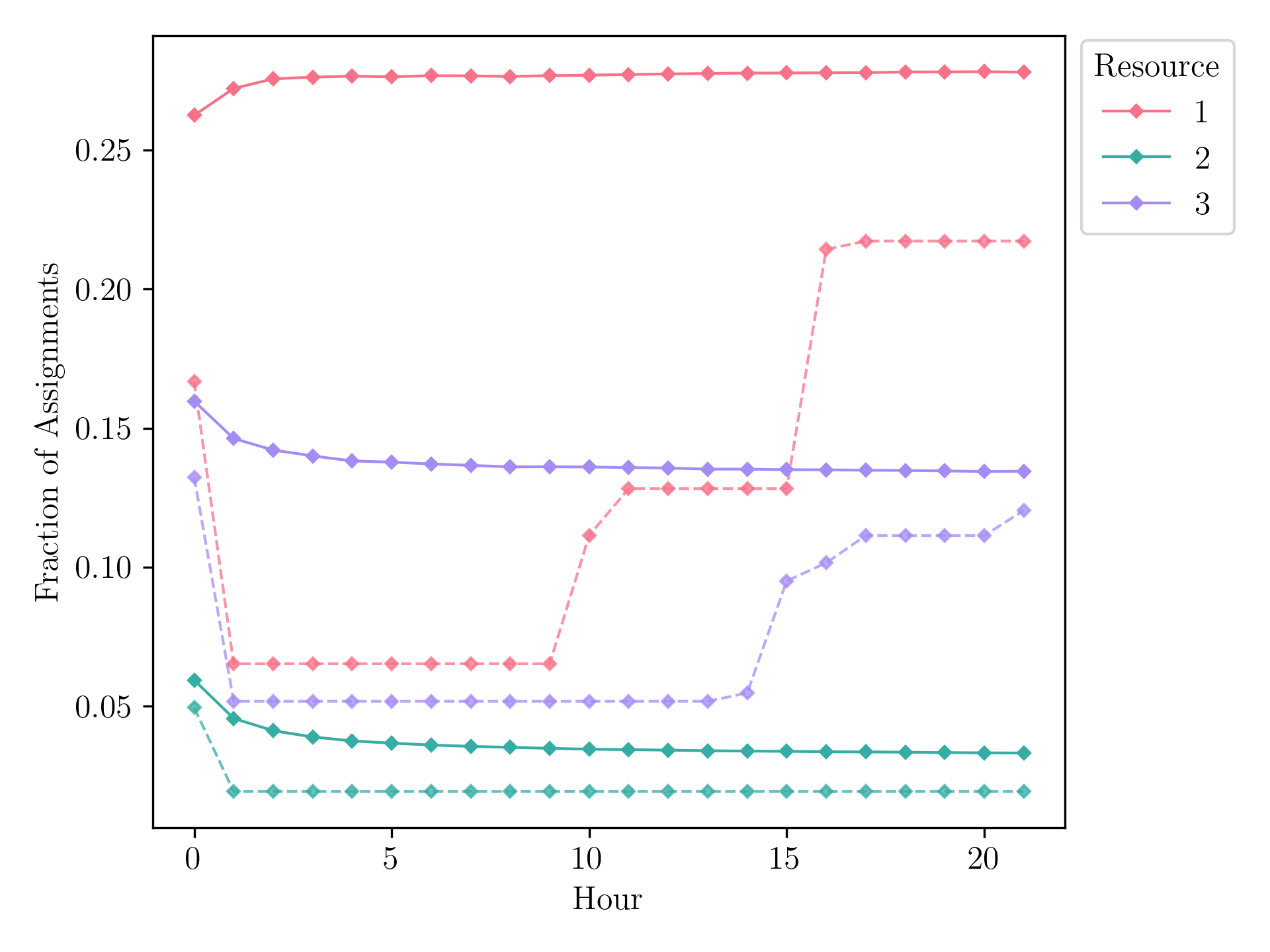}
  \caption{Running average assignment across resources under \textsc{Smart-ME} throughout a single day in the dataset (the bad instance), for $\delta = 1$. Dotted lines correspond to hourly targets for each resource.
  }
  \label{fig:real-bad-instance}
\end{figure}

%% file: parts/conclusion.tex
\section{Conclusion}\label{sec:conclusion}

We introduced and studied the problem of {throughput-constrained} online resource allocation, which is the natural {throughput} capacity analogue of classical online resource allocation problems \citep{balseiro2021regularized}. Our model applies to a wide range of {contexts (e.g., manual or automated processing such as picking, packing, sorting, loading, assembling and conveyance/transportation), where the trajectory of consumption is first-order}.
We developed a novel sublinear regret online algorithm that overcomes the challenges of multi-epoch target-following using proxy assignment decisions and demonstrated its practical performance via extensive experiments on both synthetic and real-world datasets.

Our work opens up several future directions. We conjecture that the dependence of the regret of our algorithm on the number of epochs $K$ is suboptimal; it would be interesting to derive matching upper and lower bounds with respect to this model primitive. Additionally, {while our algorithm retains its strong guarantees under nonstationary aggregate arrival rates, these guarantees rely on the underlying {\it type} distribution being i.i.d. throughout the horizon.} 
Understanding whether proxy assignment decisions remain useful when arrivals types are nonstationary (e.g., for ergodic processes) or adversarially determined would be of both theoretical and practical interest. Indeed, to the best of our knowledge the idea of proxy assignment decisions in online decision-making is novel and has potential applications in other settings. For example, the transportation arcs of a network typically have different transfer times between them due to difference in geographical distance. This means the time at which a decision-maker assigns an arrival to a resource and when the arrival actually arrives at the resource varies, creating a challenge for online allocation similar to that in our problem. Based on our work, a similar proxy allocation approach that uses current arrivals to predict the decisions made for future assignments was used to build a transfer time aware online load balancing controller at Amazon (\cite{rudden2024}). Lastly, our work considers settings where the decision-maker is necessarily penalized for target deviation. While this is true in many practical settings, a natural question to ask is what happens when the high-level plans that generate the targets change during the assignment process. It may then be more appropriate to treat the targets as ``advice'' provided to the decision-maker, who must then design algorithms that perform well when the advice is good, but are robust to errors in this advice when it is poor.


%% file: parts/single_epoch_alg.tex
\section{Dual Mirror Descent for Single-Epoch Target Following}\label{apx:balseiro-et-al}

We present the dual mirror descent algorithm developed in \citet{balseiro2021regularized} for the single-epoch problem below. For ease of presentation, we assume the reference function used for the mirror descent step is the $\ell_2$-norm, yielding online gradient descent in the dual space. 


\begin{algorithm}[!ht]
\DontPrintSemicolon 
\KwIn{Stepsize $\eta$, initial shadow prices $\mu^1\in\mathbb{R}^{m}$}
\KwOut{Sequence of assignment decisions}
\For{$t = 1, \ldots, T$}{
    Observe arrival type $j = j^t$, and make primal assignment decision:
    \begin{equation}\label{eq:balseiro-asg}
    \begin{aligned}
        x^t \in \arg\min_{x\in \mathcal{X}^t} \quad &\sum_{i \in [m]}x_i(c_{ji}-\mu_i^t)
    \end{aligned}
    \end{equation}\;

    Compute idealized average consumption:
    \begin{equation}\label{eq:balseiro-dev}
    \begin{aligned}
        a^t \in \arg\min_{a \in [0,1]^{m}} \quad &\sum_{i\in[m]}g_{1i}\left(a_i\right) +\sum_{i\in[m]}\mu^t_{i}a_{i} 
    \end{aligned}
    \end{equation}\;
    Update shadow prices via OGD step:
    \begin{align}\label{eq:balseiro-ogd}
    \mu^{t+1} = \mu^t + \eta(a^{t}-x^{t}).
    \end{align}
        }
	\caption{Dual Gradient Descent for Single-Epoch Target Following}
	\label{alg:single-epoch}
\end{algorithm}

%% file: parts/apx_nonstationary.tex
{
\section{Extension to Nonstationary Arrival Rates}\label{apx:nonstationary}

Motivated by the reality that arrival rates may vary throughout the decision-making horizon in a number of applications, in this section we relax the assumption that the number of arrivals is equal across all epochs. Specifically, we assume that there are $l_kT$ arrivals in each epoch $k \in [K]$, for some known vector \mbox{$l \in (0,1)^K$} such that $\sum_{k\in[K]}l_k = 1$. For simplicity, we assume $l_kT$ is integral for all $k$. Observe that this setting subsumes the basic setting presented in \Cref{sec:model}, in which $l_k = 1/K$ for all $k \in [K]$. This generalization is also able to model ``peak'' and ``non-peak'' periods throughout the horizon (e.g., $K = 3$, with $l_1 = 1/4$, $l_2 = 1/2$, $l_3 = 1/4$). 

We consider two possible models for the cumulative targets: (i) {\it per-arrival} targets, which pace the average assignment over all arrivals by the end of each epoch, and (ii) {\it per-period} targets, which pace the average assignment per unit time. In the example above, for instance, per-arrival targets would average the number of assignments over the $T/4$ arrivals in the first epoch; on the other hand, per-period targets would average the number of assignments over $T/3$ periods, assuming for simplicity that all epochs have equal ``time-length.''

In the following two sections, we demonstrate a reduction of each of these settings to our basic setup. For ease of notation, we let $l_{\leq k} = \sum_{k'\leq k}l_{k'}$. Moreover, for any policy $\pi$, we let $Z_{ji}^\pi(t)$ denote the number of type $j$ arrivals assigned to resource $i$ after the $t$-th arrival; as before, we let $Z_i^\pi(t) = \sum_{j\in[n]}Z_{ji}^\pi(t)$.

\subsection{Per-Arrival Targets}\label{apx:per-arrival}

At a high level, we reduce any nonstationary instance to a problem over a {stationary} instance by partitioning the larger epochs in the original instance in such a way that the final constructed instance is composed of epochs of equal length. Defining the deviation cost functions associated with these new ``sub-epochs'' to be identically zero, we observe that these two instances are in fact equivalent. Therefore, applying our algorithm to the constructed stationary instance immediately provides a sublinear regret guarantee for the original instance. We formalize these ideas below.

In this case, the total cost incurred by any policy $\pi$ over a sample path of arrivals $\omega$ is given by:
\begin{align}\label{eq:nonstat-cost-per-arrival}
V^\pi[\omega] = \sum_{j\in[n]}\sum_{i\in[m]}c_{ji}Z_{ji}^\pi(T) + \sum_{k\in[K]}\sum_{i\in[m]}l_{\leq k}Tg_{ki}\left(\frac{Z_i^\pi(l_{\leq k}T)}{l_{\leq k}T}\right).
\end{align}

Let $\mathcal{I}$ denote the instance defined by the $K$ epochs with unequal arrival rates and deviation cost functions \mbox{$(g_{ki}, k \in [K], i \in [m])$}. We construct an equivalent stationary instance, denoted by $\mathcal{I}_{new}$, as follows. Let $l_{new}T$ be the greatest common divisor of $(l_1T,l_2T,\ldots,l_KT)$. For simplicity we assume $l_{new}T = \Theta(T)$, and that $1/l_{new}$ is integral. Then, define $\mathcal{I}_{new}$ to be an instance with $K_{new} = 1/l_{new}$ epochs, and $l_{new}T$ arrivals per epoch. 

\smallskip 

\begin{example}\label{ex:inst-1}
Consider any instance $\mathcal{I}$ with $K = 3$ and $l_1 = 1/4$, $l_2 = 1/2$, $l_3 = 1/4$. Then, $K_{new} = 4$, with \mbox{$l_{new} = 1/4$}.
\end{example}

\smallskip

Let \mbox{$(\widetilde{g}_{ki}, k \in [K_{new}], i \in [m])$} denote the set of deviation cost functions for this new instance. To define this set of functions, we let $f: [K] \mapsto [K_{new}]$ be the mapping from the set of epochs in the original instance to the set of epochs in the new instance, such that, for all $k \in [K], k_{new} \in [K_{new}]$, $f(k) = k_{new}$ if and only if $l_{\leq k} T = k_{new}l_{new}T$. Equivalently, we have $f(k) = \frac{l_{\leq k}}{l_{new}}$ for all $k \in [K]$. In words, $k_{new}$ is such that the number of arrivals by the end of the $k$-th epoch in the original instance is equal to the number of arrivals by the end of the $k_{new}$-th instance.  Let $\mathcal{K}_{new} = \left\{f(k), k \in [K]\right\}$. Then, for all $i \in [m]$, we define:
\[\widetilde{g}_{ki}(a) = \begin{cases}
g_{f^{-1}(k),i}(a) \quad &\forall \ k \in \mathcal{K}_{new}\\
0 &\text{otherwise}.
\end{cases}\]
In words, $\mathcal{I}_{new}$ is such that there is no target consumption rate for the arrivals that do not correspond to the end of an epoch in the original instance; for all other epochs, the target is the same.

\smallskip 

\begin{example}[Continued from \Cref{ex:inst-1}]\label{ex:inst-2}
In this case, $f(1) = 1$, $f(2) = 3$ and $f(3) = 4$, with \mbox{$\mathcal{K}_{new} = \{1,3,4\}$}. Moreover, $\widetilde{g}_{1,i}(a) = g_{1,i}(a)$, $\widetilde{g}_{2,i}(a) = 0$, $\widetilde{g}_{3,i}(a) = g_{2,i}(a)$, and $\widetilde{g}_{4,i}(a) = g_{3,i}(a)$.
\end{example}

\smallskip

The following proposition establishes that the problem under $\mathcal{I}$ reduces to that under $\mathcal{I}_{new}$. We defer the proof of \Cref{prop:reduc-1} to Appendix \ref{apx:reduc-1}.

\begin{proposition}\label{prop:reduc-1}
Fix a sample path of arrivals $\omega$ and a sequence of assignment decisions, denoted by $Z^\pi$. Let $V^\pi[\omega]$ and $V^\pi_{new}[\omega]$ respectively denote the total cost incurred by $Z^\pi$ in $\mathcal{I}$ and $\mathcal{I}_{new}$. Then, $V^\pi[\omega] = V^{\pi}_{new}[\omega]$.
\end{proposition}

Consider now the algorithm which, given nonstationary instance $\mathcal{I}$, constructs stationary instance $\mathcal{I}_{new}$, and runs \Cref{alg:multi-stages} on $\mathcal{I}_{new}$, implementing the same assignment decisions on the true instance, for all $t \in [T]$. We refer to this algorithm as \nonstationaryalg. We have the following result, which is an immediate corollary of \Cref{prop:reduc-1}.

\begin{corollary}\label{cor:reduc-1}
For $\eta = \Theta\left(\sqrt{K_{new}/T}\right)$, the gap between the expected cost incurred by \nonstationaryalg\ and the offline benchmark is upper bounded by
$\mathbb{E}[\textsc{Reg}] = O(K_{new}^{5/2}\sqrt{T}).$
\end{corollary}

\subsubsection{Proof of \Cref{prop:reduc-1}}\label{apx:reduc-1}
\begin{proof}{\it Proof.}
The total cost incurred by $Z^\pi$ in the constructed instance is given by:
\begin{align*}
V_{new}^\pi[\omega] &= \sum_{j\in[n]}\sum_{i\in[m]}c_{ji}Z_{ji}^\pi(T) + {\sum_{k \in [K_{new}]}\sum_{i\in[m]}k l_{new} T \widetilde{g}_{ki}\left(\frac{Z_i^\pi(kl_{new}T)}{kl_{new}T}\right)}\\
&= \sum_{j\in[n]}\sum_{i\in[m]}c_{ji}Z_{ji}^\pi(T) + \underbrace{\sum_{k\in\mathcal{K}_{new}}\sum_{i\in[m]} kl_{new}T\widetilde{g}_{ki}\left(\frac{Z_i^\pi(kl_{new}T)}{kl_{new}T}\right)}_{(I)},
\end{align*}
where the second equality follows from the fact that  $\widetilde{g}_{ki}(\cdot) = 0$ for all $k \not\in\mathcal{K}_{new}$. Recall, for all $k \in \mathcal{K}_{new}$, by construction we have:
\begin{align*}
\begin{cases}
l_{\leq f^{-1}(k)}T = kl_{new}T\\
\widetilde{g}_{ki}(\cdot) = g_{f^{-1}(k),i}(\cdot).
\end{cases}
\end{align*}
Hence, we can re-write $(I)$ in the space of epochs of the original instance, i.e.,
\begin{align*}
(I) = \sum_{k\in[K]}\sum_{i\in[m]}l_{\leq k} Tg_{ki}\left(\frac{Z_i^\pi(l_{\leq k}T)}{l_{\leq k}T}\right).
\end{align*}
Plugging this back into $V_{new}^\pi[\omega]$ and comparing this to \eqref{eq:nonstat-cost-per-arrival}, we obtain the result.
\hfill\Halmos
\end{proof}

\subsection{Per-Period Targets}

Consider now the setting in which targets (and consequently deviation costs) are specified per-period; for simplicity we assume that each epoch has ``time-length'' $T/K$ in this setting. In this case, we trivially reduce to the nonstationary setting with {\it per-arrival} targets by appropriately re-scaling the deviation cost functions. In particular, with per-period targets, the total cost incurred by any policy $\pi$ over a sample path of arrivals $\omega$ is given by:
\begin{align}\label{eq:nonstat-cost-per-period}
V^\pi[\omega] = \sum_{j\in[n]}\sum_{i\in[m]}c_{ji}Z_{ji}^\pi(T) + \sum_{k\in[K]}\sum_{i\in[m]}\frac{kT}{K}g_{ki}\left(\frac{Z_i^\pi(l_{\leq k}T)}{kT/K}\right).
\end{align}

\Cref{prop:reduc-2} below establishes that we can construct an equivalent instance in which targets are per-arrival, as in Appendix \ref{apx:per-arrival}. Namely, for $k \in [K], i\in[m]$, define deviation cost function \mbox{$\widetilde{g}_{ki}(a) = \frac{1}{l_{\leq k}}\cdot\frac{k}{K}g_{ki}\left(\frac{l_{\leq k}}{k/K}a\right)$}. Consider an instance $\mathcal{I}_{new}$ defined by the set of epochs $[K]$ and deviation cost functions $(\widetilde{g}_{ki}, k \in [K], i\in [m])$.  

\begin{proposition}\label{prop:reduc-2}
Fix a sample path of arrivals $\omega$ and a sequence of assignment decisions, denoted by $Z^\pi$. Then,
\[V^\pi[\omega] =  \sum_{j\in[n]}\sum_{i\in[m]}c_{ji}Z_{ji}^\pi(T) + \sum_{k\in[K]}\sum_{i\in[m]}l_{\leq k}T\widetilde{g}_{ki}\left(\frac{Z_i^\pi(l_{\leq k}T)}{l_{\leq k}T}\right).\]
\end{proposition}

We omit the straightforward proof of this fact. The following corollary emerges as a direct result of \Cref{prop:reduc-2}.
\begin{corollary}
Consider the algorithm that takes in the original instance defined by per-period targets, converts it into a ``per-arrival'' instance, and runs \nonstationaryalg\ on the per-arrival instance. For $\eta = \Theta(\sqrt{K_{new}/T})$, this algorithm incurs $\mathbb{E}[\textsc{Reg}] = O(K_{new}^{5/2}\sqrt{T})$, where $K_{new}$ is defined as in Appendix \ref{apx:per-arrival}. 
\end{corollary}
}

%% file: parts/apx_exper.tex
\section{Computational Experiments: Additional Details}\label{apx:exper}

\subsection{Benchmark Implementations}

\Cref{alg:benchmark} describes \textsc{Myopic-Epoch} and \textsc{Smart-Myopic-Epoch}, for a given epoch $k$. Given epoch $k$, these two algorithms assume the decision-making horizon is $\mathcal{T}_k = \{(k-1)T/K+1,\ldots,kT/K\}$, and take as input the epoch-specific targets $\tilde{\rho}_k$, respectively given in \eqref{eq:myopic} and \eqref{eq:smart-myopic}. They run \Cref{alg:single-epoch} over this shorter horizon, disregarding all future targets. 

\begin{algorithm}[!ht]
\DontPrintSemicolon 
\KwIn{Current epoch $k$, stepsize $\eta$, initial shadow prices $\mu^1\in\mathbb{R}^{m}$, target consumptions $\tilde{\rho}_k \in [0,1]^m$}
\KwOut{Sequence of assignment decisions}
\For{$t = (k-1)T/K+1, \ldots, kT/K$}{
    Observe arrival $j = j^t$, and make assignment decision:
    \begin{equation}
    \begin{aligned}
       x^t \in \arg\min_{x\in \mathcal{X}^t} \quad &\sum_{i \in [m]}x_i(c_{ji}-\dual{i}{t})
    \end{aligned}
    \end{equation}\;

    Compute idealized average consumption:
    \begin{equation}\label{eq:many-stages-aux-problem-bm}
    \begin{aligned}
        a^t \in \arg\min_{a \in [0,1]^{m}} \quad &\delta\sum_{i\in[m]}|a_i-\tilde{\rho}_{ki}| + \sum_{i\in[m]}\mu^t_{i}a_{i} 
    \end{aligned}
    \end{equation}\;
    Update shadow prices via OGD step:
    \begin{align}
    \mu^{t+1} = \mu^{t} + \eta(a^{t}-x^{t}).
    \end{align}
        }
	\caption{Myopic Benchmark Algorithms Implemented in \cref{sec:exper}}
	\label{alg:benchmark}
\end{algorithm}

\subsection{Description of Maximum Likelihood Estimation Procedure for Case-Breaking Dataset}\label{apx:mle}

Recall, we assume that the number of arriving shipments in each two-minute interval is drawn from a Poisson distribution with parameter $\lambda$, where $\lambda$ is the average number of shipments observed in the dataset for each two-minute interval. Then, under the soft-min approximation \eqref{eq:soft-min}, the ideal quantity $q_i^*(t)$ is drawn independently from a Poisson distribution with parameter $\lambda_i$, defined as: $$\lambda_i = \lambda\sum_{j\in[n]}p_j\frac{e^{-c_{ji}}}{1+\sum_{i'\in[m]}e^{-c_{ji'}}}.$$ 
We abuse notation and use $T$ to denote the number of two-minute intervals across the 14 days represented in the dataset. We moreover let $q^* = (q_i^*(t), \ \forall \ i \in [m], \forall \ t \in [T])$ denote the vector of all ideal quantities observed in the data, across all two-minute intervals and resources. The associated likelihood function is given by:
\begin{align}
&\mathcal{L}(q^*;p, c) = \prod_{t=1}^T\prod_{i=1}^m e^{-\lambda_i}\cdot \frac{\lambda_i^{q_i^*(t)}}{q_{i}^*(t)!} \notag \\
\implies &\log \mathcal{L}(q^*;p, c) = -\sum_{i=1}^m\lambda_iT + \sum_{t=1}^T \sum_{i=1}^m q_i^*(t)\log\lambda_i-\sum_{t=1}^T\sum_{i=1}^m\log \left(q_i^*(t)!\right), \label{eq:neg-log-like}
\end{align}
where we hide the explicit dependence of $\lambda_i$ on $(p,c)$ for ease of notation.  

We find the maximum likelihood estimates $(\hat{p}, \hat{c})$ by running projected gradient descent on the negative log-likelihood \eqref{eq:neg-log-like}, constraining $c_{ji} \in [-50,50]$ for all $j \in [n], i \in [m]$. We let the procedure run for $n_{\text{iter}} = 10^4$ iterations, with a stepsize $\eta = 0.2$. {Due to the non-convexity of the objective in $(p,c)$, we run the procedure for 1,000 random initial solutions, and choose the solution associated with the highest log-likelihood.} 

\subsection{Additional Results for Case-Breaking Dataset}\label{apx:clt2}

In this section we include results for the second warehouse in the dataset, for the same experimental setup as in \cref{sec:real}. 

As for the first warehouse, \cref{tab:real-res-total-cost-w2} demonstrates our policy's strong performance relative to the myopic target-following algorithm. The insights are similar: for negligible values of $\delta$, both policies are very close to the offline optimum, with performance deteriorating as $\delta$ increases. 

Compared to Warehouse 1 data, for $\delta \leq 0.1$, both policies perform better relative to the offline optimum. For $\delta = 1$, on the other hand, their performance is significantly worse (8 percentage point decrease in performance for our policy, and an over 5-fold decrease for the myopic policy). \Cref{fig:real-pareto-frontier-w2} helps to explain this, as compared to \Cref{fig:real-pareto-frontier}. For $\delta = 0.1$, the offline optimum is more biased toward assignment cost-minimizing solutions, incurring slightly higher average target deviation per period. Given that target-following is the significant challenge for any online algorithm, both policies are indeed expected to perform better in this case. For $\delta = 1$, however, we are back in a case where the optimal solution remains very close to the hourly target, on average. This is more challenging for both policies to achieve, as the assignment costs for Resources 1 and 3 are significantly lower, on average ($(c_{1}, c_3) = (-0.74, -0.001)$ for Warehouse 2, as opposed to $(c_1, c_3) = (-0.33, 0.21)$ for Warehouse 1).

\begin{table}
\centering
\begin{tabular}{|c|cc|}
\hline
$\delta$ & \Cref{alg:multi-stages} &  \textsc{Smart-ME} \\
\hline
0.001 & 0.009  & 0.007 \\
0.01 & 0.034 &  0.291 \\
0.1 & 0.841 & 49.989 \\
1 & 17.687 & 2727.329 \\
\hline
\end{tabular}
\caption{Relative cost difference (in \%) versus \textsc{OFF}, for $\delta \in \{0.001, 0.01, 0.1, 1\}$}
\label{tab:real-res-total-cost-w2}
\end{table}

\begin{figure}
  \centering
  \includegraphics[width=0.5\linewidth]{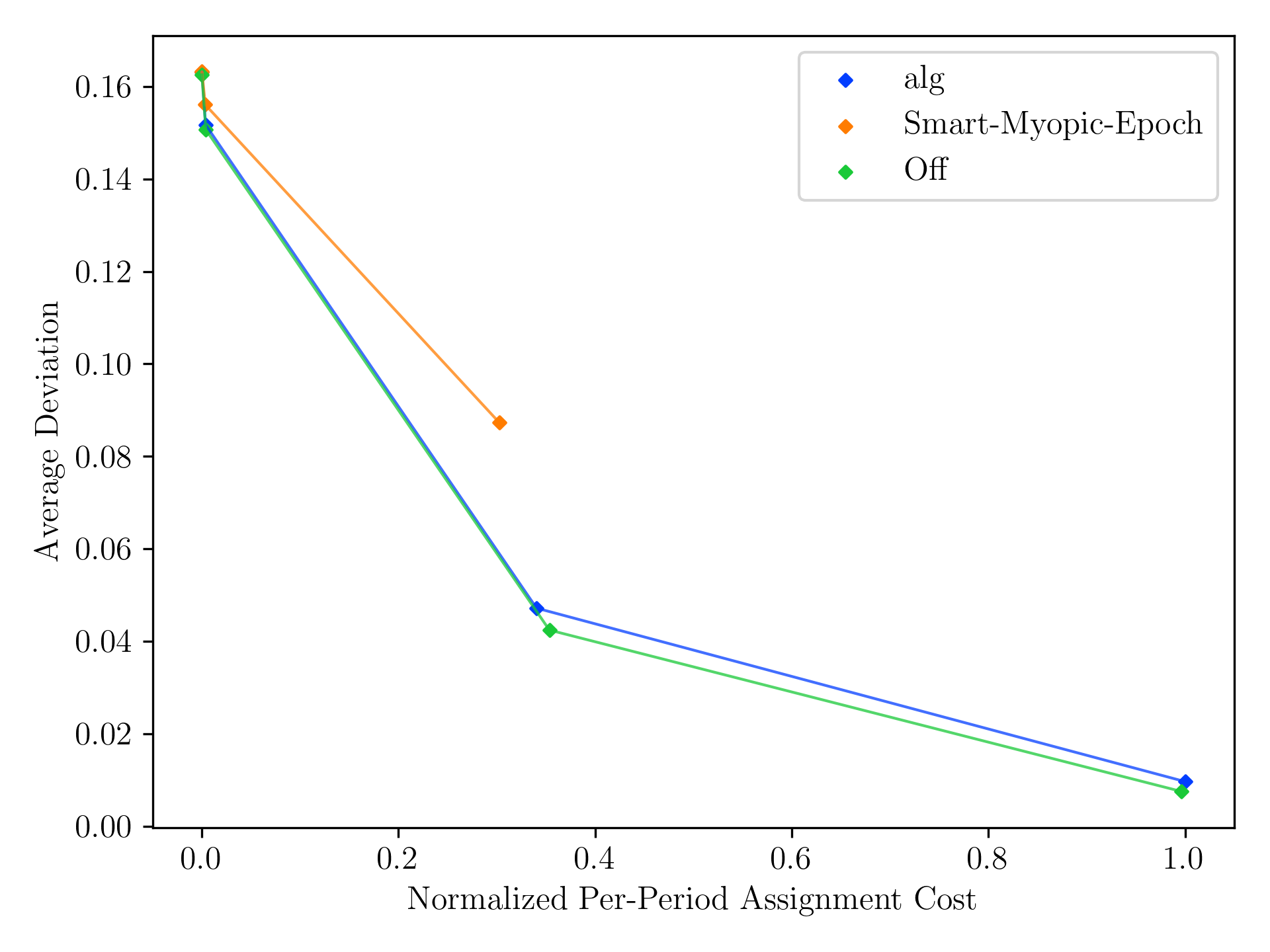}
  \caption{Mean absolute deviation from target versus normalized assignment cost per period. Dots from left to right correspond to deviation scalars $\delta = 0.001,0.01,0.1,1$, respectively.
  }
  \label{fig:real-pareto-frontier-w2}
\end{figure}

Finally, \Cref{fig:daily-cost-diff-w2} illustrates the strong performance of our policy across {\it all} instances for the majority of values of $\delta$, as for Warehouse 1. The variability in the myopic policy's performance is even more evident for this warehouse when $\delta = 1$, with the myopic policy incurring over 350 times the cost of the offline optimum on April 29.  \Cref{fig:real-bad-instance-w2}  shows the targets and running average assignments under the myopic policy throughout that day. Similar to the bad instance for Warehouse 1, in this instance the respective targets of Resources 1 and 3 are flipped, with the cheaper resource (Resource 1) having a lower target than the more expensive resource (Resource 3) for the majority of the horizon. In this case, the myopic policy fails to see the consequences of respecting this swap in future resources, and overshoots assignments to Resource 1 by over 15 percentage points for the majority of the horizon.

\begin{figure}[h]
    \centering
    \begin{subfigure}[b]{0.47\textwidth}
        \centering
        \includegraphics[width=\textwidth]{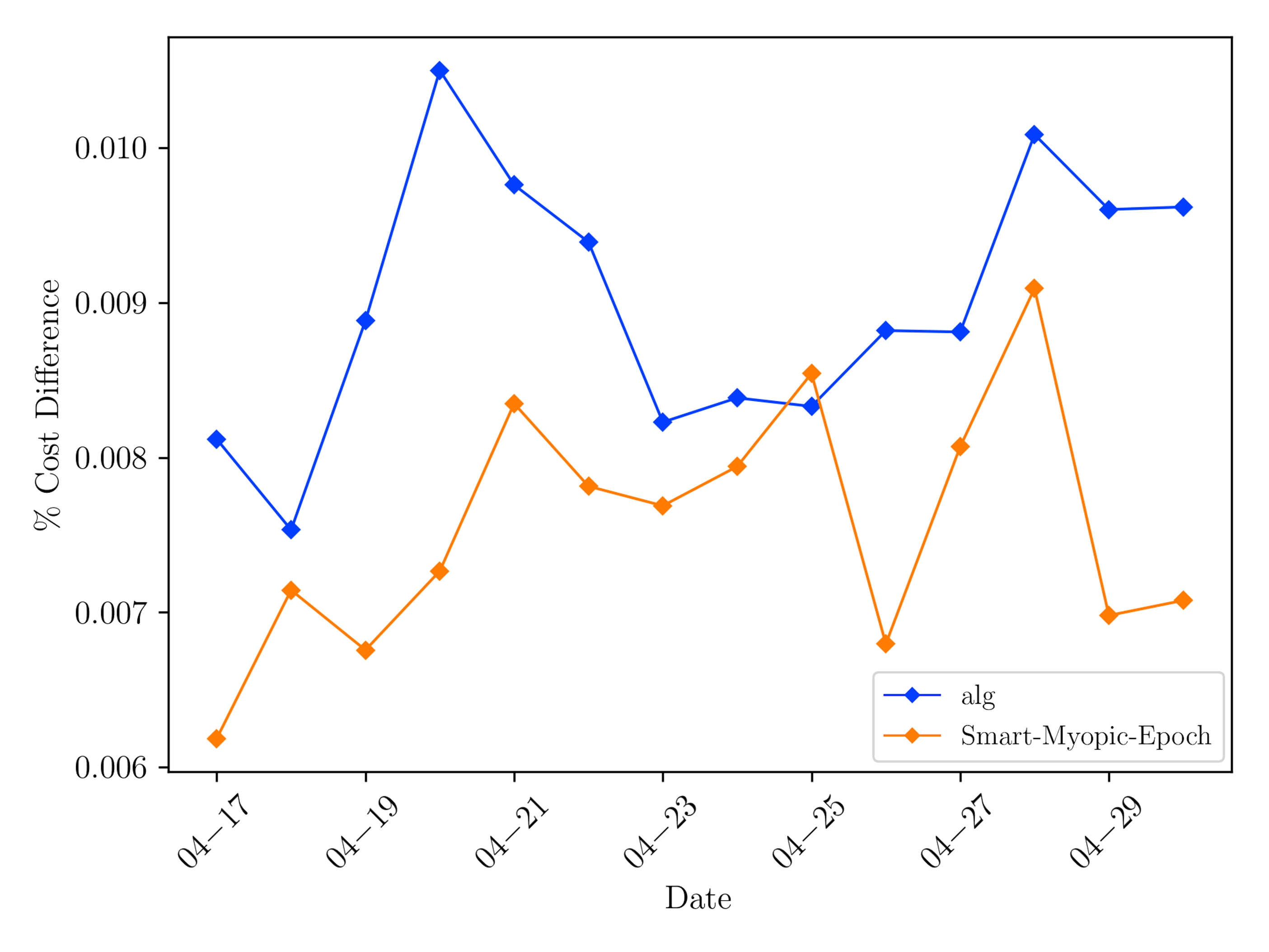}
        \caption{$\delta = 0.001$}
    \end{subfigure}
    \hfill
    \begin{subfigure}[b]{0.47\textwidth}
        \centering
        \includegraphics[width=\textwidth]{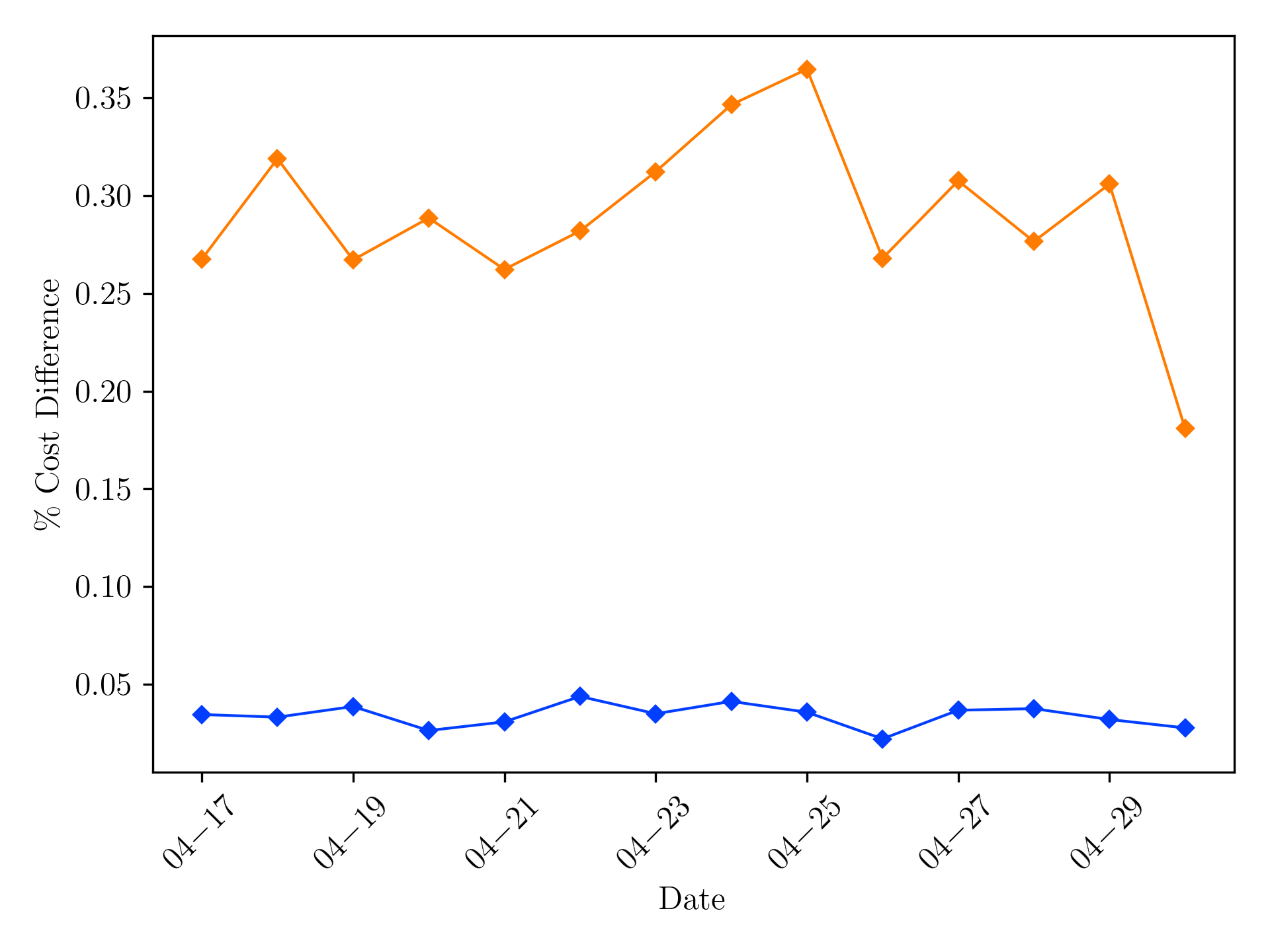}
        \caption{$\delta = 0.01$}
    \end{subfigure}
    \vfill
    \begin{subfigure}[b]{0.47\textwidth}
        \centering
        \includegraphics[width=\textwidth]{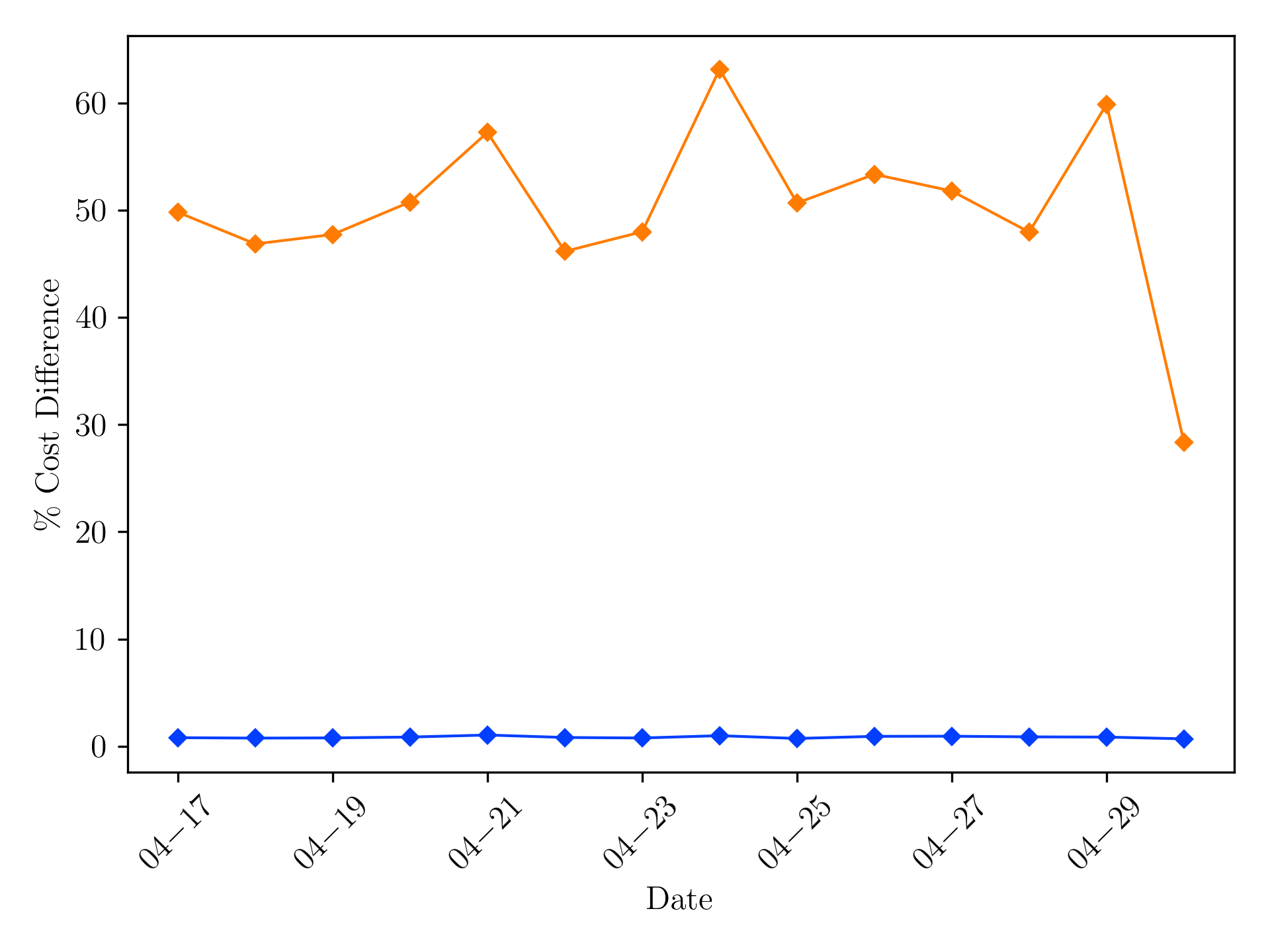}
        \caption{$\delta = 0.1$}
    \end{subfigure}
    \hfill
    \begin{subfigure}[b]{0.47\textwidth}
        \centering
        \includegraphics[width=\textwidth]{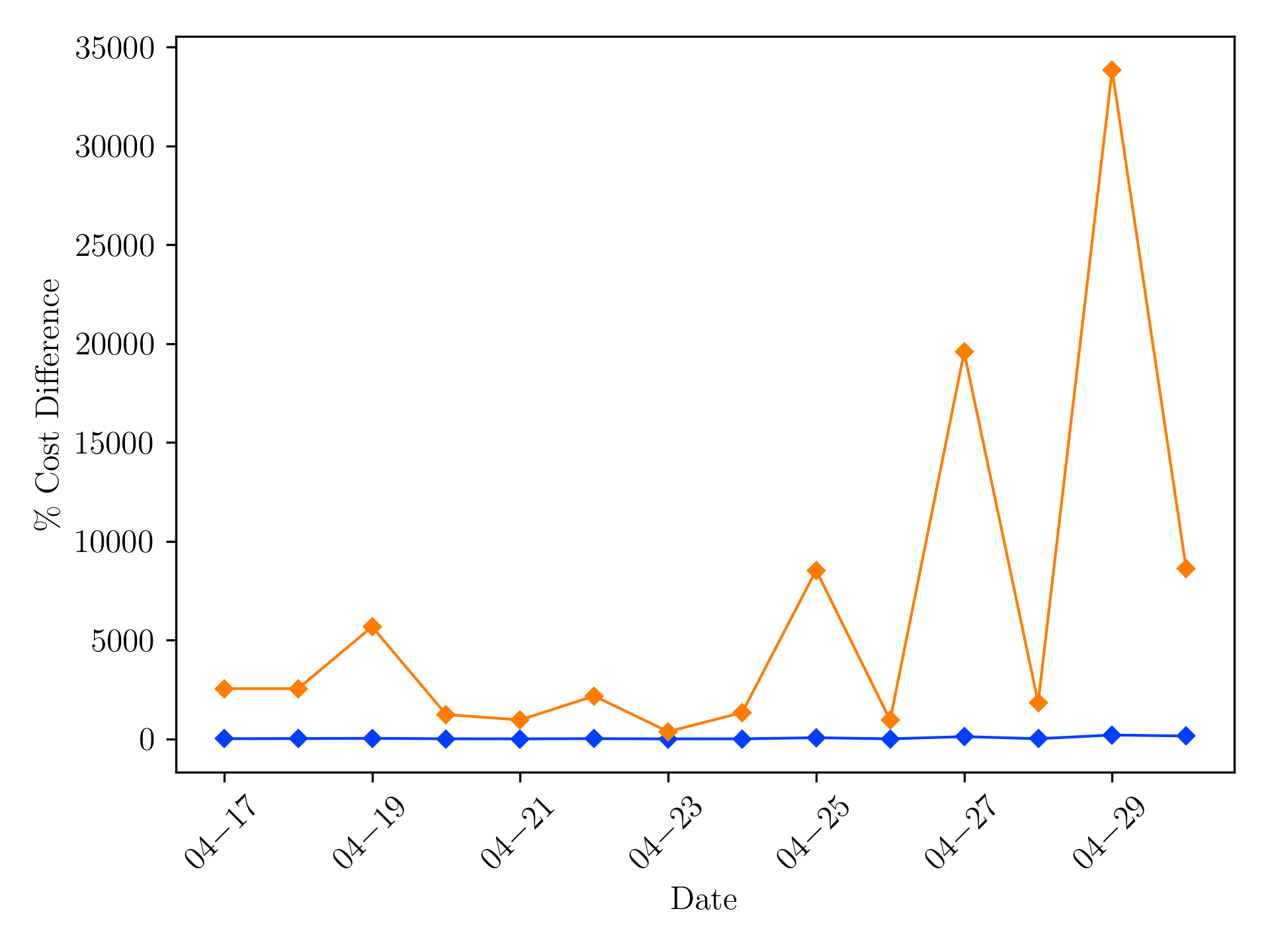}
        \caption{$\delta = 1$}
    \end{subfigure}
    \caption{Daily relative cost difference (in \%) versus \textsc{OFF}, for $\delta \in \{0.001,0.01,0.1,1\}$. The blue and orange curves respectively represent the performance of \Cref{alg:multi-stages} and \textsc{Smart-ME} for each instance.}
    \label{fig:daily-cost-diff-w2}
\end{figure}

\begin{figure}
  \centering
  \includegraphics[width=0.5\linewidth]{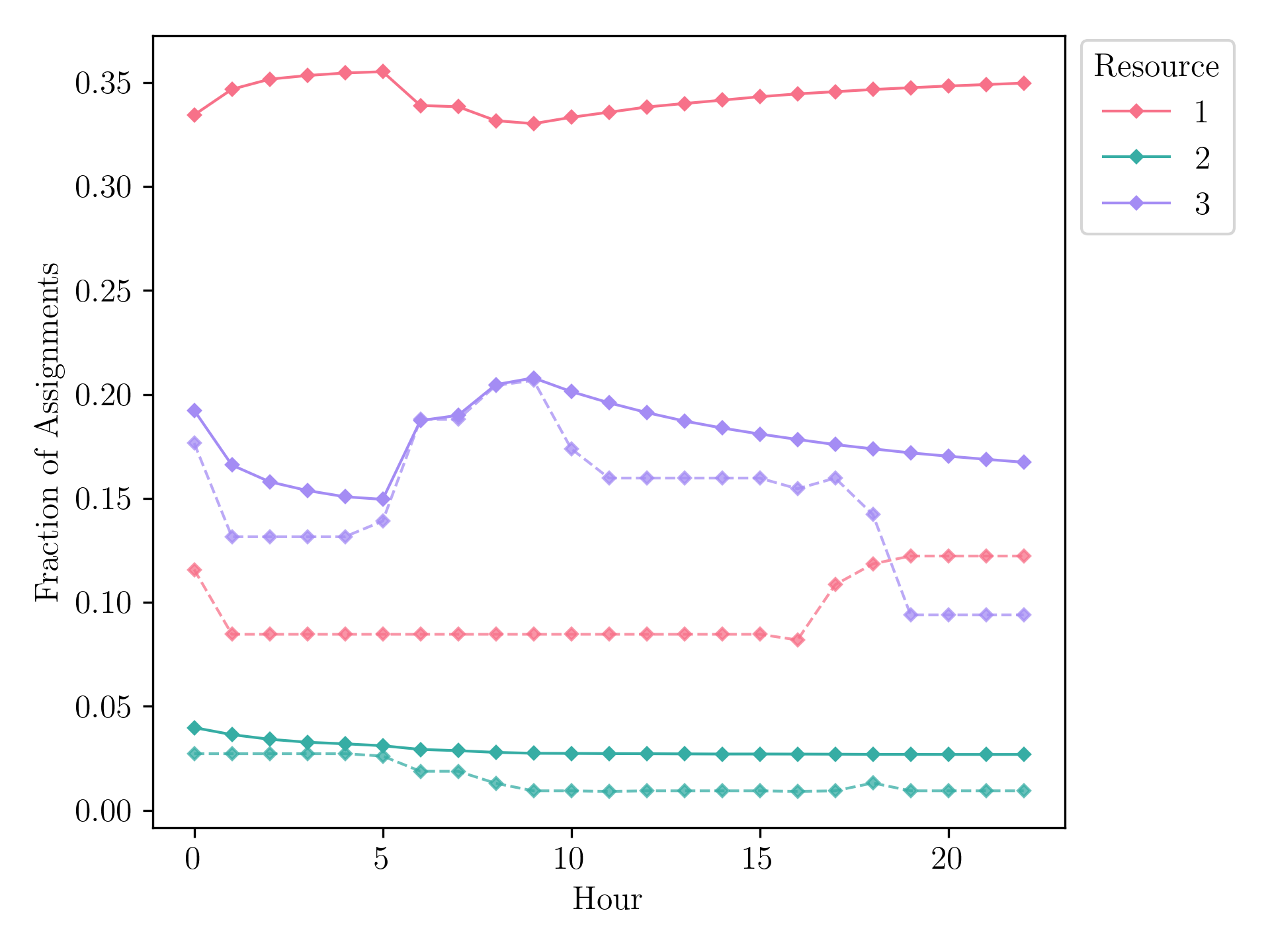}
  \caption{Running average assignment across resources under \textsc{Smart-ME} throughout a single day in the dataset (the bad instance), for $\delta = 1$. Dotted lines correspond to hourly targets for each resource.
  }
  \label{fig:real-bad-instance-w2}
\end{figure}